\documentclass[aoas]{imsart}
\usepackage{bibunits}
\usepackage{amsmath}
\usepackage{amsfonts}
\usepackage{amsthm}
\usepackage[colorlinks=true, citecolor=blue, linkcolor=blue, urlcolor=blue]{hyperref}
\usepackage{bbm}
\usepackage{empheq}
\usepackage{scrextend}
\usepackage{mathtools}
\usepackage{amssymb}
\usepackage{graphicx}
\usepackage[authoryear,sort&compress]{natbib}
\usepackage{subcaption}

\graphicspath{ {./images/} }
\usepackage{xcolor}
\usepackage{dsfont}
\usepackage{graphicx}
\usepackage{cases}
\graphicspath{ {./images/} }
\usepackage{float}
\usepackage{algorithm} 
\usepackage{algpseudocode}
\usepackage{accents}
\newlength{\dhatheight}

\newcommand{\norm}[1]{\left\lVert#1\right\rVert}

\usepackage{booktabs}

\newtheorem{Lemma}{Lemma}[section]

\newtheorem{definition}[Lemma]{Definition}
\newtheorem{theorem}[Lemma]{Theorem}
\newtheorem{prop}[Lemma]{Proposition}

\newtheorem{remark}[Lemma]{Remark}
\newtheorem{assumption}[Lemma]{Assumption}
\newtheorem{example}[Lemma]{Example}
\usepackage[utf8]{inputenc}

\theoremstyle{definition}

\date{}

\definecolor{darkblue}{rgb}{.1, 0.1,.8}
\definecolor{darkgreen}{rgb}{0,0.8,0.2}
\definecolor{darkred}{rgb}{.8, .1,.1}

\usepackage{mathtools}
\mathtoolsset{showonlyrefs}
\newcommand*{\E}{\mathbb{E}}
\newcommand*{\Var}{\mathbb{V}ar}

\newcommand*{\N}{\mathbb{N}}
\newcommand*{\R}{\mathbb{R}}
\renewcommand{\P }{{\mathbb P}}

\newcommand{\ve}{\varepsilon}
\newcommand{\indvec}{(\hat{\mathbbm{1}}_1 (X) , \ldots, \hat{ \mathbbm{1}}_p(X))^\top}
\newcommand{\indvecp}{(\hat{\mathbbm{1}}_1 (X') , \ldots, \hat{ \mathbbm{1}}_p(X'))^\top}
\newcommand{\hatEDP}{\hat {\mathcal E}^{\text{DP}}}

\newcommand{\PR}{\mathbb{P}}

\DeclareMathOperator*{\argmax}{arg\,max}

\usepackage[normalem]{ulem}  

\begin{document}

\begin{frontmatter}
\title{Differentially private testing for relevant dependencies in high dimensions 
}
\runtitle{Relevant dependencies in high dimensions under privacy}
\begin{aug}

\author[A]{\fnms{Patrick}~\snm{Bastian}\thanksref{t1}}
\author[B]{\fnms{Holger}~\snm{Dette}\thanksref{t2}}
\author[B]{\fnms{Martin}~\snm{Dunsche}\thanksref{t3}}
\thankstext{t1}{E-mail: patrick.bastian@rub.de}
\thankstext{t2}{E-mail: holger.dette@rub.de}
\thankstext{t3}{E-mail: martin.dunsche@rub.de}
\thankstext{eq}{Authors contributed equally and are listed alphabetically.}
\address[A]{Department of Mathematics, Aarhus University}
\address[B]{Fakultät für Mathematik, Ruhr-University Bochum}

\begin{abstract}

\noindent We investigate the problem of detecting dependencies between the components of a high-dimensional vector. Our approach advances the existing literature in two important respects.  First, we consider the problem  under privacy constraints. Second, instead of testing whether the coordinates are pairwise independent, we are interested in determining whether certain pairwise associations between the components (such as  all pairwise Kendall's $\tau$ coefficients) do not exceed a given threshold in absolute value. Considering hypotheses of this form is motivated by the observation that in the high-dimensional regime, it is rare and perhaps impossible to have a null hypothesis that can be modeled exactly by assuming that all pairwise associations are precisely equal to zero.\\
 The formulation of the null hypothesis as a composite hypothesis makes the problem of constructing tests already non-standard in the non-private setting. Additionally, under privacy constraints, state of the art procedures rely on permutation approaches that are rendered invalid under a composite null. We propose a novel bootstrap based methodology that is especially powerful in sparse settings, develop theoretical guarantees under mild assumptions and show that the proposed method enjoys good finite sample properties even in the high privacy regime. Additionally, we present applications in medical data that showcase the applicability of our methodology.

\end{abstract}
\end{aug}


\end{frontmatter}

\maketitle

\defaultbibliographystyle{apalike}
\defaultbibliography{literature}
\begin{bibunit}

\section{Introduction}\label{sec:intro}
  \def\theequation{1.\arabic{equation}}	
	\setcounter{equation}{0}
    
The ability to measure and detect statistical dependence lies at the heart of many scientific questions. From the early works of \cite{Pearson1920}; \cite{Kendall1938} to modern machine learning methods, independence testing has been applied for this purpose across many fields such as genetics, neuroscience, medicine  or economics. Classical approaches target low-dimensional settings where $p \ll n$ and generally do not perform well when $p$ is comparable to, or even larger than $n$. In the era of big data, attention has shifted toward this high-dimensional regime, making classical methods inadequate. Consequently, specialized procedures for such high-dimensional settings have been developed and for a 
comprehensive review of the state of the art we refer the reader to the review at the end of this section. 

However, to the best of our knowledge, two crucial aspects are almost entirely missing from the current literature on high-dimensional independence testing. The first one is  privacy. In many fields, including medical studies, behavioral research, and other contexts where sensitive personal information is involved, statistical analysis must reconcile fundamentally competing aims: drawing meaningful conclusions from data while protecting the privacy of individuals who contribute to it. Traditional anonymization or aggregation techniques are often insufficient, as subtle differences in released statistics can still reveal individual participation. Differential privacy (DP) \citep{dwork2006differential} addresses this challenge by providing a formal and quantifiable notion of privacy. It ensures that the inclusion or exclusion of a single observation has only a limited effect on the output, thereby bounding the information that can be inferred about any individual. This framework has become the de-facto standard in privacy-preserving data analysis, offering a principled foundation for statistical inference under privacy constraints. The second is more of statistical nature. Most existing methods aim for detecting arbitrarily  small dependencies between the components of high-dimensional vectors and are based on tests for hypotheses 
\begin{equation}\label{eq_classical_hypotheses}
\begin{split}
     &H_0:\theta_{ij}=0 \quad \text{for all } 1\leq i<j\leq d\\
     &H_1:\theta_{ij} \not = 0    \quad   \text{for at least one pair }   (i,j) \text{ with }   1 \leq i < j \leq d \,,
\end{split}
\end{equation}
where $\theta_{ij}$ is some (population) measure of dependence between $i$th and $j$th component (such as the covariance or Kendall's $\tau$) of a $d$-dimensional vector. However, in the big data era, where the dimension $d$ and the sample size is large, hypotheses of this form are usually not of interest in practice. In the high-dimensional regime it is very uncommon that all pairwise components of a vector are independent. Consequently, using a consistent test with a sufficient amount of data will virtually always result in rejection of \eqref{eq_classical_hypotheses}. 
This point of view is in line with \cite{tukey1991}, who  succinctly stated it in the context of the comparison of multiple means: ``\textit{Statisticians classically asked the wrong question—and were willing to answer with a lie, one that was often a downright lie. They asked “Are the effects of A and B different?” and they were willing to answer “no.''}''

As pointed out by \cite{Berger} in the context of comparing two univariate means, a possible remedy for these issues is to consider the null hypothesis that all pairwise associations do not exceed a given threshold, say $\Delta>0$. More specifically, we will 
 consider the hypotheses
\begin{equation}\label{eq_relevant_hypotheses}
    \begin{split}
   &  H_0 (\Delta) : ~~  | \theta_{ij} | \leq \Delta   ~~ \text{ for all } 1 \leq i < j \leq d\,, \\
   & H_1 (\Delta):  ~~ | \theta_{ij} | > \Delta   ~~   \text{ for at least one pair }   (i,j) \text{ with }   1 \leq i < j \leq d\,. 
    \end{split}
\end{equation} 
Thus we are interested in testing for at least one {\it practically significant} association, which is larger (in absolute value) than the threshold $\Delta $, and we call the hypotheses in \eqref{eq_relevant_hypotheses} {\it relevant}  
hypotheses throughout this paper.
For a more thorough discussion of relevant hypotheses we refer the reader to Section \ref{subsec:rel_hyp}.  
To illustrate our point of view  here  more concretely, we refer to an  example of gene expression networks discussed in \cite{Tsaparas2006}, where a $28$-dimensional gene expression network is evaluated based on  Pearson correlation coefficients to determine  co-expressed genes. For a meaningful analysis it is necessary to prevent congelation of the network due to spurious correlations, corresponding to a threshold $\Delta$ that demarcates spurious and meaningful correlations (they choose $\Delta=0.7$ even in this fairly low dimensional setting). We discuss this example in more detail in Section \ref{sec:4}.
\smallskip

In the high-dimensional setting, either of these two aspects alone already poses considerable challenges for the development of statistical methodology (see the discussion of the related literature below). In this paper, we develop methodology that addresses both aspects. To this  end, we design a framework for  testing for dependencies in high-dimensional vectors under DP which accommodates a broad class of commonly used dependence measures and is therefore applicable to a wide range of settings. One particularly challenging problem of this endeavor will be the private estimation of so called {\it ``extremal sets''} that plays  a crucial role when constructing powerful tests for hypotheses of the form \eqref{eq_relevant_hypotheses}. As we will see in Section \ref{sec:highdim}, existing approaches to related problems either perform poorly in finite samples or are not practically feasible. The main contributions of the present paper are the following: 

\begin{itemize}
    \item[(1)] We introduce a practical framework for testing for relevant dependencies between the components of a high-dimensional vector that achieves good finite-sample performance combined with strong privacy guarantees.
    \item[(2)] We establish rigorous statistical guarantees for our approach that enable further generalizations and extensions.
    \item[(3)] We demonstrate the versatility of the developed method on two high-dimensional data sets, illustrating that meaningful inference under privacy constraints is still possible in extremely high-dimensional settings.
    \item[(4)]  The new procedure preserves privacy and additionally  attains finite-sample performance that is at least comparable - if not better than that of existing non-private approaches. As the primary focus of this article is private inference we demonstrate this superiority in Section \ref{appadditionalsim} of the online supplement. 
    \item[(5)] We release an open-source implementation. \footnote{\url{https://github.com/martindunsche/Highdimensional_U_statistics_under_privacy}}
\end{itemize}

\textbf{Related work}
High-dimensional (in-)dependence testing is by now a fairly mature field with a substantial amount of publications. For Gaussian data, we refer, among many others,  to \cite{Jiang2015} and  \cite{boddetpar2019} who investigated the asymptotic properties of  likelihood ratio tests. For more general distributions, dependence is usually quantified by different correlation measures such as Pearson's $r$, Spearman's $\rho$, and Kendall's $\tau$, and different functions are used to aggregate these estimates of the pairwise dependencies. For example, \cite{Bao2015} and \cite{Li2021} use linear spectral statistics of the matrix of the estimates of the pairwise dependencies,
while, among others,  \cite{yaoetal2017} and \cite{Leung2018} propose tests based on the Frobenius norm. Other very popular methods of aggregating estimates of the pairwise dependencies are maximum-type tests, which have good power properties
against sparse alternatives and have been investigated for various
covariance/correlation statistics, see \cite{hanetal2017}, \cite{drttonetal2020}  
  and  \cite{heetal2021} for some recent work and the references therein.
A common aspect of this literature is  that  the authors consider the hypotheses in \eqref{eq_classical_hypotheses} of exact pairwise independence. Their  methodology is based on the asymptotic distribution of a test statistic under the hypothesis of independence and can therefore not be extended to testing relevant hypotheses of the form \eqref{eq_relevant_hypotheses}. This problem has been addressed by  \cite{patrick_annals}, but, to our best knowledge, there does not exist any work on differentially private inference in this context. We substantially deviate from the methodology in the aforementioned reference, by introducing so called {\it ``extremal sets''} and corresponding estimators.  Using this information we  implement a powerful parametric bootstrap test, which accommodates privacy constraints and exhibits superior performance compared to the test of \cite{patrick_annals}, even in many situations where private inference is not required. 

A substantial body of research has investigated differentially private inference procedures across a wide range of statistical problems 
 \citep[see, for example,][]{pmlr-v54-rogers17a,  sei2021privacy}.
 Most of this literature has focused on testing tasks for specific parameters, such as sample means, rather than on methods that are broadly applicable to general classes of hypothesis testing problems. \cite{chaudhuri2024differentially} studied $U$-statistics  to enhance private estimation and, in turn, inference in finite dimensional settings. In a closer vein to the present paper \cite{liu2025differentially} analyzed $U$-statistics with an emphasis on independence testing, making use of classical permutation tests. Unfortunately this approach is not applicable in the relevant hypothesis framework as the necessary permutation invariances are not valid under the null hypothesis \eqref{eq_relevant_hypotheses}, which we will consider. 

The majority of work on DP hypothesis testing does not consider the high-dimensional regime and relies on parametric bootstrap approximations of the (quantiles of the) relevant statistic, or by direct analysis of the noise introduced for privacy protection  \citep[see][for an overview]{dunsche2022multivariate}. 
On the other hand, \cite{liu2022differential} proposed a general framework for DP estimation in high dimensions. While obtaining strong theoretical results, their methodology does not yield inferential guarantees, which can be used for hypotheses testing,  and also may not be computationally feasible in many cases. Moreover, \cite{cai2024optimal} investigated  principal component analysis (PCA) in high-dimensional spiked covariance models, while \cite{canonne2020private}  and \cite{pmlr-v178-narayanan22a} developed a private framework for identity testing in high dimensions.

\section{Background}\label{sec:background}
  \def\theequation{2.\arabic{equation}}	
	\setcounter{equation}{0}

In this section, we briefly revisit key concepts for the subsequent development of our methodology. In particular we recall the specific notions of testing relevant hypotheses, $U$-statistics and Differential Privacy (DP).

\subsection{Relevant Hypotheses}\label{subsec:rel_hyp}
For a  $d$-dimensional vector  $X_1=(X_{11}, \ldots X_{1d})^\top $ let  $\theta_{ij} = \theta (X_{1i}, X_{1j})$ denote a dependence measure between the $i$th and $j$th component. 
We propose to investigate if all associations  $(\theta_{ij})_{1 \leq i<j\leq d}$ are in some sense ``small'' by testing the hypotheses 
\eqref{eq_relevant_hypotheses}, 
where $\Delta > 0 $ is a given threshold, which defines when a dependence between the components $i$ and $j$ is considered as (scientifically) not relevant. As pointed out in the introduction, the consideration of hypotheses of this form is motivated by the observation that in many applications it is very unlikely that all pairwise associations are completely $0$, in particular if the dimension  $d$ is large. We thus argue that it is more reasonable to test for at least one {\it practically significant} association. 

An essential ingredient in this approach is the specification of the    threshold $\Delta$, and its choice depends sensitively on the particular problem under consideration. 
Essentially, this boils down to the important question when a correlation (or another dependence measure) is {\it practically significant}, which has a long history in applied statistics. For the particular case of dependence measures that we consider in the present paper several authors classify the strength of association between variables for their particular application into categories such as ``small'',  ``medium'' or ``large''. The precise demarcation thresholds vary across disciplines and subject areas and we refer the interested reader to  \cite{Quintana2016}, \cite{Brydges2019} and \cite{Lovakov2021} 
for a discussion of this choice for concrete applications. In this paper, we will consider two data examples, one from genomics and one from cancer research. For the genomic data and in particular gene expression networks it is customary to discard correlations below a certain threshold to facilitate a meaningful analysis. In this context \cite{Tsaparas2006} proposed $0.7$ as a threshold for Pearson correlations and considered correlations below $0.7$ as spurious. For the cancer data on the other hand, we will apply our methodology to two distinct groups within the data set - patients with cancer and those without - and then use the smallest $\Delta$ for which the null hypothesis in \eqref{eq_relevant_hypotheses} is rejected as a means of structurally discriminating between the two groups. We thus obtain a natural, data dependent choice for the threshold $\Delta$. In the following remark, we make this argument more precise and explain why this yields a valid inference procedure.

\begin{remark}\label{rem:measure_of_evidence}
 {\rm Note that the hypotheses $H_0(\Delta)$ in \eqref{eq_relevant_hypotheses} are nested and that families of test decisions $\phi(\Delta)$ for such hypotheses are often monotone in $\Delta$. Consequently, rejecting $H_0(\Delta)$ for $\Delta=\Delta_1>0$ also implies rejecting $H_0(\Delta)$ for all $0<\Delta<\Delta_1$. The sequential rejection principle then allows us to  simultaneously test the hypotheses \eqref{eq_relevant_hypotheses} for different choices of $\Delta> 0$ until we find the minimum value $\hat \Delta_\alpha$  for which $H_0(\hat \Delta_\alpha)$ is not rejected, that is  
\begin{align}
       \hat \Delta_\alpha:=\min \big \{\Delta \,| \, \phi(\Delta)=0  \big  \}~.
\end{align}
where we define the minimum of an empty set to be 0. Consequently, one may postpone the selection of $\Delta$ and derives a test decision in the fashion as comparing the $p$-value to the prescribed type I error.
    }
\end{remark}

\subsection{U-statistics}\label{subsec:U_statistic}

We will phrase the testing problem \eqref{eq_relevant_hypotheses} in the framework of $U$-statistics as many dependence measures used in practice can be expressed in this way. To be precise, let $X_1,\ldots ,X_n$ denote independent identically distributed $d$-dimensional random vectors with distribution function $F$. Note that formally $F$ depends on the dimension $d$, which in this paper we allow to vary with $n$, but we will not reflect this dependence in our notation throughout this paper. For some positive integer $r$ let 
\begin{align} \label{hd2}
h =(h_1, \ldots , h_p)^\top :\big ( \R^d\big)^r \rightarrow \R^p
\end{align}
denote a measurable symmetric function  with finite expectation 
\begin{align} \label{hd0}
\theta=(\theta_1,\ldots,\theta_p)^\top :=\E_F[h(X_1, \ldots ,X_r)] \in \mathbb{R}^{p }~,
\end{align}
 which defines our  parameter of interest.  We are interested in the relevant hypotheses
\begin{equation}\label{eq_rel_hyp_high}    
      H_0: ~~  \norm{\theta}_\infty \leq \Delta~, \quad H_1:  ~~ \norm{\theta}_\infty > \Delta
\end{equation} 
 for some $\Delta>0$, where $\| \cdot \|_\infty$ denotes the maximum-norm. 
 In the context of testing for pairwise dependencies, the dimension $p$ will typically be given by $p=d(d-1)/2$ as illustrated in the following example.

\begin{example} \label{ex1}
{\rm
For the dependence measure between the $i$th and $j$th 
component of the vector $X_{1}= (X_{11} , \ldots  , X_{1d} )^\top $
 as introduced in Section \ref{sec:intro}, we assume that
\begin{align*}
\theta_{ij} =   
\theta(X_{1i},X_{1j})
= \mathbb{E}[\tilde h ( X_{1i},X_{1j}, \ldots ,X_{ri},X_{rj})]~~ ~~~~ 1 \leq i < j \leq d~.
\end{align*}
Here, $\tilde h: (\R^{2})^r \to \mathbb{R}$ is a kernel of order $r$
evaluated at $( X_{1i},X_{1j}), \ldots, (X_{ri},X_{rj} )$. In this case the function $h:\R^{dr}\to \R^{p}$  in \eqref{hd2}  is  defined by
\begin{align*}
h (X_1,...,X_r)  & = {\rm vech} \big ( ( h_{ij} (X_1,...,X_r) )_{i,j=1, \ldots , d}  \big ) \\
& = {\rm vech} \big ( (\tilde h ( X_{1i},X_{1j},...,X_{ri},X_{rj}) )_{i,j=1, \ldots , d}  \big )~,
\end{align*}
where the second equality defines the functions $h_{ij} : \mathbb{R}^{dr} \to \mathbb{R}$ in an obvious manner and 
vech($\cdot$) is the operator that stacks the columns  above  the diagonal of a symmetric 
$ d \times d$ matrix as a vector with $p = d(d -1)/2$ components. Note that the index $(i,j)$ 
in the definition of the function $h_{ij}$
is only used to emphasize that each $h_{ij}$ acts on different components of the $d$-dimensional vectors  $X_{1}, \ldots ,X_{r}$. 
Similarly, the vector $\theta$ is defined by
$\theta  =  {\rm vech} \big ( ( \theta_{ij} )_{i,j=1, \ldots , d}  \big )$, and the components of the vector
$U= {\rm vech} \big ( ( U_{ij} )_{i,j=1, \ldots , d}  \big ) $ in \eqref{hd3} are given by
\begin{align*}
 U_{ij}  &= {n \choose r}^{-1}\sum_{1 \leq l_1<...<l_r \leq n}  h_{ij} ( X_{l_1},...,X_{l_r})\\
 &= {n \choose r}^{-1}\sum_{1 \leq l_1<...<l_r \leq n} \tilde h (X_{l_1i},X_{l_1j},...,X_{l_ri},X_{l_rj})\,.
\end{align*}
Finally, we note that it is easy to see that with these notations the hypotheses  \eqref{eq_rel_hyp_high} are equivalent to \eqref{eq_relevant_hypotheses}.
}
\end{example}

We now return to the general case discussed at the beginning of this section. In order to estimate the parameter $\theta$ we consider a $U$-statistic of order $r$ given by
\begin{align} \label{hd3}
    U= (U_1 , \ldots , U_p)^\top = {n \choose r}^{-1} \sum_{1 \leq l_1< \ldots <l_r \leq n}h(X_{l_1}, \ldots ,X_{l_r})\,.
\end{align}

As our primary goal is developing inferential methodology for the hypotheses \eqref{eq_relevant_hypotheses}, we will need estimates of the asymptotic covariance structure 
 $\zeta_1 =(\zeta_{1,ij})_{i,j=1, \ldots , p}$
of the vector $U$, where 
\begin{align}
    \label{eq_var_def}
    \zeta_{1,ij}:=\text{Cov}_F(h_{1,i}(X_1),h_{1,j}(X_1))
\end{align}
and 
\[
    h_{1}(x)=(h_{11}(x), \hdots, h_{1p}(x))^\top=\E_F[h(X_1,....,X_r)|X_1=x]
\]
is the linear part of the Hoeffding decomposition.
To estimate (a multiple)  of $\zeta_1 =(\zeta_{1,ij})_{i,j=1, \ldots , p}$, we utilize the classical Jackknife estimator
\begin{equation}\label{eq_jackknife}
    \hat\zeta_{1}:= (n-1) \sum_{l=1}^n (U^{(l)}-U)(U^{(l)}-U)^\top~,
\end{equation}
where $U^{(l)}$ denotes the leave one out $U$-statistic of \eqref{hd3}. In Lemma \ref{VarConv} of the online supplement we show that this yields a maximum-norm consistent estimator for $r \cdot \zeta_1$ even in the ultra high-dimensional setting where $p=o(\exp(n^{1/5}))$. 

\subsection{Differential Privacy}\label{subsec:differential_privacy}
We recall basic notions used throughout this paper and forward interested readers to \cite{dwork2014algorithmic} and   \cite{bun2016concentrated} for an overview of DP or Zero-Concentraded Differential Privacy (zCDP), respectively. 

We call two datasets $X$ and $X'$ neighbors (denoted $X \sim X'$ or $d_H(X,X')=1$, where $d_H$ is the Hamming distance) if they differ in exactly one individual. Given a possibly vector-valued query (statistic) $f$, its global sensitivity with respect to a norm $\|\cdot\|$ is defined as
\[
\Delta f \;:=\; \sup_{X \sim X'} \| f(X) - f(X') \|,
\]
and quantifies the largest change in $f$ when a single record in the data set $X$ is modified for any $X$. 
We define for two distributions $P$ and $Q$ and $\alpha>1$ their  Rényi-divergence by  
$$
D_\alpha(P,Q)=\frac{1}{\alpha-1}\log \Big( \int p(t)^\alpha q(t)^{1-\alpha} d\mu(t)\Big) ,
$$ 
where $p$ and $q$ are densities of $P$ and $Q$ with respect to some dominating measure $\mu$. 

\begin{definition}[Definition 8.1 in \cite{bun2016concentrated}, Approximate zCDP]
\label{defcdp}
A randomized algorithm $\mathcal{M}$ is called
$\delta$-approximate-$\rho$-zCDP if for all neighboring data sets $X$ and $X'$, there exist events 
$E$ (depending on $\mathcal{M}(X)$) and $E'$ (depending on $\mathcal{M}(X')$) such that 
$$\PR[E] \ge 1 - \delta\quad \text{and}\quad\PR[E'] \ge 1 - \delta~,$$ and we have  for all $\alpha>1$
\[
D_\alpha(P, Q)\le \rho\alpha\quad \text{and} \quad D_\alpha(Q, P)\le \rho\alpha
\] 
where $P$ and $Q$ are the distributions of $\mathcal{M}(X)$ and $\mathcal{M}(X')$ conditional on $E$ and $E'$, respectively.
\end{definition}
In the case $\delta=0$, $0$-approximate-$\rho$-zCDP recovers the classical $\rho$-zCDP. By Lemma 8.2 in  \cite{bun2016concentrated}, $\delta$-approximate-zCDP satisfies the composition and post-processing property. For completeness and later use, we will also recall the definition and the privacy guarantees of the most prominent algorithm.
\begin{Lemma}
\label{Lem_steinke_gauss} Let $T$ denote a $\mathbb{R}^d$-valued statistic.
 The Gaussian mechanism $\mathcal{M}(X) = T(X) + \frac{\Delta_2T}{\sqrt{2 \rho}}  Y$ where $Y \sim \mathcal{N}_d(0, I_{d \times d})$ and $\Delta_2 T:= \sup_{X \sim X'} \| T(X) - T(X') \|_2$, preserves $\rho$-zCDP.
\end{Lemma}
In the context of maximum-type hypotheses such as \eqref{eq_rel_hyp_high} it will be important to extract the largest coordinates of a vector in a differentially private manner. In Algorithm \ref{alg:rep_noisy_max} we therefore formulate an algorithm that generalizes the classical Report-Noisy-Max algorithm  \citep[see e.g.][]{dwork2014algorithmic}, and the following proposition shows that this algorithm is $zCDP$.
\begin{prop}\label{prop:rnm}
Algorithm \ref{alg:rep_noisy_max} is 
$\ve^2/8$-zCDP.
\end{prop}
\vspace{-0.3cm}
\begin{algorithm}[h]
  \caption{Regularized Report-Noisy-Max (RL-GAP)}\label{alg:rep_noisy_max}
  \begin{algorithmic}[1]     
    \Require vector $q=(q_1,\hdots, q_{p})^\top\in\R^p$, privacy parameter $\epsilon$, $\ell_1$ sensitivity $\Delta_1$,
            regularizer $
            \nu:\{1,\dots, p\}\to\mathbb{R}$
 
    \State \Return $ \argmax_{j\in\{1,\hdots, p\}}
        \Bigl\{\,q_j + \nu(j)
        + \mathrm{Gumbel}\!\bigl(\tfrac{2\Delta_1}{\ve}\bigr) \Bigr\}$
  \end{algorithmic}
\end{algorithm}
\vspace{-0.3cm}
\begin{remark}
\label{remreg}
{\rm 
The privacy analysis can be carried out without the regularizer $\nu$. However, we prefer to state 
the slightly more general version with the regularizer $\nu$ to provide practitioners with greater flexibility. For instance, by adjusting $\nu$, one could overweight earlier indices $j$ and underweight later ones, which might be reasonable if the user has prior knowledge.}
\end{remark}

\section{Baseline Methodology}\label{sec:highdim}

  \def\theequation{3.\arabic{equation}}	
	\setcounter{equation}{0}
In this section we consider the $U$-statistic framework  introduced in Section \ref{subsec:U_statistic} and  lay out the challenges one encounters when trying to extend existing testing methodology  for the relevant hypotheses \eqref{eq_rel_hyp_high}
from the private finite dimensional to the private and high-dimensional  setting. Although none of these methods discussed here  can be finally used in the high-dimensional regime, some of their component techniques are useful for the development of our advanced methodology in Section \ref{sec:4}. There, we will propose and theoretically validate a differentially private method that achieves good finite sample performance under mild assumptions - even in the high-dimensional regime. As explained in Example \ref{ex1}, the solution to the problem of testing for relevant dependencies under privacy  constraints appears as a special case if the dependence measure can be expressed as a $U$-statistic.

\subsection{Concentration-based tests and their limitations}
\label{sec:hoeffding}
\subsubsection{A simple  test based on a concentration inequality}
In the following paragraph we will ignore privacy aspects for simplicity of presentation, as they can easily be integrated into the discussion without changing any of the conclusions.

A first simple and very conservative approach can be based on concentration inequalities for $U$-statistics. More precisely, under the null hypothesis  in  \eqref{eq_rel_hyp_high} we have  
\begin{align}
    \label{hol1}
    \sqrt{n} \big ( \norm{U}_{\infty} 
    -\Delta \big )  \leq \sqrt{n}\max_{1 \leq i \leq p}(|U_i|-|\theta_i|)\leq \sqrt{n}\max_{1 \leq i \leq p}|U_i-\theta_i|~,
\end{align}
where  $\norm{U}_{\infty}=\max_{1 \leq i \leq p}(|U_i|$.
For bounded kernels, the classical Hoeffding inequality  \citep{Hoeffding1963} then yields under the null hypothesis in  \eqref{eq_rel_hyp_high} that 
\begin{align*}
    \PR(\sqrt{n}\max_{1 \leq i \leq p} (|U_i|-\Delta)>t)\leq p\max_{1 \leq i \leq p}\PR(\sqrt{n}(|U_i- \theta_i |)>t)\leq 2p\exp\Big(\frac{-t^2}{2\|h\|_\infty r}\Big)~, 
\end{align*}
where $\| h \|_\infty = \max_{i=1}^p \| h_i \|_\infty $ denotes the maximum of the sup-norms of components of the vector $h= (h_1, \ldots , h_p)^\top$.  Therefore, it is easy to see that the decision rule
\begin{align}
\label{hol2}
\phi(x) = 
\begin{cases}
  1, & \text{if } 
    \displaystyle \max_{1 \leq i \leq p} \bigl(|U_i| - \Delta \bigr) 
    > \sqrt{\tfrac{2 \log(2p/\alpha)\,\|h\|_\infty\, r}{n}}, \\[1.2ex]
  0, & \text{otherwise}
\end{cases}~.
\end{align}
defines  a consistent level $\alpha$ test. Unfortunately, it is well known that tests of this type are extremely conservative (see also our numerical results in Figure \ref{fig:power_curves_grid_con} in  Section \ref{sec:experiments}).
As the performance of this approach will deteriorate even further when additional noise is introduced to
ensure privacy, we will need to improve upon it.

\subsubsection{Extension of existing finite dimensional methodology}
\label{sec312}
As a first approach we will reflect upon the results concerning the fixed dimension setting presented in the PhD thesis of \cite{Dunsche2025}. As we will see later, the methods developed therein fail in the high-dimensional regime, but this approach nonetheless serves as a good introduction to our general methodology developed in Section \ref{sec:4}.

Based on the observation that differentially private testing procedures inflate the variance of the test statistic in an asymptotically but not finite sample negligible manner, a parametric bootstrap procedure was constructed that takes the privatization noise into consideration. It utilizes an analog of the upper bound in \eqref{hol1} for a consistent and private estimator of   $\norm{\theta}_\infty$, say $\norm{U}_{\infty}^{\text{DP}}$. 
The right-hand side of the corresponding inequality can be approximated by the  $\| \cdot \|_\infty$-norm of a Gaussian vector $Z\sim \mathcal N(0,\zeta_1)$, where  $\zeta_1 =(\zeta_{1,ij})_{i,j=1, \ldots , p}$ is defined in \eqref{eq_var_def}. 
Therefore, we define the decision rule
\begin{equation}\label{eq_finite_dim}
   T^{\text{DP}} := \sqrt{n}\bigl(\norm{U}_{\infty}^{\text{DP}} - \Delta\bigr) > q_{1-\alpha}^*~,
\end{equation}
where $q_{1-\alpha}^*$ denotes the $(1-\alpha)$-quantile of the distribution of $\| Z \|_\infty^{\text{DP}} $ with $Z\sim \mathcal N(0,\hat \zeta_1^{\text{DP}})$. Here, $\hat \zeta_1^{\text{DP}}$ is a consistent private estimator of the covariance matrix $\zeta_1 $. More details on this parametric bootstrap test can be found in Algorithm \ref{alg_monte_carlo_quantile_mult} in Section \ref{appendixalg} of the online supplement.
For finite dimension, this yields a consistent and asymptotic level-$\alpha$ test under suitable regularity conditions. A precise formulation of this statement  and a proof are provided in Section ~\ref{App:finite_dimensional} of the online supplement.
However,  its performance deteriorates with increasing dimension, since standard private estimates of the asymptotic variance may become inconsistent in the high-dimensional regime. For instance, under the basic additive Gaussian mechanism for private covariance estimation  \cite[see Algorithm \ref{alg_gausscov} in][]{10.1145/2591796.2591883},
the private estimator $\hat{\zeta}_1^{\text{DP}}$ is already inconsistent in the regime $p \simeq \sqrt{n}$ with respect to the entry-wise maximum norm.

Moreover, even in the finite dimensional setting, the test \eqref{eq_finite_dim} is suboptimal, including the non-private case. To highlight the difficulties one even encounters here note that the quantile for the supremum based test statistic $T^{\text{DP}}$ in the decision rule \eqref{eq_finite_dim} is calculated 
from a privatized maximum norm of the a $p$-dimensional normal distribution $\mathcal N(0,\hat \zeta_1^{\text{DP}})$ after an application of an inequality of the type \eqref{hol1}. However, results on the directional differentiability of the supremum norm  \citep[see Theorem 2.1 in][] {carcamo2020directional} show that the asymptotic distribution of the statistic $T^{\text{DP}}$ is given by the maximum norm of a $k$-dimensional distribution $\mathcal N (0, \zeta_1(k))$, where \begin{align}
        \label{eq:limit:cov}
         \zeta_1(k) = \big (
   \text{sign}(\theta_i\theta_j)\zeta_{1,ij} \big )_{  i,j \in \{i_1,...,i_k\}} 
    \end{align} denotes  the 
$k \times k$ matrix, which is obtained from the matrix $\zeta_1 = \big (\zeta_{1,ij} \big )_{  i,j \in \{1,...,p\}}$  by selecting the specific rows and columns with  indices in the  {\it extremal set } 
\begin{equation}\label{eq:rel_def}
  \mathcal E:= \{i_1,...,i_k\}:=\left\{\, i =1,\hdots, p \; : \;
    |\theta_i| = \norm{\theta}_\infty \right\} ~.
\end{equation}
Thus, asymptotically we are comparing the maximum norm of a $k$-dimensional normal distribution with the maximum  norm of a $p$-dimensional normal distribution, resulting in an extremely conservative test with not much power. 

This phenomenon becomes particularly striking if the cardinality is substantially smaller than the dimension, that is $k\ll p$.  In this case, the bootstrap test \eqref{eq_finite_dim} constructs too large quantiles based on the aggregated noise of all coordinates, leading to suboptimal performance in finite samples. This situation becomes even worse when $p$ grows with $n$ (here neglecting the fact that private estimation already fails).
\smallskip 

Despite these discouraging observations, this discussion also suggests a solution of this problem. Estimate the coordinates in the extremal set \eqref{eq:rel_est}
and the signs of the corresponding coefficients $\theta_{ij}$ in the vector $\theta$ and apply a modified version of the test \eqref{eq_finite_dim} using the quantiles of the (privatized) maximum 
$\max_{1 \leq i \leq k}Z_i$, where $ Z\sim \mathcal N (0, (\hat \zeta_1(k))^{\text{DP}})$ follows a $k-$dimensional normal distribution and $\hat \zeta_1^{\text{DP}} (k)$ is the corresponding privatized estimator of the $k\times k$  covariance matrix  \eqref{eq:limit:cov}.   
In a non-private setting this solves the problem, but unfortunately it fails when taking privacy into account for the following reasons: 
\begin{enumerate}   
    \item The resulting vector of estimated relevant coordinates might still be high-dimensional in the sense that $k\gtrsim n$, which leads to inconsistent private covariance estimation.  
    \item Privately estimating the extremal set   $\mathcal E$  is a highly non-trivial problem. Even if the number of elements of the set ${\cal E}$ were known, inferential methodology based on  standard private selection methods, like top-$k$ selection \citep{qiao2021oneshot} or offline sparse vector techniques \citep{lyu2016understanding}, perform poor in the present setting, and we refer to Section   \ref{sec:svt} of the online supplement for a more thorough discussion of this fact.
    \item  Approaches based on estimating sub-sets of the extremal set which are still sufficiently small to ensure feasibility of the bootstrap also need to be designed carefully: the quantity 
    $\text{sign}(\theta_i \theta_j)$ in the
    covariance entries in \eqref{eq:limit:cov} can be estimated by $\text{sign}(U_iU_j)$. However, this statistic  can
have a sensitivity of constant order because  the sign  $\text{sign}(U_iU_j)$ might flip with non-negligible probability if one of the coordinates $\theta_i$ or $\theta_j$ vanishes.
\end{enumerate}
\section{Extremal set estimation: balancing privacy and statistical accuracy in high dimensions}
\label{sec:4}

  \def\theequation{4.\arabic{equation}}	
	\setcounter{equation}{0}

Our goal is to design a \textit{differentially private} testing procedure for the hypotheses in  \eqref{eq_rel_hyp_high} in the high-dimensional setting where $p$ diverges with $n$.  The discussion in the previous section suggests that this task requires an efficient procedure for estimating the extremal set $\mathcal E$  defined 
in \eqref{eq:rel_def}.  A canonical non-private estimator of this  set is given by
\begin{equation}\label{eq:rel_est}
  \hat{\mathcal E}= \Big \{\, i =1,\hdots p \; : \;
   |U_i|\geq \norm{U}_\infty
   -  \sqrt{\tfrac{\log(p)\,\log(n)}{n}}
  \Big \}~,
\end{equation}
but it is a highly non-trivial problem how to privatize this estimate. In the following discussion we will motivate an alternative approach that a) reduces the dimension, making private covariance estimation feasible and that b) retains favorable statistical properties in a wide range of real world scenarios while also maintaining appropriate privacy guarantees. We start with the most obvious method, namely the sparse vector technique (SVT) and explain why it fails in the present context. Motivated by an analysis of the failure, we then propose our approach, which includes an adaptive choice of the top-$k$ components. Here adaptivity refers to the estimation of the cardinality $k = \# {\cal E}$ of the extremal set \eqref{eq:rel_def}, which is a non-trivial choice to make in advance and heavily depends on the data at hand.

A natural approach to construct a private estimator of $\mathcal{E}$ is the SVT \citep[see e.g.][]{lyu2016understanding,NEURIPS2020_e9bf14a4}, which is described in Section \ref{sec:svt} of the online supplement for the sake of completeness. For clarity, we also provide the full algorithmic description in Section \ref{appendixalg} of the online supplement. The basic idea is to privately identify all queries that surpass some prescribed threshold $t$. Setting the queries $q_i= |U_i|-\norm{U}_\infty$ for all $i=1,\hdots p$ and the threshold as $t=-\sqrt{{\log(p)\,\log(n)}/{n}}$, then yields a private  and consistent (under mild assumptions) estimator of the extremal set $\mathcal E$.
 Unfortunately, while each individual query has small sensitivity of order $\tilde{{O}}(1/n)$, the finite sample performance of this approach suffers severely from the composition of the $k$ successes (in which $q_i$ surpasses $t$) and the number of queries $p$ in the  high-dimensional setting. The resulting extremal set estimator is not usable in practice as we either obtain a poor estimate or poor privacy guarantees. For an empirical illustration of these issues, we refer the reader to Section \ref{sec:svt} of the online supplement.

Another common approach for estimating $\mathcal E$ is obtained by privately reporting the top-$k$ values of $|U|$. When $k$ is known in advance, Algorithm 1 in \cite{qiao2021oneshot} could be applied directly to obtain a private estimate of the extremal set. However, in our setting $k$ is unknown. As discussed at the end of the previous section, a poor choice will lead to systematic size inflation or deflation when $k$ is chosen too small or large, respectively. Using the private top-$k$ value algorithm is thus not feasible without a good estimate of $k$.  Moreover, even with knowledge of $k$, the privacy cost of the top-$k$ algorithm scales with $\sqrt{k}$ which is undesirable already for fairly small $k$.
\smallskip

Motivated by this discussion, we propose a two step procedure that
\begin{itemize}
    \item[(1)] constructs a differentially private estimator $\hat k$ of $k$. 
    \item[(2)] based on the estimator $\hat k$ privately estimate the  coordinates of the extremal set $\{i_{1},\hdots, i_{k}\}$ in \eqref{eq:rel_def} by $\{\hat i_{1},\hdots, \hat i_{\hat k}\}$ in a one-shot approach instead of iteratively sampling them.
\end{itemize}
\smallskip

\textbf{(1) Adaptive top-$k$ components:} For the private estimation of $k$ we will modify an approach for counting queries proposed by \cite{pmlr-v151-zhu22e}, such that it is effective in our setting as well. Let $|U|_{(1)}\geq \hdots\geq |U|_{(p)}$ denote the order statistics of the absolute values of the components of the vector $U=(U_1,\hdots, U_p)^\top$ of $U$-statistics. We begin by choosing a suitable number of coordinates that will be included in the estimator of the set $\mathcal{E}$. The construction is motivated by the observation that there is often a small number of coordinates with a large signal that are clearly separated from the remaining bulk of coordinates. 
To make our point clear, instead of iteratively identifying potential coordinates of the extremal set, as it is done by SVT in Algorithm \ref{alg_gsvt} of the online supplement, we aim to identify a single index that separates the extremes from the bulk of the data $(|U|_i)_{1 \leq i \leq p}$.  In doing so, we reduce the complexity  to a ``one-dimensional''  problem in a composition sense, making it very cheap from a privacy perspective. We illustrate our observation in Figure \ref{fig:histogram}, which displays the pair-wise Kendall's $\tau$'s for the genome data set discussed in Section \ref{subsec:real_world_examples}.
 
\begin{figure}
    \centering
    \includegraphics[width=6cm]{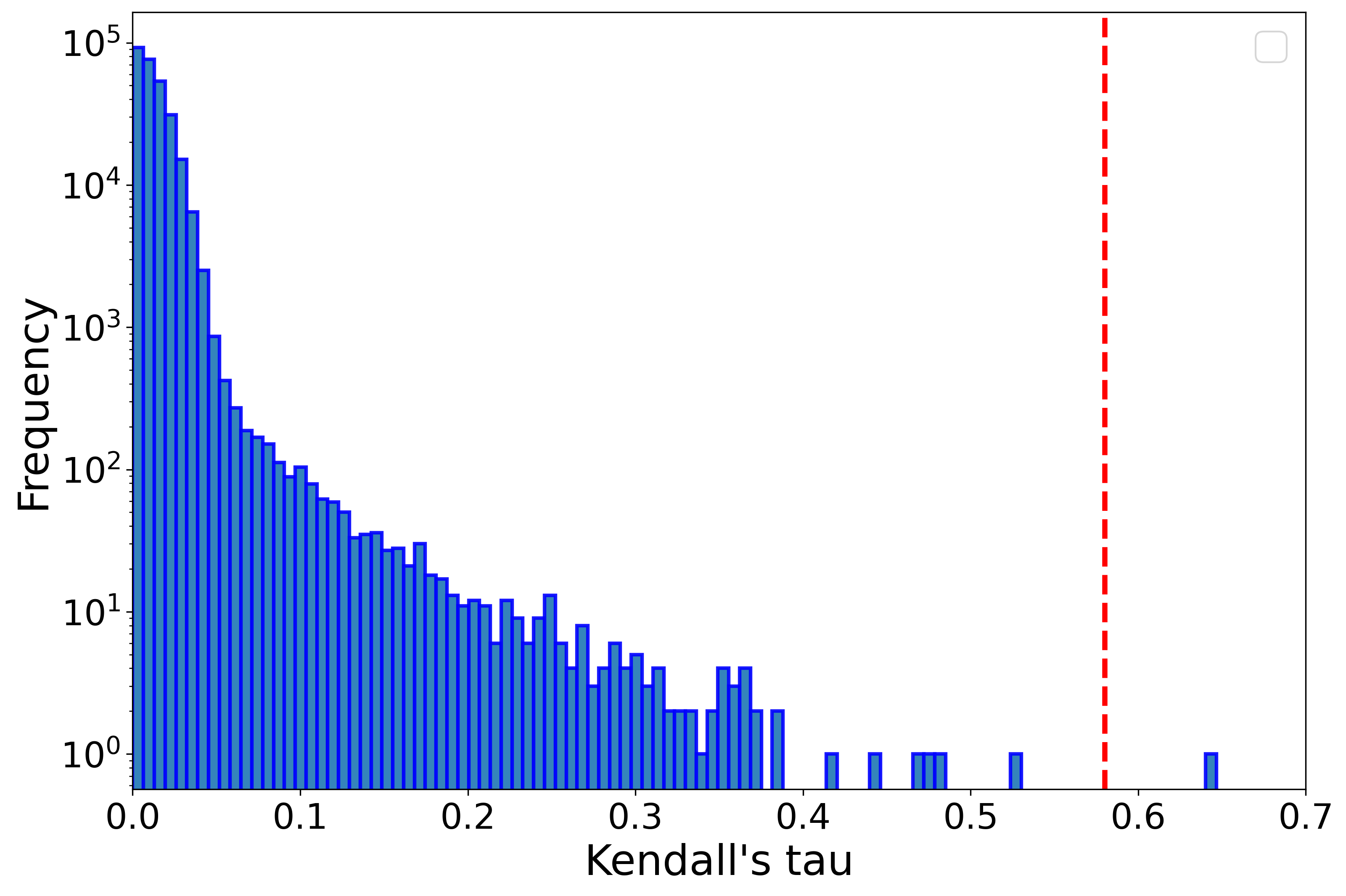}
    \caption{\it Histogram of pair-wise (absolute) Kendall's $\tau$ coefficients between different genomes from the \cite{1000GenomesProjectConsortium2015} of the $21.55$ Mb - $21.65$ Mb window restricted to chromosome 22.}
    \label{fig:histogram}
\end{figure}
For a closer alignment with the notation in  \cite{pmlr-v151-zhu22e} we 
characterize the extremal set ${\cal E}$ by a $p$-dimensional  vector of indicators $(\mathbbm{1}_1, \ldots, \mathbbm{1}_p)^\top$, where $\mathbbm{1}_j=1$ if and only if $j\in {\cal E} = \{i_{1},\hdots, i_{k}\}$. To estimate the cardinality $k= \#{\cal E}$, let us first consider the increments of the 
ordered absolute values  of the components of the  $U$-statistic, that is  
\begin{align}
\label{eq:defin:queries}
    q_j := q_j(X)=|U|_{(j)}  - |U|_{(j+1)}\quad j=1,...,p-1~.
\end{align}

Using these increments, we now estimate $k$ by  Algorithm \ref{alg:rep_noisy_max}, which is a differentially private maximum-selection mechanism that reports the coordinate at which the maximum increment is achieved, that is 
\begin{align*}
     \hat k=\textproc{RL-GAP}(q=(q_1,\hdots, q_{p-1}),\ve,4r/n \norm{h}_\infty,\nu=0)~.
\end{align*}
The associated privacy guarantees are stated in  Proposition~\ref{prop:rnm}. In practice, this will often yield a good separation between the coordinates with a large signal and the remaining coordinates (see also Figure \ref{fig:histogram}). There exist also settings, where such a  clear separation is not possible, and we will see later  how our method discriminates between these options in a data-adaptive way.
\smallskip

\textbf{(2) Estimation of the extremal set:} Having obtained an estimator $\hat k$ for $k$, we now move one step further and define estimates of the actual coordinates $i_{1},\hdots, i_{k}$ with privacy guarantees. For that purpose, we make use of the propose-test-release framework of \cite{10.1145/1536414.1536466} and 
 consider the  statistic $(\hat{\mathbbm{1}}_1 (X) , \ldots, \hat{ \mathbbm{1}}_p(X))^\top$, where
\begin{equation*}
  \hat{\mathbbm{1}}_j (X) :=\begin{cases}
        1; \quad \text{if } |U_j|  \text{ is among the
top $\hat k$ U-statistics $ |U|_{(1)} , \ldots , |U|_{(\hat k)} $}\\
        0; \quad \text{else}
    \end{cases}~,
\end{equation*}
and define 
\begin{equation}\label{eq:private_extremal_set}
\hat {\cal E}^{\text{DP}}:= \big \{ j=1, \ldots , p ~|~\hat{\mathbbm{1}}_j (X) =1 \big \}~.
\end{equation}
However, as  pointed out in \cite{pmlr-v151-zhu22e}, the global $\ell_2$ sensitivity of this  query is $\sqrt{2\hat k}$. To solve this issue  we now make use of local sensitivity and the propose-test-release framework \cite[as introduced, for example, in][]{10.1145/1250790.1250803}. More specifically, we will prove in Lemma \ref{lem:local_sensitivity} of the online supplement that the local sensitivity of $\indvec$ is indeed $0$ whenever there is a valid separation between the extremal  coordinates and the bulk (see Assumption \ref{As_Sparse} below for a precise formulation).  This means the following: 
for a  fixed dataset, say $X$, with a sufficiently large gap (of order $O(1/n)$) between the ordered $U$-statistics, changing one entry actually does not change the estimated extremal set. The estimated set in equation \eqref{eq:private_extremal_set} can then be released using the propose-test-release framework. In certain cases this set might still contain more than $O(\sqrt{n})$ points, which is problematic for the privatization of the covariance matrix. To solve this issue, we randomly select $\log(p)$ coordinates from the estimated set and proceed from there. We summarize the described procedure in Algorithm \ref{alg:topk} below and its privacy guarantee is formalized in the following result.
\begin{theorem}\label{thm_topk}
    Algorithm \ref{alg:topk} satisfies $\delta-$approximate-$\rho$-zCDP. 
\end{theorem}
We emphasize that the propose-test-release approach outputs an estimate of the extremal set while using privacy only twice, which is in stark contrast to the SVT-approach.
However,  Algorithm \ref{alg:topk} outputs $\bot$ if it does not find a set of well separated $U$-statistics. In this case  we need a different way of proceeding, and we present two options for this purpose:
\smallskip

\textbf{Option 1  (rely on concentration):}
We can proceed as in the previous section and use a privatized version of the concentration based test 
\begin{equation}\label{eq:hoeffding_test}
\mathbbm{1}\Big\{\norm{U}_\infty^{\text{DP}}-\Delta>\sqrt{\log(2p/\alpha)2\|h\|_\infty r/n}\Big\}~,
\end{equation}
where $\norm{U}_\infty^{DP}:=\norm{U}_\infty+\frac{\Delta_2\norm{U}_\infty}{\sqrt{2 \rho}} Z$ with $Z\sim \mathcal N(0, 1)$, $\Delta_2\norm{U}_\infty=2r\norm{h}_\infty/n$ and $\rho$ is the desired zCDP privacy budget.
This test keeps asymptotically
the nominal level and is  able to detect relevant signals when the sample size is sufficiently large. However, as discussed in Section \ref{sec:hoeffding} this test is extremely conservative and - even neglecting the additional variance due to privacy - not efficient for reasonable sample sizes. We highlight this fact by an empirical study in  Section \ref{sec:experiments}. 
\smallskip 

\textbf{Option 2 (extreme value analysis):}
Assuming that the sign adjusted pairwise correlations $\text{sign}(\theta_i\theta_j)\zeta_{1,ij}$ of $U_i$ and $U_j$ are non-negative for all statistics that satisfy  
$$
|\theta_i|\geq \Delta-\gamma~
$$  
for some $\gamma>0$, we can construct a test, that does not require knowledge about the extremal set, which is less conservative than the concentration based approach and is also asymptotically pivotal. For this purpose we use the observation  that
\begin{align}
\label{hd100}
    \sqrt{n}\max_{1 \leq i \leq p}(|U_i|-\Delta) \leq   \sqrt{n}\max_{\substack{1 \leq i \leq p\\ |\theta_i|\geq \Delta-\gamma}}\text{sign}(U_i)(U_i-\theta_i)+o_\PR(1)~
\end{align}
\citep[see][]{patrick_annals}. 
Results on high-dimensional Gaussian approximation then yield that the right hand side of \eqref{hd100} can, in a suitable sense, be approximated by
\[
    \sqrt{n}\max_{\substack{1 \leq i \leq p\\ |\theta_i|\geq \Delta-\gamma}}\text{sign}(U_i)Z_i
\]
for a certain Gaussian vector $Z\in \R^p$ that has the same covariance structure as the vector $\big (\text{sign}(U_i)(U_i-\theta_i) \big )_{1 \leq i \leq p;  |\theta_i|\geq \Delta-\gamma} $. Assuming that $\|h\|_\infty\leq L_\infty$ for some $L_\infty\in (0,+\infty)$, we further use  the Bhatia-Davis inequality \citep{Bhatia:Davis:2000} to observe that under $H_0(\Delta)$ it holds that
\[
    \mathbb{V}ar(Z_i)\simeq\mathbb{V}ar(U_i)\leq L_\infty^2-(\Delta-\gamma)^2 ~,
\]
for all indices $i \in \{ 1, \ldots , p\} $ with  $|\theta_i|\geq \Delta-\gamma$. Therefore, we obtain
\[   \sqrt{n}\max_{\substack{1 \leq i \leq p\\ |\theta_i|\geq \Delta-\gamma}}\text{sign}(U_i)Z_i\leq \sqrt{n}\max_{1 \leq i \leq p}\tilde Z_i ~, 
\]
whenever the left hand side quantity is positive. Here the random variables $\tilde Z_i$ are defined by
\begin{align*}
    \tilde Z_i=\begin{cases}
       \text{sign}(U_i) Z_i\frac{\sqrt{L_\infty^2-(\Delta-\gamma)^2}}{\sqrt{\mathbb{V}ar(Z_i)}} \quad& |\theta_i|\geq \Delta-\gamma \\
        Y_i \quad & |\theta_i|< \Delta-\gamma~,
    \end{cases}~.
\end{align*}
where $(Y_i)_{1 \leq i\leq p}$ is a sequence of iid normal distributions with variance $1-(\Delta-\gamma)^2$  that is independent from the data. We may then use the assumption of non-negative sign adjusted pairwise correlations, the consistency of the sign estimators and Slepian's Lemma \citep{Slepian1962} to upper bound this quantity (in distribution, i.e. first order stochastic dominance)  by
\[
   \sqrt{n}\max_{1 \leq i \leq p} Y_i 
\]
which converges, appropriately rescaled, to a Gumbel distribution with scale parameter $\sqrt{L_\infty^2-(\Delta-\gamma)^2}$.
We can then proceed by using the associated Gumbel quantiles for our test decision whenever Algorithm \ref{alg:topk} outputs $\bot$. 

\begin{algorithm}[H]
  \caption{P-REL: Adaptive private estimation of the extremal set $\mathcal E$}
  \label{alg:topk}
  \begin{algorithmic}[1]
    \Require Ordered $|U|_{(1)},\hdots, |U|_{(p)}$, appr. zCDP budget parameters $\delta,\rho$ and threshold $t:=4 r/n L_\infty$.
    \Ensure Estimated extremal set $\hatEDP$ ~.
    \Function{\textnormal{P-REL}}{$U$, $\delta$, $\rho$, $t$}
     \State Set $q_j = |U|_{(j)}  - |U|_{(j+1)}$ for $j=1,\hdots, p-1$.
    \State Obtain $\hat k$ by invoking Algorithm \ref{alg:rep_noisy_max} with $q=(q_1, \ldots , q_{p-1})$ and $\ve = 2\sqrt{\rho}$.
    \State Set $\sigma = t/\sqrt{\rho}$ and construct a high–probability lower bound \;
           $ \hat{q}_{\hat k} = q_{\hat{k}}+ \mathcal{N}(0,\sigma^{2})- \sigma z_{1-\delta}$ for $q_{\hat k}$.
    \If{$\hat{q}_{\hat k} > t$}
        \If{$\hat k\leq \log(p)$} 
      \State \Return $\hat i_{1},\dots,\hat i_{\hat k}$.
                \Else
      \State \Return Randomly draw $\log(p)$ indices from $\hat i_{1},\dots,\hat i_{\hat k}$.
      \EndIf
    \Else
      \State \Return $\bot$.
    \EndIf
	\EndFunction

  \end{algorithmic}
\end{algorithm}
\begin{algorithm}[H]
  \caption{\textproc{P-HD-U-TEST}: Private High-Dimensional U-statistic test}
  \label{alg:HD_test}
  \begin{algorithmic}[1]
    \Require Data $X$, privacy $\delta,\rho$, $\alpha$, $\ell_2$ sensitivity $\Delta_2 \hat \zeta_1 (k)$, Bootstrap iterations $B$, Gumbel parameter $\gamma$.
     \Ensure reject or fail to reject $H_0$ in \eqref{eq_relevant_hypotheses} and output $\{\hat i_{1},\hdots, \hat i_{\hat k}\},\norm{U}_\infty^{\text{DP}} $ and $\hat\zeta_1^{\text{DP}}$.
     \Function{P-HD-U-TEST}{$X$, $\rho$, $\delta$, $\Delta$,$\alpha$, $\Delta_2 \hat \zeta_1$, $B$, $\gamma$}
      \State Compute U-statistics $|U|:=(|U|_1,\hdots, |U|_p)^\top$.
      \State Run $\hat{\mathcal E}^{\text{DP}}=$\textproc{P-REL}$(|U|,\delta,\rho/3,t=4r/nL_\infty)$.\Comment{Algorithm \ref{alg:topk}}
      \If{$\hat{\mathcal E}^{\text{DP}}=\{\hat i_{1},\hdots, \hat i_{\hat k}\}$ for some $\hat k\leq \log(p)$}
       \State Compute $\hat \zeta_1(\hat k)$ and $\hat S=\big(\text{sign}(U_iU_j)\big)_{i,j}$.
       \State Obtain bootstrap quantile
       $\hat q_{1-\alpha}^*=\textproc{HQU}(\hat \zeta_1(\hat k),\hat S, n, B, 2r/nL_\infty, \Delta_2\hat\zeta_1(k), \rho/3)$. \Comment{Algorithm \ref{alg_monte_carlo_quantile_highdim}}
       \State Compute private estimator $\norm{U}_\infty^{\text{DP}}=\norm{U}_\infty+Z$ with $\rho/3$. \Comment{Gaussian mechanism}
      \State Derive test decision $\text{dec}=\mathbbm{1}\{\norm{U}_\infty^{\text{DP}}\geq\hat q_{1-\alpha}^*+\Delta\}$.
        \If{$\text{dec}=1$}
      \State \Return Reject $H_0$ and output $\{\hat i_{1},\hdots, \hat i_{\hat k}\},\norm{U}_\infty^{\text{DP}} $ and $\hat\zeta_1^{\text{DP}}$.
      \Else
      \State \Return Fail to reject $H_0$ and output $\{\hat i_{1},\hdots, \hat i_{\hat k}\},\norm{U}_\infty^{\text{DP}} $ and $\hat\zeta_1^{\text{DP}}$.
        \EndIf
      \Else
      \State \Return \textproc{P-GUMBEL-TEST}($|U|$, $2\rho/3$, $n$, $p$,$\gamma$)~. \Comment{Algorithm \ref{alg:Gumbel}}
      \EndIf
\EndFunction
  \end{algorithmic}
\end{algorithm}

We emphasize that in practice, although  this method yields better results than the simple concentration  bound approach, our empirical results in Section \ref{sec:experiments} demonstrate, that it is still far from being comparable to Algorithm \ref{alg:topk} when separation is in fact possible. Thus we only recommend its stand-alone application in the case where it is clear that a separation is not possible.

The resulting procedure is summarized in Algorithm \ref{alg:HD_test}, which  defines a test for the relevant hypotheses \eqref{eq_relevant_hypotheses}.  In the following section we prove that this test is a valid procedure from an asymptotic point of view.

\subsection{Statistical guarantees}\label{subsec:asymptotic_theory}

We make the following  assumptions.
\begin{assumption}\label{As_Sparse}$ $
\begin{itemize}    
    \item[(P)] There exists a set $\mathcal{B}$ with $\PR(\mathcal{B})=1-o(1)$ such that for any $l, j_1,...,j_l$ the equality
    \begin{align*}
        &\PR( U_{j_1}  \geq  t_{1},   \ldots , U_{j_l}  \geq  t_{l}  ) = \PR( \tilde U_{1} \geq  t_{1} , \ldots  ,  U_{{\hat k}}  \geq  t_{\hat k}  
        |\hat i_1=j_1,...,\hat i_{\hat k}=j_l, \hat k=l)
    \end{align*}
    
  holds on $\mathcal{B}$, where   $\tilde U:=(\tilde U_1,\hdots,\tilde U_{\hat k})= (U_{\hat i_1},\hdots,U_{\hat i_{\hat k}})^\top$ is the sub-vector of $U$  defined by   the output of Algorithm \ref{alg:topk}.
    \item[(V)]There exist constants  $\underline{b}>0$ and $c \in (0,\Delta)$ such that 
	$
	\min_{1 \leq i \leq p, |\theta_i|>c}\zeta_{1,ii}>\underline{b}
	$
	 for all  
     $p=p(n), n \in \mathbb{N}$, where $\zeta_{1,ii}$ is defined in \eqref{eq_var_def}. Here and in the following,  a minimum over the empty set is defined as  $+\infty$.
     \item[(B)] The components of the kernel $h=(h_1 , \ldots , h_p)^\top $ are bounded in absolute value by a constant
 $L_\infty>0$.
     \item[(E)] There exist a constant $\gamma>0$ that such that the  covariances $\zeta_{1,ij}$ defined in \eqref{eq_var_def} satisfy
     \[
        \min _{{1 \leq i<j\leq p~, |\theta_i|\land |\theta_j|\geq \Delta-\gamma}}\text{sign}(\theta_i\theta_j)\zeta_{1,ij}\geq 0
     \]
     
\end{itemize}
\end{assumption}

Assumptions   (V) and (B) are fairly mild and standard in the context of high-dimensional dependence testing via $U$-statistics \citep[see, for example,][]{drttonetal2020,patrick_annals}. Assumption (P) is a technical condition and needed to show that the test keeps its level at the boundary of the hypotheses in \eqref{eq_relevant_hypotheses}, that is $\| \theta \|_{\infty} = \Delta$. It ensures that whether or not and where a gap is detected has negligible impact on the distribution of the vector $U$. It is fulfilled in a variety of scenarios. Examples include situations where no gap can be detected or when there exists exactly one $k$ such that
    \begin{align}
    \label{eq:gap}
          |\theta|_{(k)}-|\theta|_{(k+1)}\geq \max_{j \neq k}|\theta|_{(j)}-|\theta|_{(j+1)}+\sqrt{\log(n)\log(p \lor n)/n}~,
    \end{align}  
i.e. there is a largest gap that can be separated from all other gaps (here $|\theta|_{(1)}\geq \ldots \geq |\theta|_{(p)}$ denote the ordered values of $|\theta|_{1} , \ldots , |\theta|_{p}$). In particular, we prove in  Lemma \ref{Lem:uti_rnm} of the online supplement,  that $\hat k=k$ with high probability in that scenario. Assumption (E) is required for using the extreme value approach instead of the concentration based approach when Algorithm \ref{alg:topk} outputs $\bot$. It can also be weakened to allow up to $q =o(p)$ negative $\zeta_{1,ij}$.
\smallskip

Under the above assumptions we may now construct a test procedure based on the output of Algorithm \ref{alg:HD_test}. We summarize it together with its asymptotic properties in  Theorem \ref{thm:consistency_Ustats_High} below.

\begin{theorem}\label{thm:consistency_Ustats_High}
Let $\log(p)=o(n^{1/5})$, assume that Assumptions (V), (E) and (B) hold.   The test decision $\phi$ of Algorithm \ref{alg:HD_test} defines a consistent and asymptotic level-$\alpha$ test for hypotheses \eqref{eq_rel_hyp_high}. More precisely, 
\begin{enumerate}
    \item[(1)] If  $\|\theta\|_\infty<\Delta-\gamma$ holds for some $\gamma>0$, we have
 $   
       \lim\limits_{n \to \infty} \mathbb{P}_{\theta}(\phi=1) 
      = 0~. 
$ 
    \item[(2)] If  $\|\theta\|_\infty=\Delta$ and Assumption  (P) holds, we have 
$       \lim\limits _{n \to \infty} \mathbb{P}_{\theta}(\phi=1) 
      \leq \alpha~.   
$
    \item[(3)] Assume that $\|\theta\|_\infty>\Delta+\gamma\sqrt{\log(p)/n}$ for sufficiently large $\gamma>0$ and either
    \begin{itemize}
        \item [a)] Algorithm \ref{alg:topk} outputs $\bot$ with high probability or,
        \item [b)] \eqref{eq:gap} is satisfied for some $k \leq \log(p)$~,
    
        \end{itemize}
         holds, then 
$    
       \lim\limits_{n \to \infty} \mathbb{P}_{\theta}(\phi=1) 
      = 1~.
$
\end{enumerate}
Moreover, Assumption (E) may be dropped when using a concentration based  instead of the extreme value based approach in the case where Algorithm \ref{alg:topk} returns $\bot$.
\end{theorem}
Our next result provides privacy guarantees for Algorithm \ref{alg:HD_test}.

\begin{theorem}\label{thm:privacy_test}
    Algorithm \ref{alg:HD_test} satisfies $\delta-$approximate-$\rho$-zCDP.
\end{theorem}

\begin{remark}\label{rem:assumptions_discuss}
~~
{\rm \begin{itemize}
\item[1)] The results can be extended in a straightforward manner to statistics $T$ that are approximated by a $U$-statistic in the sense that
$   T=U+R_n, $
where $R_n \in \R^p$ is a remainder that converges to $0$  sufficiently fast in the maximum-norm, that is $\| R_n \|_\infty = o_\PR ( \sqrt{\log(p)}/n ) $. Prominent examples, where this is possible, are V-statistics, certain functionals of the Kaplan-Meier estimator \citep[see][]{Gijbels:Veraverbeke:1991} and  Spearman's $\rho$.
\item[2)] In Algorithm~\ref{alg:HD_test}, we allocated the privacy budget equally across all sub-procedures. Our simulations indicate that, when a largest gap of size $O(1/n)$ is present, correctly identifying it is the key step of the methodology. Therefore, one might consider allocating a larger portion of the privacy budget~$\rho$ to this step.

\item[3)]The methodology can be extended to $U$-statistics with unbounded kernels by appropriate truncation methods such as discussed in \cite{NEURIPS2024_290848c8}.
\end{itemize}
}
\end{remark}

\begin{remark}
 {  \rm 
 A careful  inspection of the proof of Theorem \ref{thm:consistency_Ustats_High} shows that 
 there exists parameters $\theta \in \mathbb{R}^p$ with $\| \theta\|_\infty = \Delta $ such that there is equality in part (2). For a prominent  example,  consider the case  where the components of $\theta$ satisfy for some $\gamma>0$ 
\begin{align*}
|\theta_i| \quad 
\begin{cases}
=\Delta & (1 \leq i \leq s), \\[6pt]
\leq \tfrac{\Delta}{2+\gamma} & (s < i \leq p).
\end{cases}
\end{align*}
}
\end{remark}

\begin{remark}
    {\rm If the null hypothesis in \eqref{eq_rel_hyp_high} is rejected, the next step in the statistical analysis  is to privately identify the relevant coordinates, i.e. the set
$$
\mathcal R:=\left\{i\in \{1,\hdots, p\} \,\mid \, |\theta_{i}|>\Delta\right\}~.
$$
Although this is not the main objective of this paper, we briefly discuss a first solution of this problem.
For this purpose we use Algorithm \ref{alg:topk} with  the queries
$$
q_i=\max\{|U|_{(i)}, \Delta\}-\max\{|U|_{(i+1)},\Delta\}~~~~~~(i=1, \ldots , p-1),
$$
and denote the output by $\big (\hat k_{\cal R}, \hat {\cal R}^{\rm DP} \big ) $. 
With this adjustment, the noisy max  Algorithm \ref{alg:rep_noisy_max} will output an index $\hat k_{\cal R} $ 
corresponding to a component of $\theta$ with $|\theta|_{(\hat k_{\cal  R} )}>  \Delta$ with high  probability, since all gaps corresponding to components  with   $| \theta_i | \leq \Delta$ are set to zero (asymptotically). We hence obtain that either $\hat{\mathcal{R}}^{\rm DP}\subset \mathcal R$ or $\hat{\mathcal{R}}^{\rm DP} = \mathcal R$. In the former case it is of course possible to recover the missing coordinates via an offline SVT (exponential mechanism) until an index exceeding $\Delta$ is found.
This modification, however, introduces a trade-off. Suppose the largest gap occurs at an index $\hat k_{\cal R} $  separating the set $\mathcal R$ of relevant coordinates  and the set ${\cal R}^c$ of non-relevant coordinates, that is $|\theta |_{(\hat k_{\cal R})} > \Delta$, and $|\theta |_{(\hat k_{\cal R})} \leq \Delta$. In this case, the gap shrinks from its original size $|U|_{(\hat k_{\cal  R})}-|U|_{(\hat k_{\cal  R}+1)}$ to $|U|_{(\hat k_{\cal  R})}-\Delta$. Consequently, the probability of correctly identifying the gap decreases. To obtain the same detection accuracy one therefore has to allocate a larger portion of the privacy budget to Algorithm \ref{alg:topk}.
}
\end{remark}

\section{Finite sample properties}\label{sec:experiments}

  \def\theequation{5.\arabic{equation}}	
	\setcounter{equation}{0}

In this section we investigate the finite sample properties  of the proposed procedure by means of a simulation study and illustrate its application in two data examples. Throughout this section we concentrate   on the problem of detecting relevant dependencies as discussed in the introduction and in Example \ref{ex1}.

\subsection{Simulation study }
\label{simul}
We  first illustrate the performance of the set estimation in Algorithm \ref{alg:rep_noisy_max} in different scenarios. Later, we will also further investigate how these properties propagate into the testing procedure. 
\smallskip

\textbf{Experimental Setup:} 
We  generate data from 
a $d$ dimensional normal distribution with covariance $\Gamma$ to be specified later, i.e.
\begin{align}
    \label{sim1}    X_1,...,X_n \sim \mathcal{N}_d(0,\Gamma)
\end{align}
and consider as measure of (monotone) dependence the classical Kendall's $\tau$ given by
\[
    \tau=(\tau_{ij})_{1 \leq i<j \leq d}=\Big(\E[\text{sign}(X_{1i}-X_{2i})\text{sign}(X_{1j}-X_{2j})]\Big)_{1 \leq i<j\leq p}~.
\]
An unbiased estimator of $\tau_{ij}$ is given by an  $U$-statistic of degree $r=2$
\[
\hat \tau_{ij}=  
\hat \tau_{ij} (X) =  
\frac{2}{n(n-1)}\sum_{1\leq k<l\leq n}\text{sign}(X_{ki}-X_{li})\, \text{sign}(X_{kj}-X_{lj})~,
\]
with corresponding kernel 
$$
h_{ij}(x_1,x_2)=   \tilde h ( x_{1i},x_{1j},x_{2i},x_{2j}) = 
\text{sign}(x_{1i}-x_{2i})\text{sign}(x_{1j}-x_{2j})~.
$$
The vector of $U$-statistics  is then defined by
$U=\text{vech} \big ( ( \hat{\tau}_{ij} )_{1 \leq  i < j \leq  p}  \big )$  with sensitivity 
\begin{equation*}
    \big | \| \text{vech} (\hat\tau (X) \|_\infty  - \| \text{vech} (\hat\tau (X'))\|_\infty \big |\leq  \frac{4}{n}~,
\end{equation*}
i.e. $L_\infty =1$. Therefore, by Lemma \ref{Lem_steinke_gauss}, we have
\[
    \norm{\hat \tau}_\infty^{\text{DP}}=\norm{\hat \tau}_\infty + Y
\]
with $Y\sim \mathcal N(0,\frac{8}{n^2\rho})$ yields a $\rho-$zCDP private estimator of $\|\theta \|_\infty$.
\smallskip

\textbf{Simulation setup:}  As sample size we choose $n=250,500$ and $1000$, while the dimension $d $ is $\lceil \sqrt{2n}\rceil $ and $n$ corresponding to a moderate ($p \approx n$)  and high-dimensional ($p\approx n^2/2$) case, respectively. The nominal level for the test is chosen as $\alpha =0.05$ and $B=500$ replications are used to calculate the critical values. For the  parameter $\rho$ we choose $\rho = 0.1, 0.25$ and $1$ ranging from  strict to lax privacy. We repeated the experiments over $500$ simulation runs.
We consider Kendall's $\tau$  matrices $\mathcal T=(\tau_{ij})_{i,j=1, \ldots , d}$, and the covariance matrix $\Gamma = (\Gamma_{ij} )_{i,j=1, \ldots , d} $ in \eqref{sim1} is then obtained  by means of the formula  
$$
 \Gamma_{ij} = \sin \Big ( \frac{\pi }{2} 
     \tau_{ij} \Big )~,
     $$      
     which holds for all elliptical distributions with continuous margins. We now consider  two \textbf{F}avorable settings, for 
       which it is easy to check that condition  \eqref{eq:gap} is satisfied such that Theorem \ref{thm:consistency_Ustats_High}  is applicable. Here we expect a good performance of the test defined by Algorithm \ref{alg:HD_test}. Further simulation results can be found in Section \ref{appadditionalsim}. There we study the robustness of our approach  in two \textbf{U}nfavorable settings,  for which the performance of the test cannot be predicted by theory, and demonstrate that the test has a reasonable performance in such cases as well. The two favorable designs are defined as follows.
\begin{itemize}
    \item [\textbf{F1)}] A dense signal with $\norm{\text{vech}(\mathcal T)}_\infty=0.5$ on $p/2$ coordinates, i.e.
    \begin{equation}
    \label{taumatricesF1}
        \mathcal T=
        \mathbf{I}_d
        +0.5\sum_{1\leq i<j\leq \lfloor d/\sqrt{2}\rfloor} (e_ie_j^\top + e_je_i^\top)~,
        \end{equation}
    where $e_i \in \mathbb{R}^d$ denotes the $i$th unit vector and $\mathbf{I}_d$ is the $ d \times d$ identity matrix.
    \item [\textbf{F2)}] A sparse signal with $\norm{\text{vech}(\mathcal T)}_\infty=0.5$ on only three coordinates, i.e.
    \begin{equation}
    \label{taumatricesF2}
            \mathcal T=
            \mathbf{I}_d +0.5(e_1e_2^\top + e_2e_1^\top)+0.5(e_2e_3^\top + e_3e_2^\top)+0.5(e_1e_3^\top + e_3e_1^\top)
    \end{equation}
\end{itemize}

We will first  explain  how to read the figures. Note that we display the rejection probabilities for different values of the threshold $\Delta$, which is decreasing from the left to the right. The horizontal red dotted line corresponds to the nominal level $\alpha=0.05$, while the vertical red dotted line demarcates our choice of $\norm{\text{vech}(\mathcal T)}_\infty=0.5$. The values for the threshold  $\Delta$ in the hypotheses \eqref{eq_relevant_hypotheses} are displayed on the $x$-Axis (in decreasing order), such that the right part of the figures correspond to the alternative and the left to the null hypothesis. Consequently, it is desirable that the rejection curves converge to 1 in the top right quadrant and that they are close to $0$ in the bottom left quadrant, crossing precisely at the intersection of the two red lines.
\smallskip

\textbf{Comparison to Hoeffding Approach:}
To the best of our knowledge there exists no other differentially private method that can test hypotheses of the form \eqref{eq_relevant_hypotheses} in moderate or high-dimensional settings. To demonstrate the improvement of the test defined by Algorithm \ref{alg:HD_test}, we thus begin with a quick comparison of our method  implemented in Algorithm \ref{alg:HD_test} (see Figure \ref{fig:power_curves_grid:a}) to 
the concentration-based baseline we described in Section \ref{sec:hoeffding} (see Figure \ref{fig:power_curves_grid:b}) that already outperforms more naive procedures like the Bonferroni correction and related procedures. Across all simulated settings our procedure attains strong empirical power close or equal to one at a signal size for which the rejection rate of concentration-based method has not even exceeded the nominal level $\alpha$. 
\begin{figure}[ht]
\centering

\begin{minipage}{0.4\linewidth}
\centering
\includegraphics[width=\linewidth]{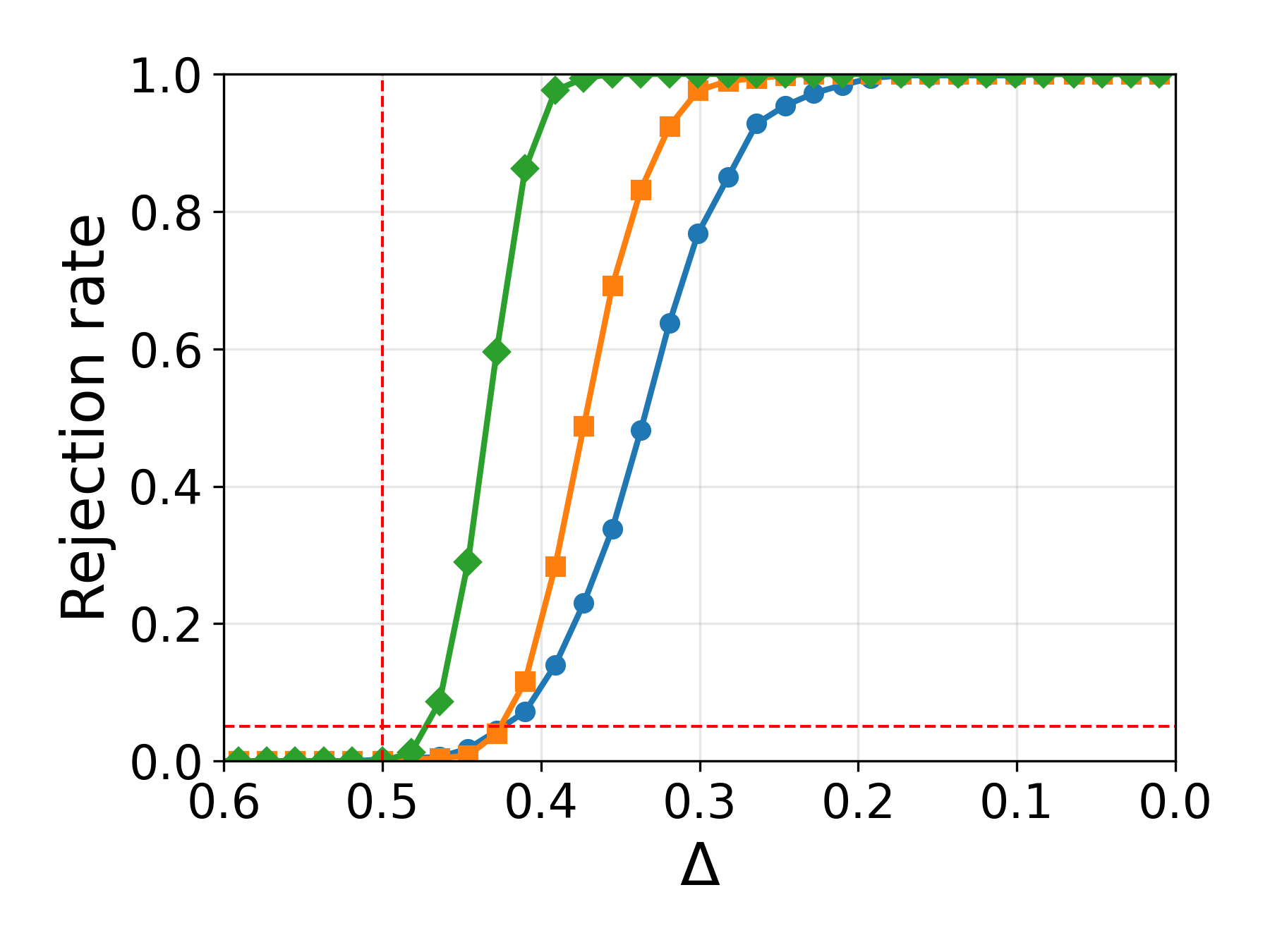}
\vspace{-0.6cm}
\subcaption{Algorithm \ref{alg:HD_test}}
\label{fig:power_curves_grid:a}
\end{minipage}
\hspace{0.03\linewidth}
\begin{minipage}{0.4\linewidth}
\centering
\includegraphics[width=\linewidth]{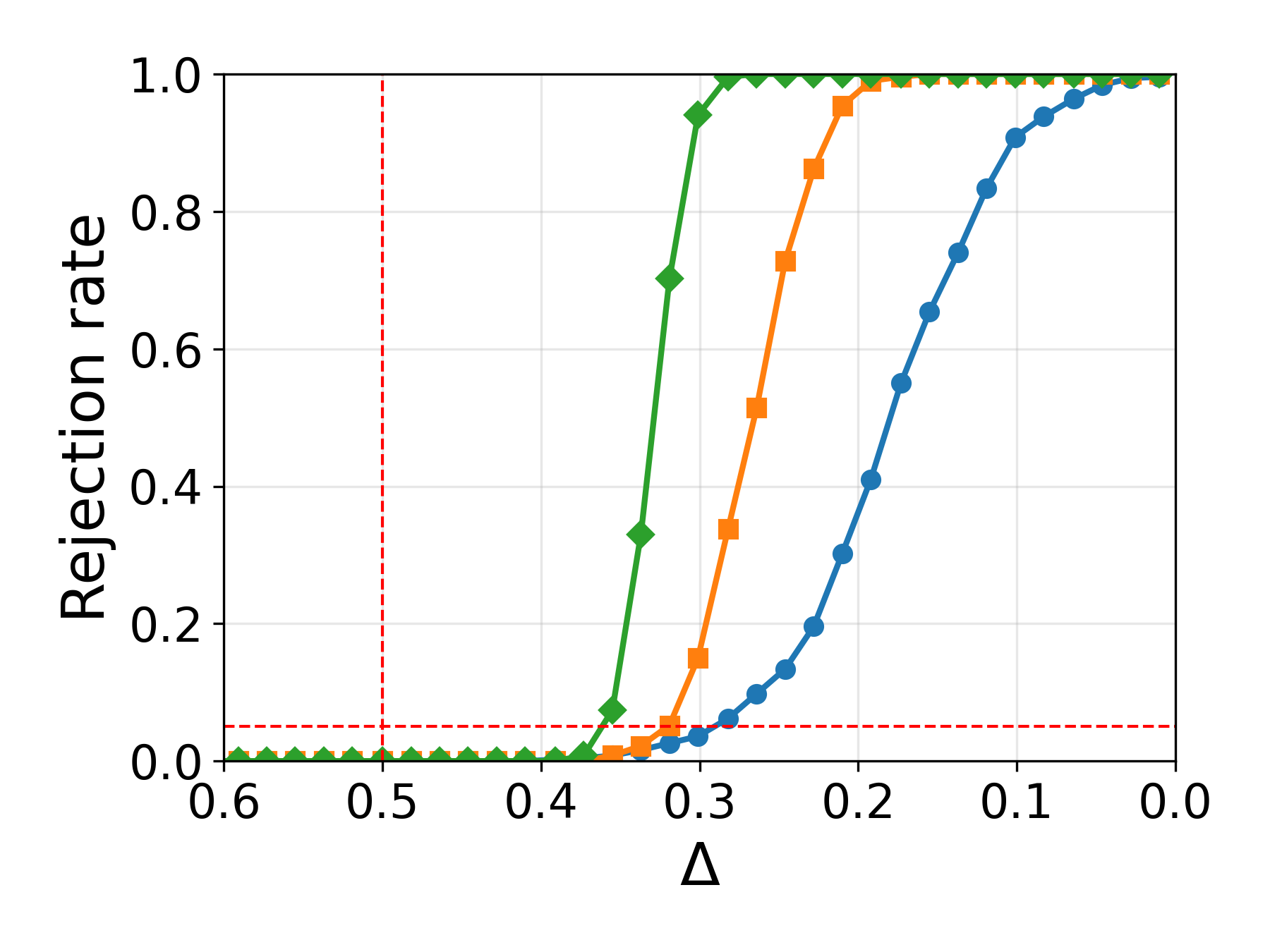}
\vspace{-0.6cm}
\subcaption{Hoeffding based test \eqref{eq:hoeffding_test}}
\label{fig:power_curves_grid:b}
\end{minipage}

\vspace{-0.2cm}

\makebox[\linewidth][c]{%
  \includegraphics[width=0.6\linewidth]{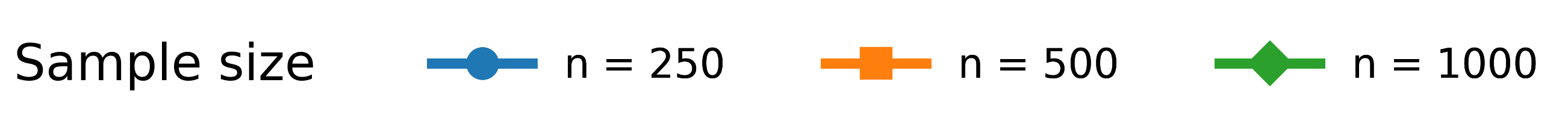}
}
\vspace{-0.6cm}
\caption{\it Empirical rejection probabilities of the test defined by Algorithm \ref{alg:HD_test} and the test based on a concentration inequality for privacy parameters $\rho=0.1$ in setting \textbf{F1)} with $n \in \{250,500,1000\}$ and $p=d(d-1)/2$ for $d=\lceil\sqrt{2n}\rceil$ (moderate dimensional regime).}
\label{fig:power_curves_grid_con}
\end{figure}

\textbf{Performance in the moderate dimensional regime:}
In Figure  \ref{fig:power_curves_grid} we display rejection probabilities of the test defined by  Algorithm \ref{alg:HD_test}  in models \textbf{F1})  and \textbf{F2}) for the moderate dimensional regime. 
We observe that the test keeps its size in all settings under consideration with a type I error $\alpha$ quickly decaying to 0 for $\Delta$ larger than $0.5$. For larger sample sizes - where the gap can be detected more reliably - we see that the test approximates the nominal level more accurate.
Regarding the rejection rate under the alternative we observe that the proposed method is able to detect both dense and sparse correlations well, with a detection boundary that quickly gets sharper as the sample size increases. We further observe that the privacy constraints impact the power most notably for the lower sample size $n=250$. Here the gap is detected less reliably in the strict privacy setting. As a consequence the test then defaults to the Gumbel approximation, leading to lower detection rates. 
\begin{figure}[H]
\centering

\begin{tabular}{@{}ccc@{}}
\includegraphics[width=0.3\linewidth]{figures/power_merged_eps0.1_d45_M2.png} &
\includegraphics[width=0.3\linewidth]{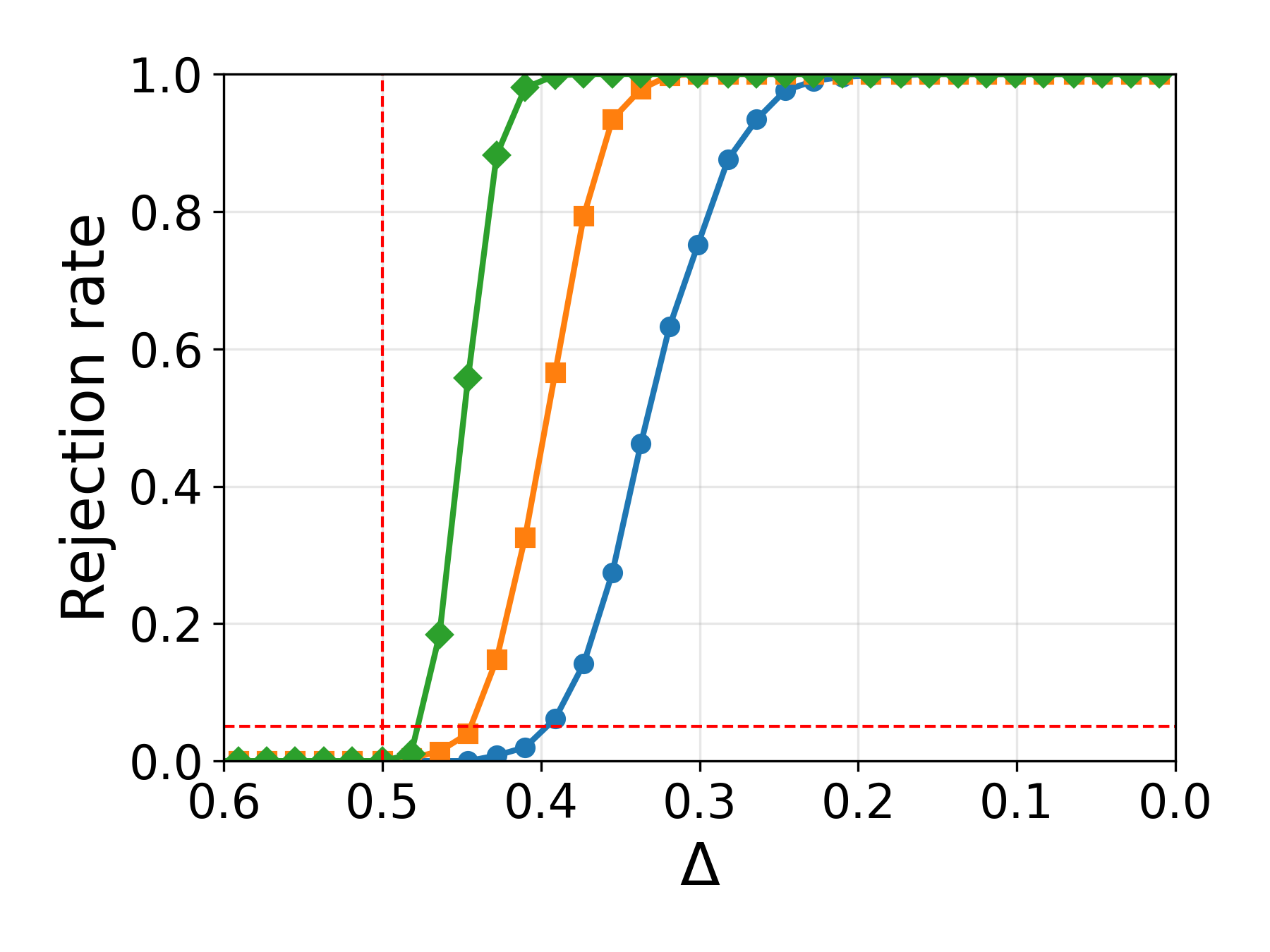} &
\includegraphics[width=0.3\linewidth]{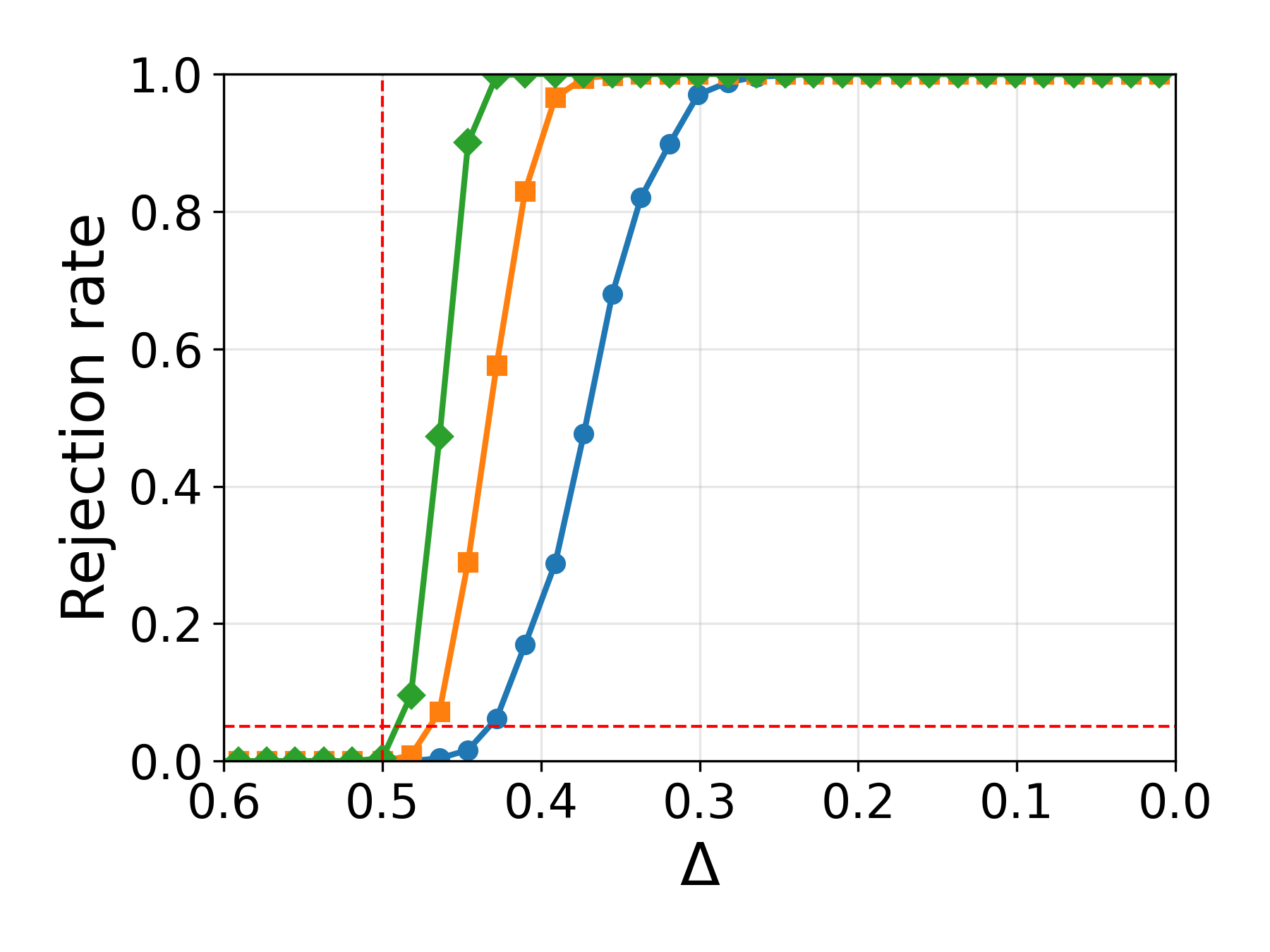} \\[-0.52cm]

\includegraphics[width=0.3\linewidth]{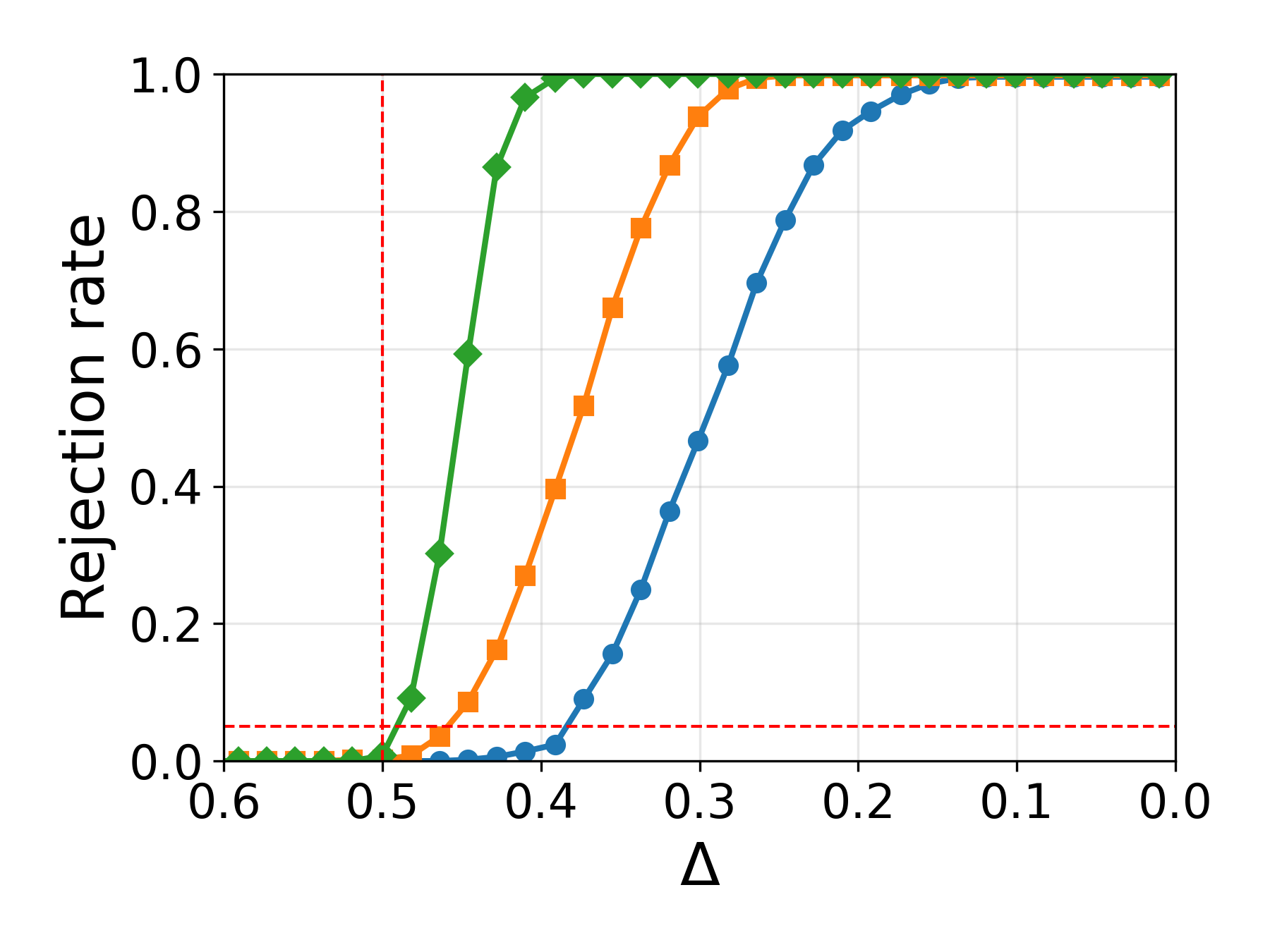} &
\includegraphics[width=0.3\linewidth]{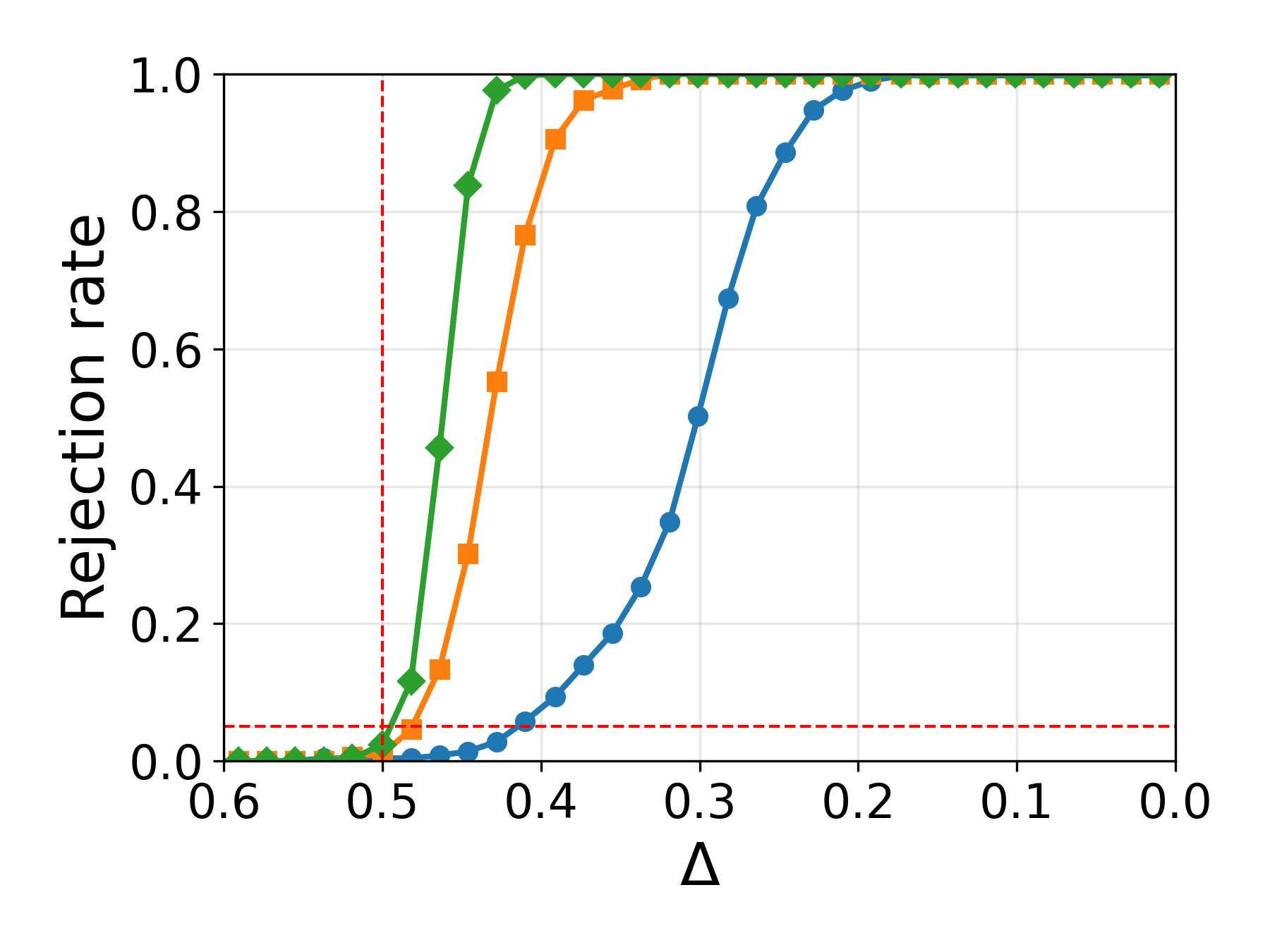} &
\includegraphics[width=0.3\linewidth]{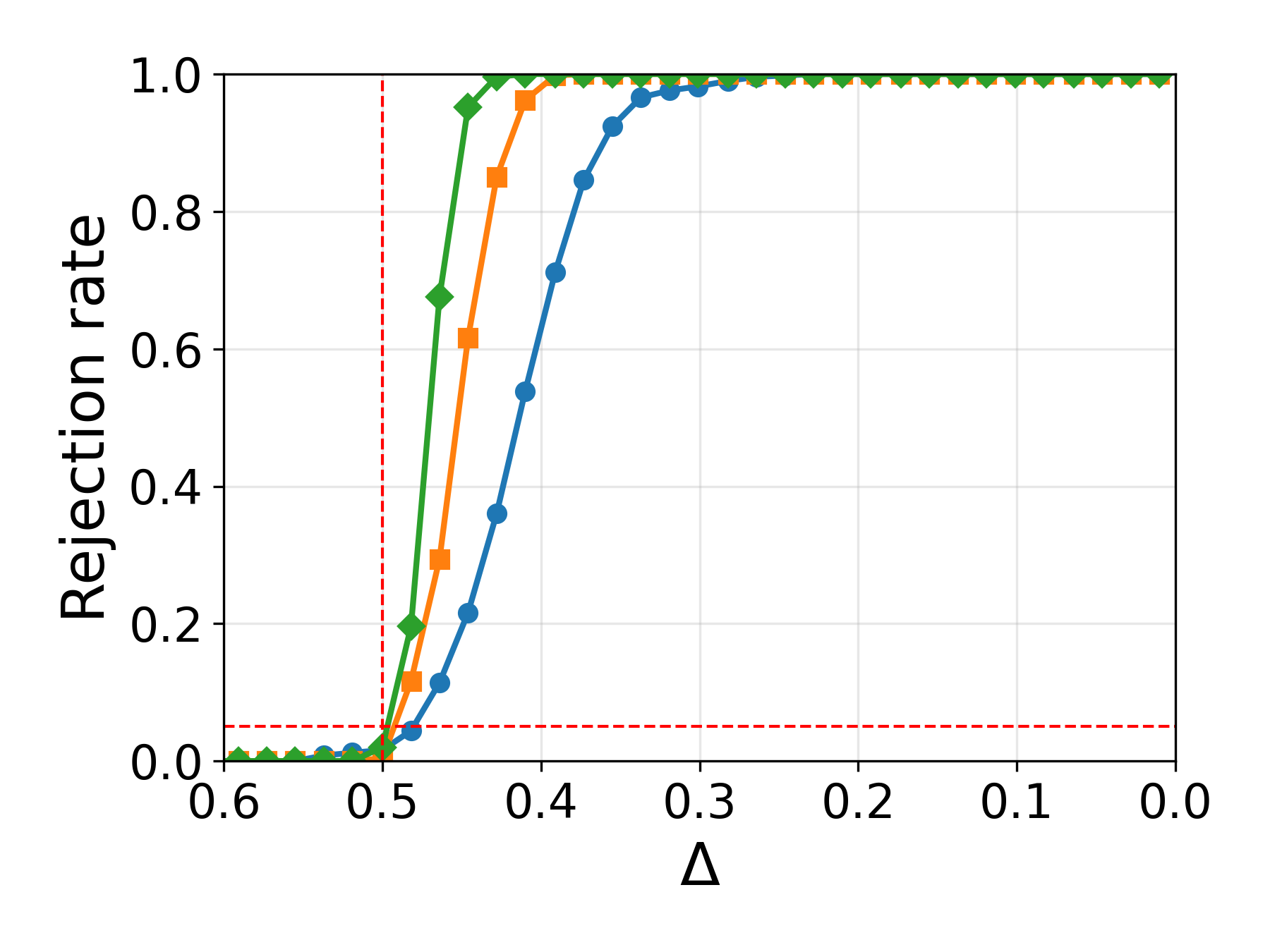}
\end{tabular}

\vspace{-0.3cm}

\makebox[\linewidth][c]{%
  \includegraphics[width=0.6\linewidth]{figures/legend_sample_size_left.png}
}
\vspace{-0.6cm}
\caption{\it Empirical rejection probabilities of the test defined by Algorithm \ref{alg:HD_test} for different privacy parameters $\rho=0.1,0.25,1$ and models \textbf{F1)} (first row) and \textbf{F2)} (second row) with $n \in \{250,500,1000\}$, $p=d(d-1)/2\approx n$ with $d=\lceil\sqrt{2n}\rceil$ (moderate dimensional regime).}
\label{fig:power_curves_grid}
\end{figure}
\begin{figure}[H]
\centering

\begin{tabular}{@{}ccc@{}}
\includegraphics[width=0.3\linewidth]{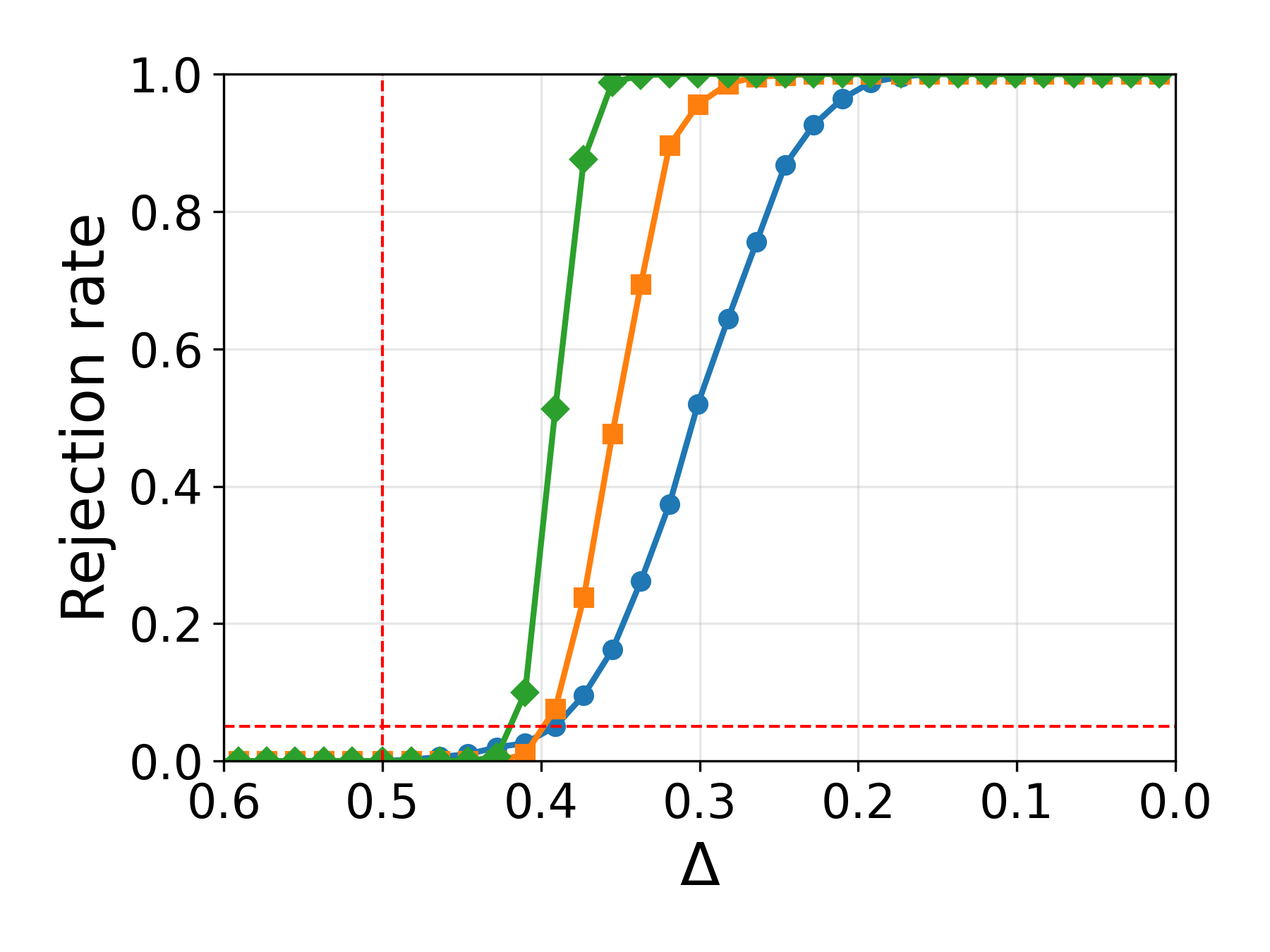} &
\includegraphics[width=0.3\linewidth]{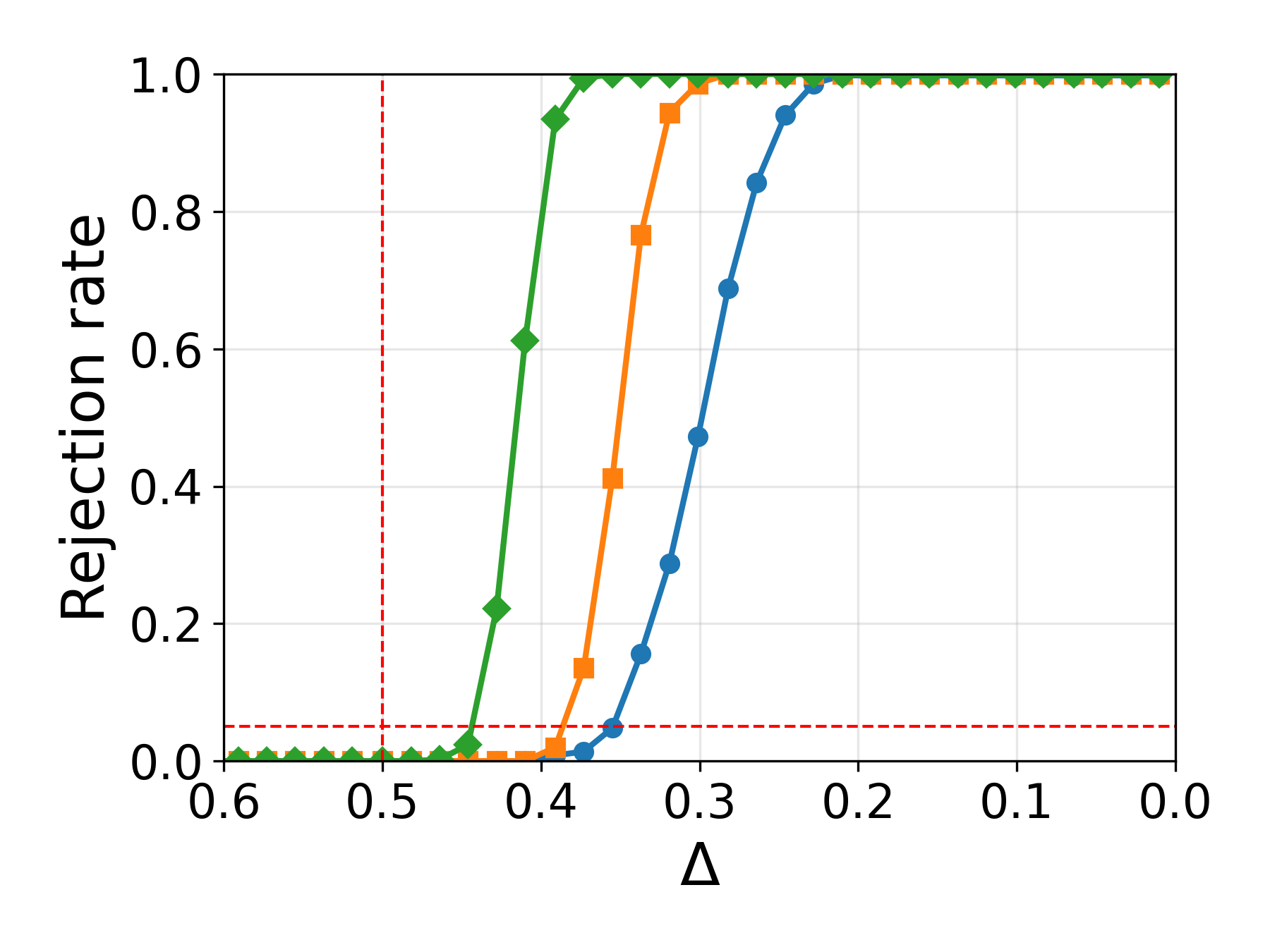} &
\includegraphics[width=0.3\linewidth]{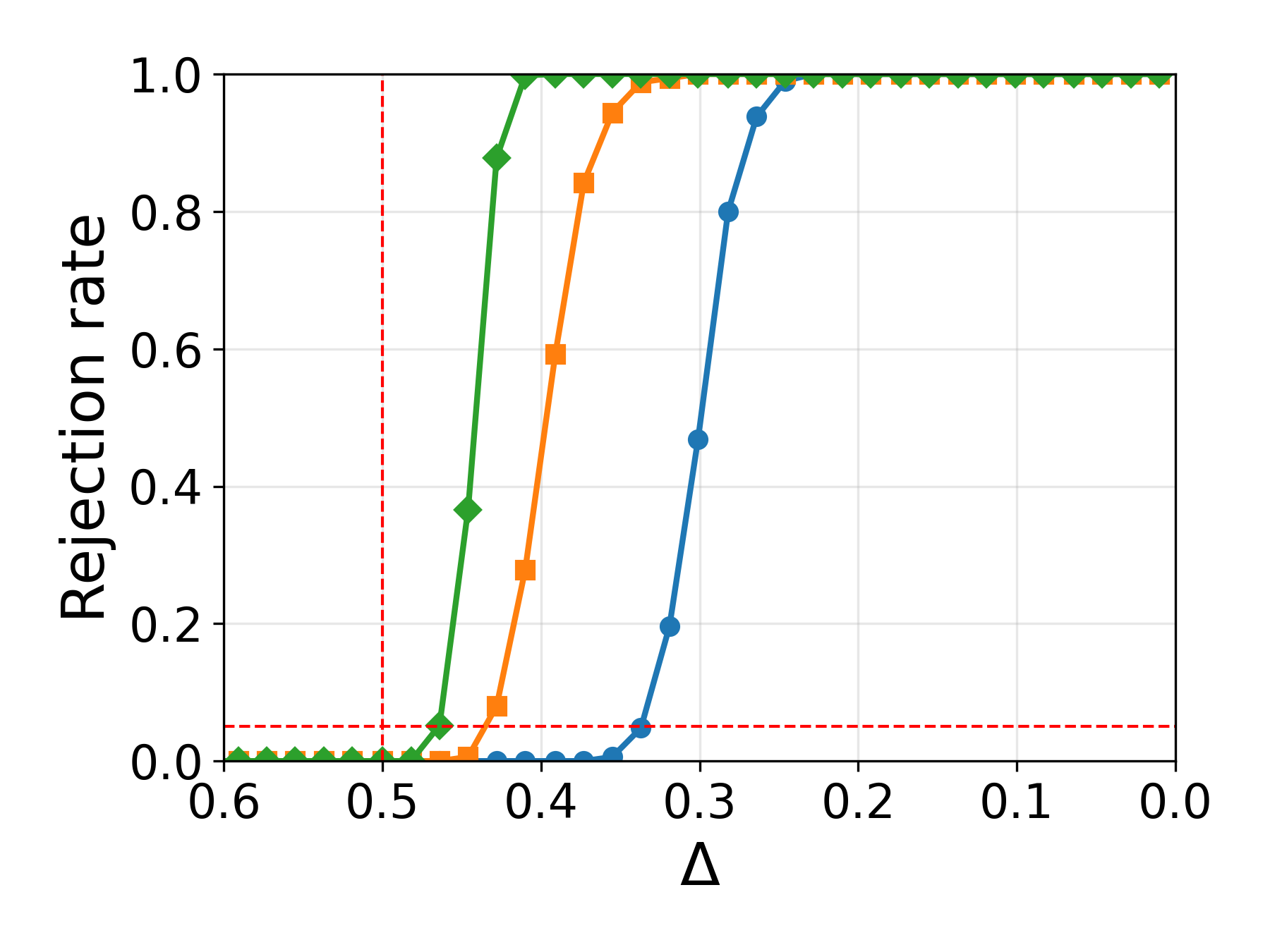} \\[-0.52cm]

\includegraphics[width=0.3\linewidth]{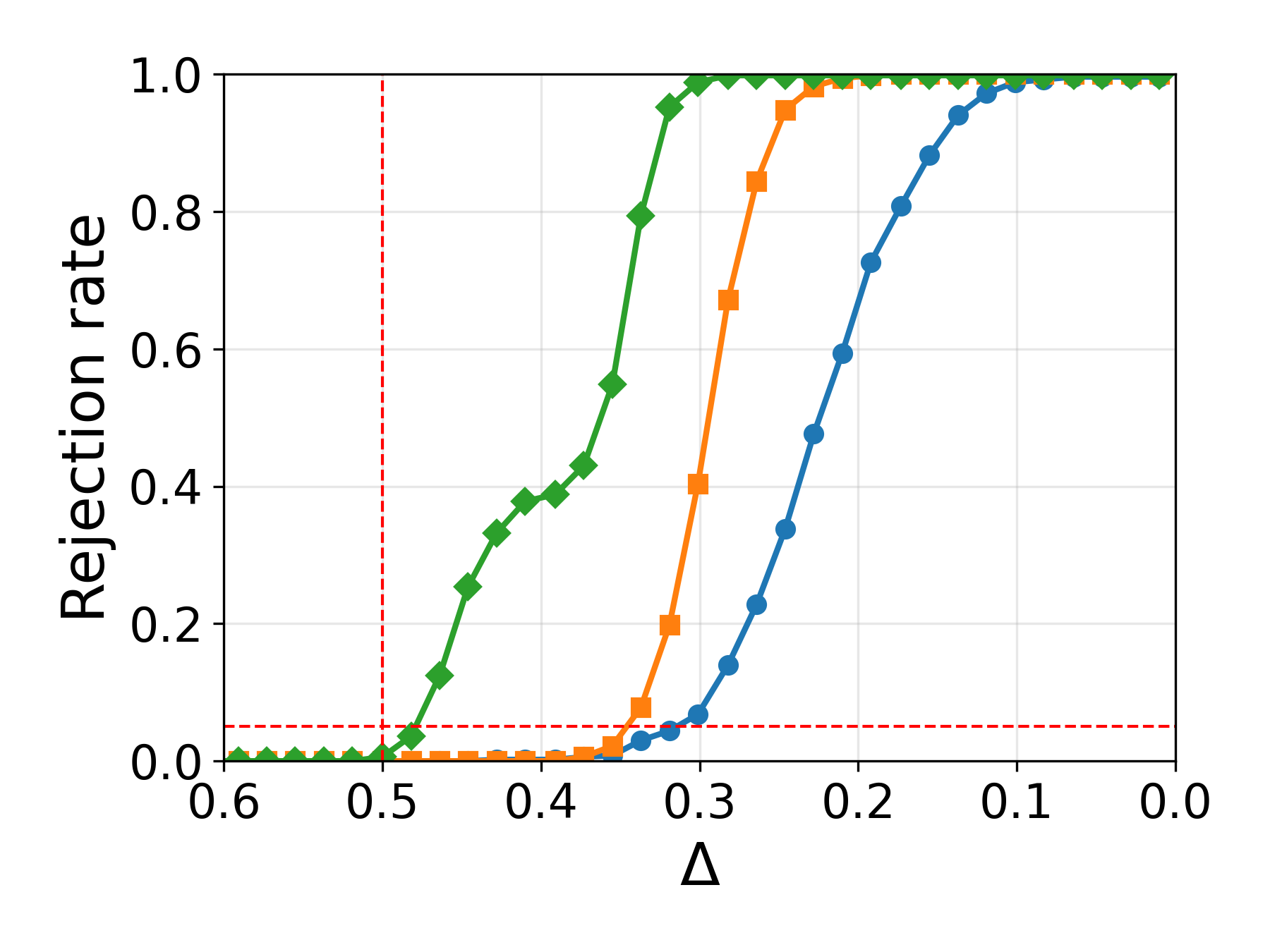} &
\includegraphics[width=0.3\linewidth]{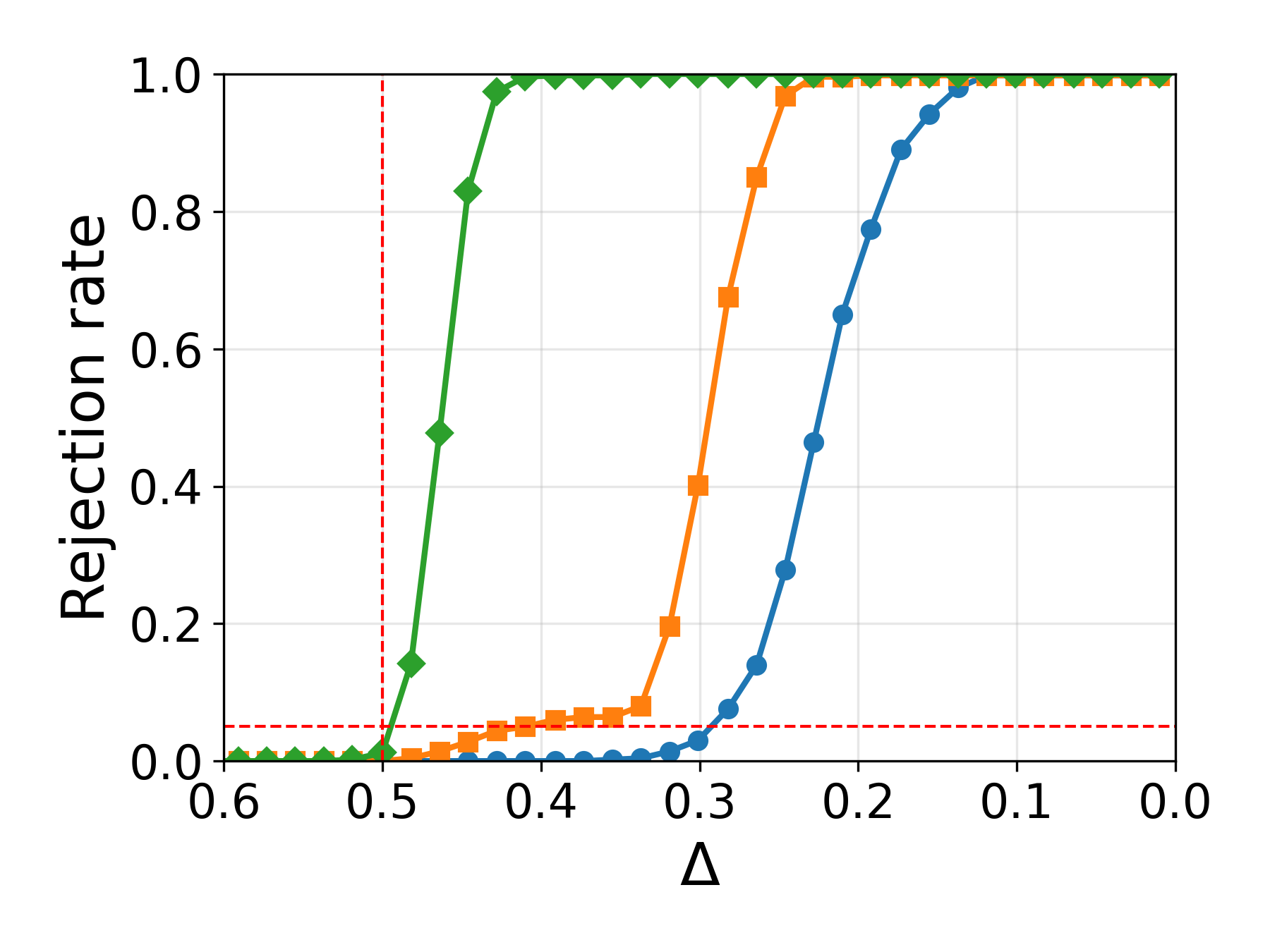} &
\includegraphics[width=0.3\linewidth]{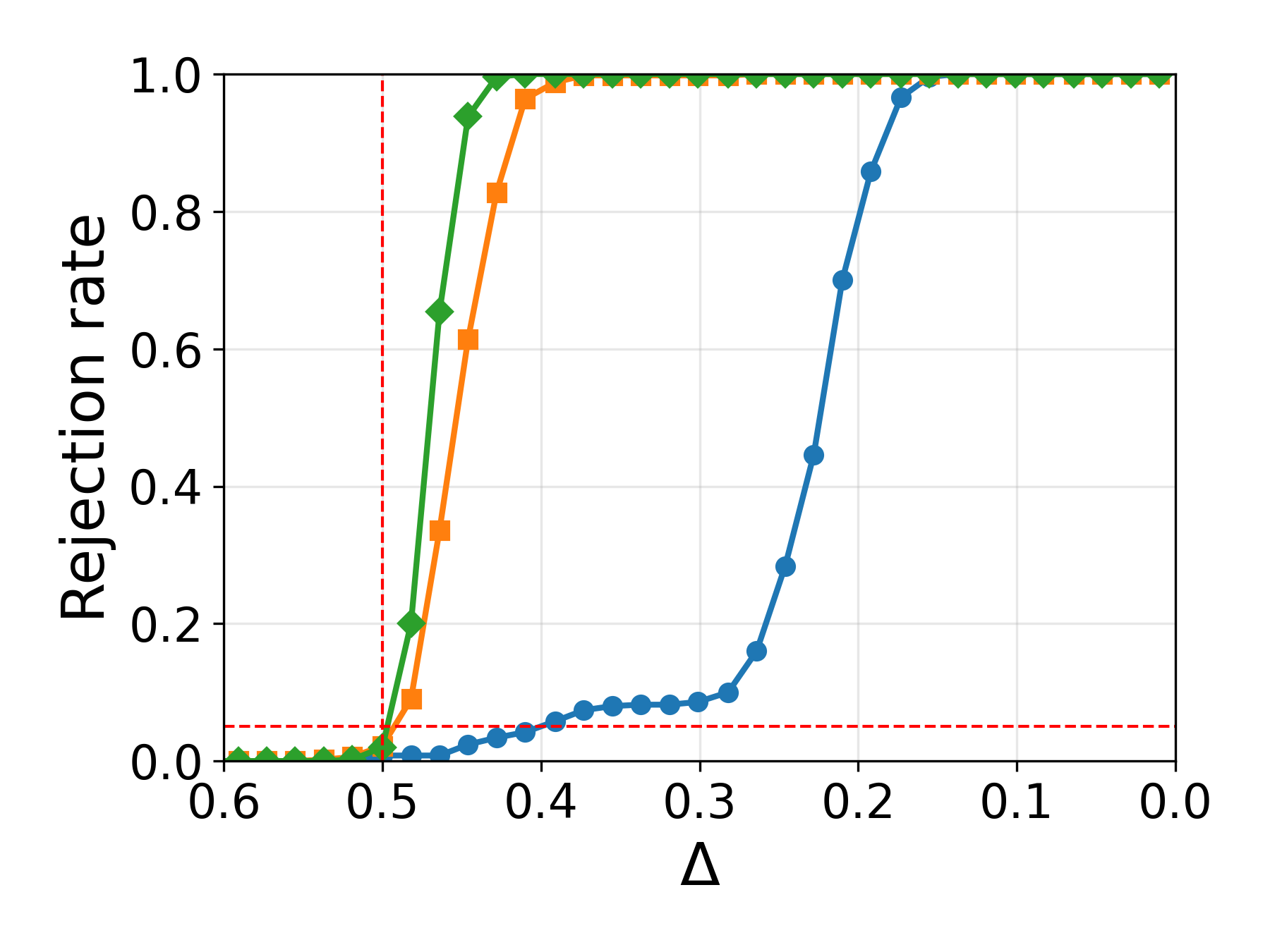}
\end{tabular}

\vspace{-0.3cm}

\makebox[\linewidth][c]{%
  \includegraphics[width=0.6\linewidth]{figures/legend_sample_size_left.png}
}
\vspace{-0.6cm}
\caption{\it Empirical rejection probabilities of the test defined by Algorithm \ref{alg:HD_test} for different privacy parameters $\rho=0.1,0.25,1$ and models \textbf{F1)} (first row) and \textbf{F2)} (second row) with $n \in \{250,500,1000\}$, $p=d(d-1)/2$ with $d=n$ (high-dimensional regime).}
\label{fig:power_curves_grid_high}
\end{figure}
\textbf{Performance in the high-dimensional regime:}
In Figure \ref{fig:power_curves_grid_high} we present corresponding results in the high-dimensional regime. A comparison with Figure \ref{fig:power_curves_grid} shows only a minor loss in power, which is also  predicted by our results. The main driver of this loss is that in the high-dimensional regime, the detection of the gap is more difficult due to the larger amount of coordinates present, as this yields a smaller observable gap with high probability. Whenever the gap can be reliably detected (typically for sample size $n=500$ in low/moderate privacy regimes or for $n=1000$ for high privacy) the behavior of the test is independent of the ambient dimension.

\subsection{Real-World Examples}\label{subsec:real_world_examples} 
~~\\
\textbf{Monotone dependence between genomes}
We analyze a data set from the \cite{1000GenomesProjectConsortium2015}, 
restricted to chromosome~22. For this study we select  
$n=2000$ individuals and focus on a genomic window between 
base pairs $21.55$\,Mb and $21.65$\,Mb, corresponding to a $100$\,kb region. In population genetics alleles in close physical proximity are often inherited together as they are less likely to be separated by recombination. This phenomenon is well documented \citep{Slatkin2008} and known under the name linkage disequilibrium (LD). We hence expect the existence of some base pairs with high correlation and will investigate whether the proposed methodology is able to detect them or not. To be precise we are interested in detecting Kendall correlations exceeding $\Delta=0.493$ to search for genes that may exhibit co-expression. The choice of the threshold $\Delta$ is here motivated by \cite{Tsaparas2006}, who used a Pearson correlation threshold of $0.7$ to define genes as co-expressed. This corresponds to a Kendall's $\tau$ threshold of $\sim 0.493$ for a bivariate elliptically distribution.

For each variant and sample we extract the \texttt{GT} (genotype) field, 
which encodes the pair of alleles as \texttt{0|0}, \texttt{0|1}, or \texttt{1|1}, 
where ``0'' denotes the reference allele and ``1'' the alternative. 
These are converted into integer allele counts
\[
\texttt{0|0} \mapsto 0, 
\quad 
\texttt{0|1} \mapsto 1, 
\quad 
\texttt{1|0} \mapsto 1, 
\quad 
\texttt{1|1} \mapsto 2. 
\]
This yields a genotype matrix
$
X =(X_{ij})_{i=1, \ldots, n}^{j=1, \ldots , d} 
$
with $n=2000$ samples and $d=750$ variants, where $X_{ij} \in \{0,1,2\}$ denotes the allele count for individual $i$ at variant $j$.

To assess dependencies between genes, we compute pairwise Kendall’s rank correlation 
coefficients, where we break ties (tie-adjusted versions have higher sensitivity, making private inference substantially more difficult) by adding independent normal noise with standard deviation $10^{-6}$, leading to a practically negligible bias of order $10^{-12}$ \citep[see][]{Kitagawa2018}. This yields a symmetric matrix of empirical Kendall's $\hat \tau$ coefficients
\[
\hat \tau = (\hat \tau_{jk})_{1 \le j,k \le d} \in [-1,1]^{d \times d}~,
\]
and the corresponding vector $U = \text{vech} (\hat \tau) $ of  $U$-statistics has dimension $p = 280,875$. In this example it is possible to identify a gap between the differences $q_i = |U|_{(i)} - |U|_{(i+1)}$, see   
Figure~\ref{fig:gaps_genome} (a) (the threshold from Algorithm \ref{alg:topk} is given by
$t = 8/n = 0.004$ and depicted by the dotted red line). 
We  observe that many gaps clearly exceed this bound, 
also already highlighting that a substantial subset of variant pairs exhibit non-zero dependence.
\begin{figure}
\centering

\begin{tabular}{@{}cc@{}}
\includegraphics[width=0.4\linewidth]{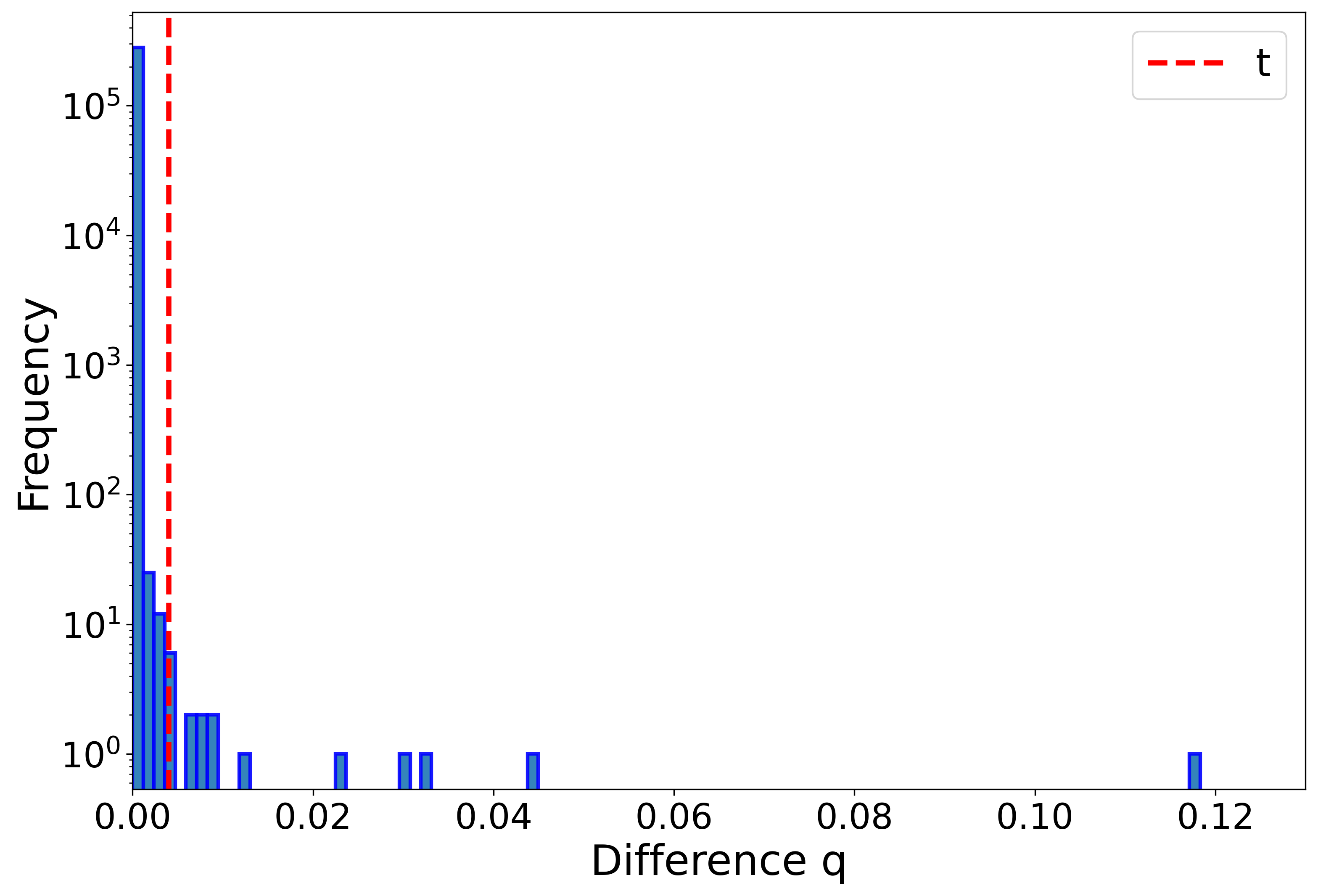} &
\includegraphics[width=0.4\linewidth]{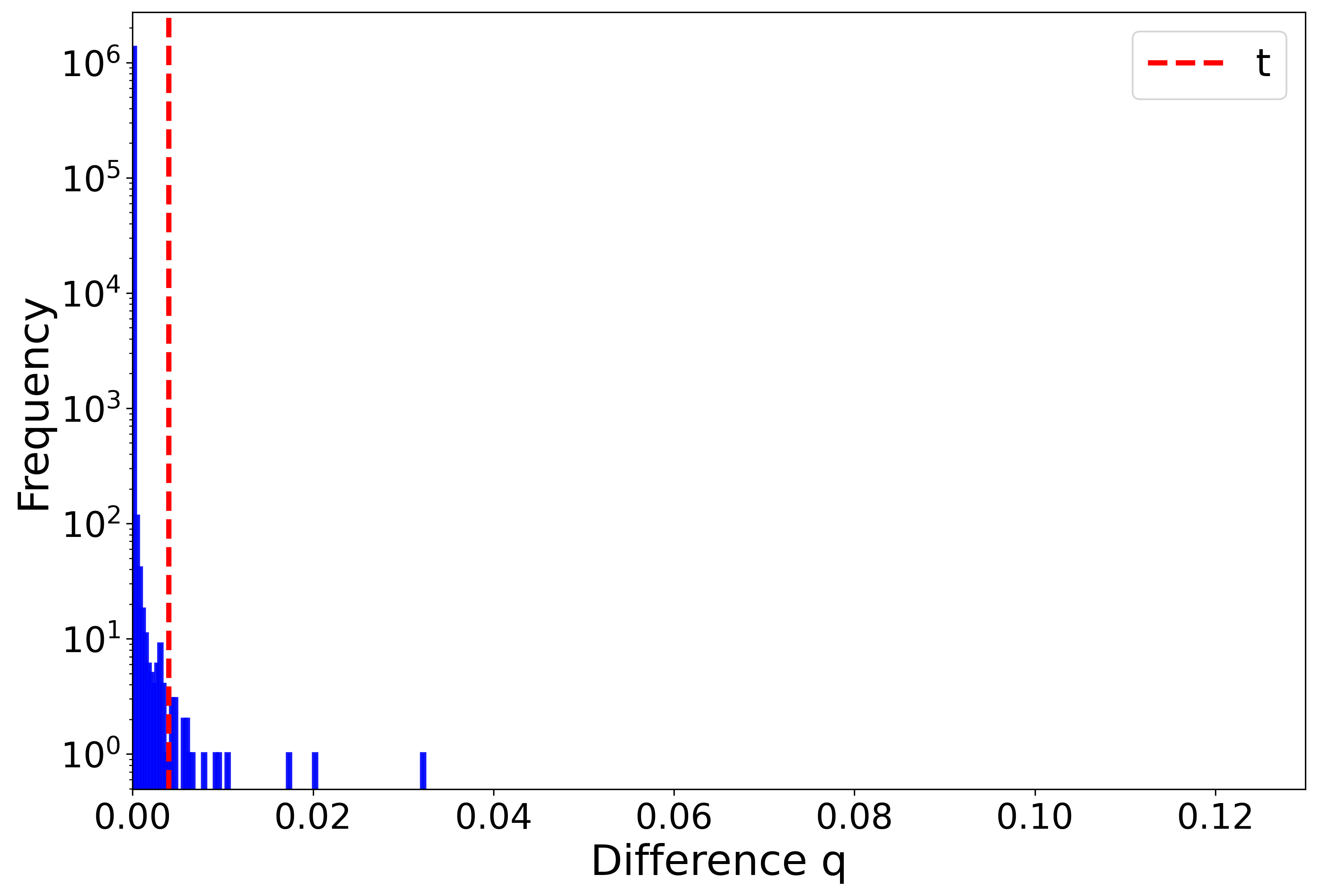} \\

\parbox{0.4\linewidth}{\centering\small (a) 21.55--21.65 Mb window} &
\parbox{0.4\linewidth}{\centering\small (b) 20.45--20.65 Mb window}
\end{tabular}

\vspace{-0.2cm}

\caption{\centering\it
Differences $q_i = |U|_{(i)} - |U|_{(i+1)}$ between successive order statistics of
$U=\mathrm{vech}(\hat{\tau})$ for the genome data of \cite{1000GenomesProjectConsortium2015}.}
\label{fig:gaps_genome}
\end{figure}

We now apply the  test  defined by  Algorithm \ref{alg:HD_test} with a moderate privacy level of $\rho=1,\, \delta=1/n$. 
 The extremal set $\hat{\mathcal E}^{\text{DP}}$ identified by Algorithm \ref{alg:topk} is a singleton containing the coordinate corresponding to $|U|_{(1)}$. Repeating Algorithm \ref{alg:topk} for different seeds exhibited only small changes, almost always choosing one of the gaps $q_i$ to the right of the red line. For the subsequent testing steps of our method we used the same parameters as in the simulation study ($\alpha=0.05$, $B=500$) and rejected the null of $\|\tau\|_\infty \leq \Delta$ for all  $\Delta \leq  0.63$. We hence obtain strong evidence 
that within the 100\,kb region of chromosome~22, 
there exist groups of variants with Kendall's tau larger than $0.493$, indicating the existence of co-expressed genes. In particular, the high magnitude of the detected correlations is fully 
consistent with the presence of Linkage Disequilibrium in this genomic region. Moreover, due to the nature of our approach we also immediately obtain the gene pair with the highest Kendall's $\tau$ with no extra cost in terms of privacy. In comparison, the naive approach based on Hoeffding's inequality only rejects $H_0(\Delta)$ for  $\Delta \leq  0.47$, yielding weak evidence for the existence of co-expressed genes.

To provide a fuller picture of the capabilities of the proposed method on real data we also consider a scenario where the genomic window is larger and there is no clear gap separating the bulk and the largest correlations.  To this end we chose the genomic window between base pairs $20.45$\,Mb and $20.65$\,Mb, corresponding to a $200$\,kb region and keep the sample size at $n=2000$. As evidenced in Figure \ref{fig:gaps_genome} (b),  a gap is visible in the data; however, privacy constraints make reliable detection difficult due to the Gumbel noise introduced by the Report Noisy Max algorithm. Note that shifting the majority of the privacy budget to this step can mitigate this problem. However, we do not adopt this approach, as our goal is to illustrate that even without such an adjustment, Algorithm \ref{alg:HD_test} still attains satisfactory performance.
The resulting test rejects $H_0(\Delta) $ for all $\Delta\leq 0.53$, yielding a weaker, but qualitatively comparable result to the case where the  gap is clearly visible that still outperforms the Hoeffding approach. This illustrates both the fact that identifying and using gap structures if they are present is a worthwhile endeavor, and also that the proposed method still performs acceptably even if no such gap is detected.

\smallskip

\textbf{Mass-spectrometry re-analysis (prostate cancer vs.\ healthy)}

We re-analyze the protein mass spectrometry data described by \cite{adam2002serum}.
Each sample $i$ provides intensities $X_{i,j}$ at many time-of-flight values $t_j$,
with time-of-flight related to the mass-to-charge ratio ($m/z$) of blood-serum proteins.
There are $157$ healthy and $167$ prostate-cancer patients. Following prior work  \citep{tibshirani2005sparsity,Levina},
$m/z$ sites below $2000$ are discarded.
Intensities are then averaged in consecutive blocks of $10$, yielding $p=23653$ correlation  features per sample. We present a heatmap of the matrix of Kendall's $\tau$ coefficients in Figure \ref{fig:mspec}.
\begin{figure}[ht]
\centering

\begin{tabular}{@{}cc@{}}
\includegraphics[width=0.35\linewidth]{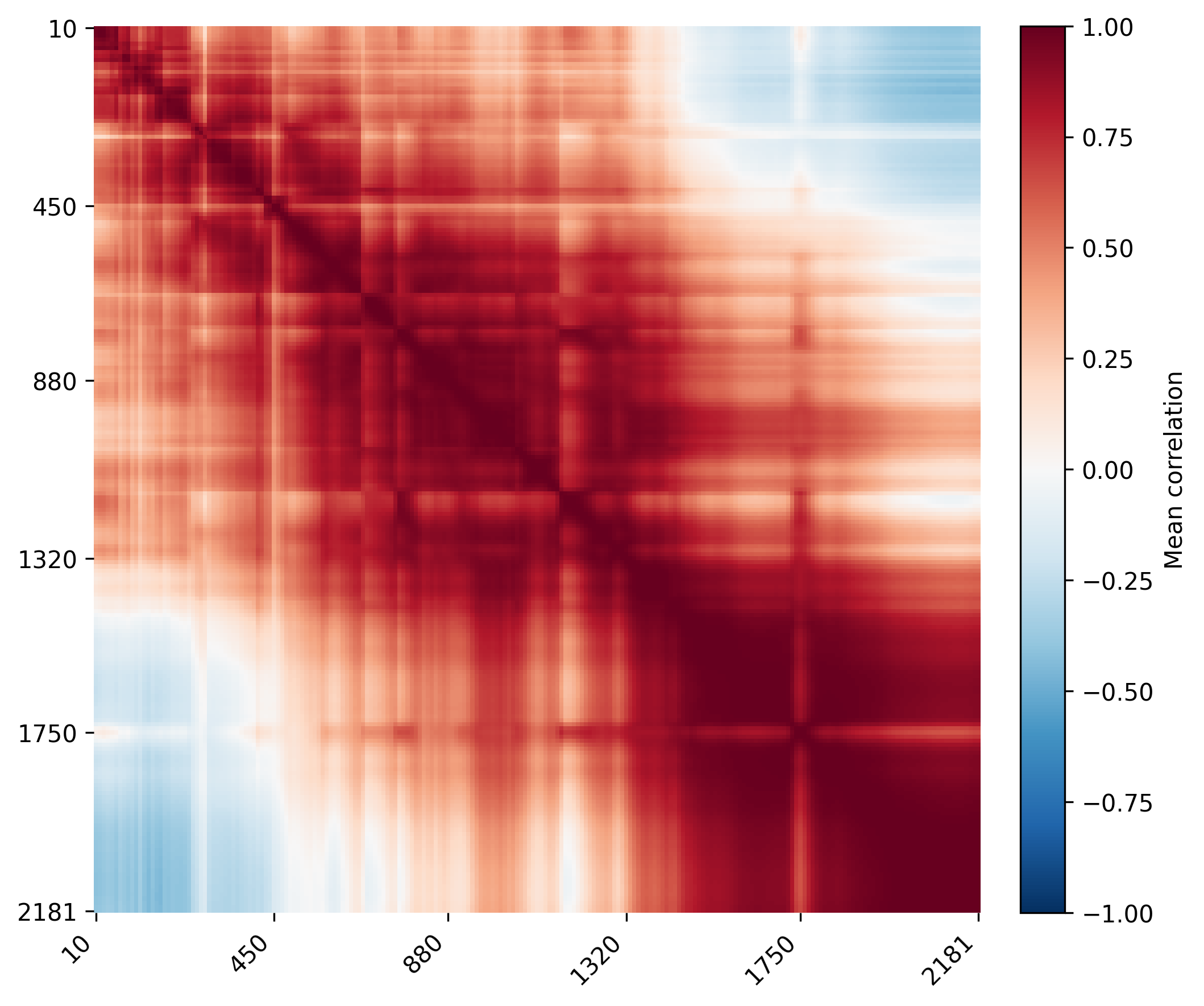} &
\includegraphics[width=0.35\linewidth]{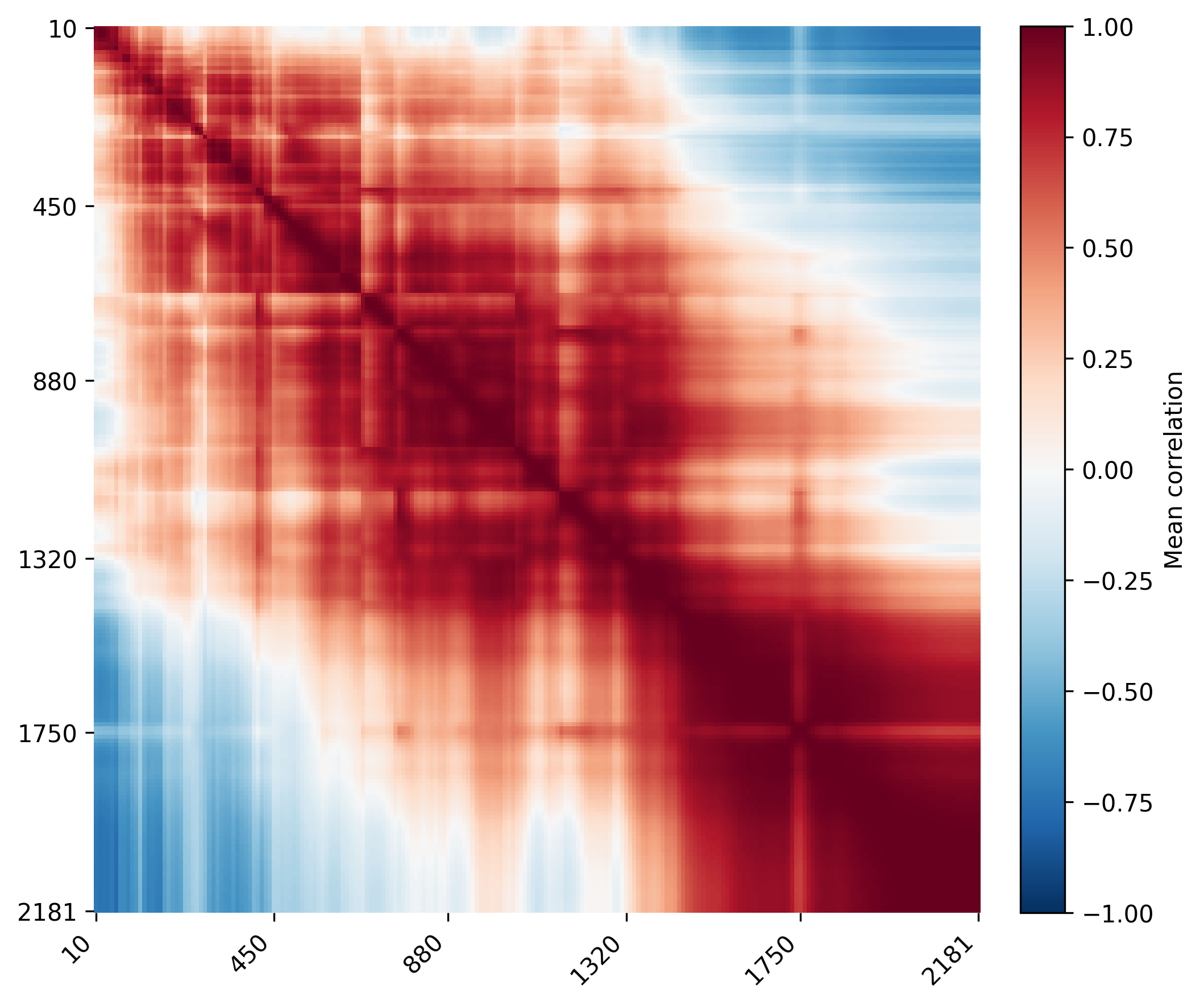}
\end{tabular}

\vspace{-0.3cm}

\caption{\it Block-compressed Kendall's $\tau$ correlation heatmap of healthy (left) group and cancer group (right).}
\label{fig:mspec}
\end{figure}

Both correlation heatmaps clearly display a banded structure but appear visually hard to distinguish. Therefore, it is essential to formulate a testing problem that can detect subtle differences consistently.
Following the previous analyses,  we note that the healthy patients seem to exhibit a bandedness structure that the cancer patients lack (observe the light vs dark blue color in the upper right and lower left corners). While the classical hypothesis framework is not able to demarcate these two structures, \cite{patrick_annals} instead consider the relevant hypotheses-pair of the form
\begin{equation}
  H_0(\Delta): \max_{|i-j|\ge m} \lvert \rho_{ij} \rvert \leq \Delta, \quad \text{vs.} \quad H_1(\Delta):  \max_{|i-j|\ge m} \lvert \rho_{ij} \rvert > \Delta~,
\end{equation}
where $\rho_{ij}$ is the spearman correlation and $m\in \N$ determines the size of the band and is chosen as $m=125$. With this the authors were able to demarcate the two data sets in a non-private setting based on the fact that the hypothesis $H_0(0.1)$ is rejected only for the cancer patients.

We will now analyze this data set with the proposed differentially private methodology at privacy level $\rho=0.1$ with nominal level $\alpha=0.05$. Instead of  Spearman's $\rho$, we will consider Kendall's $\tau$ as in our case the lower rank of the $U$-statistic yields a lower sensitivity bound for the test statistic, and thus requires less noise to privatize. Applying the proposed method we did not detect any gap between the signals. The increments of the ordered $U$-statistics are small across all differences, not exceeding $0.01$ even once, whereas $t=8/n\approx 0.048$ is the threshold in Algorithm \ref{alg:topk}.

We thus resorted to the extreme value based test in Algorithm \ref{alg:HD_test}, which rejects $H_0(\Delta)$ for the healthy group for any $\Delta\leq 0.11$. On the other hand,  the test rejects  for the cancer patients $H_0(\Delta)$  whenever  $\Delta \leq 0.38$. This means that, qualitatively, we obtain the same demarcation between healthy and ill patients as \cite{patrick_annals}  in the non-private setting. 

\section{Conclusions}
  \def\theequation{6.\arabic{equation}}	
	\setcounter{equation}{0}

In this work we consider the problem of testing for practically relevant dependencies in high-dimensional data  under DP constraints. While simple concentration based approaches can work in settings with either a very strong signal or very high sample size  they perform poorly in more realistic scenarios with moderate to strong signal and moderate sample size. We propose a novel private testing procedure  that improves upon this baseline of performance substantially. The improvement is particularly strong for data where the largest signals are separated from the bulk by a gap. In a wide range of scenarios the method even has  good finite sample performance if the size of the gap is fairly small (of order $O(1/n)$). In the case where no such gap is present in the data, the new method still improves upon approaches based on concentration bounds.
As we investigate the problem in the general framework of high-dimensional $U$-statistics our method is also applicable for other testing problems.

We prove the validity of our approach in the asymptotic scenario, where the dimension   increases exponentially with the sample size. The algorithm has excellent finite sample properties if the private and accurate estimation of an extremal set - the set of indices of the coordinates where the signal attains its maximum norm - is possible. We construct an estimator for this purpose, which is accurate in the presence of a gap in the statistics and only requires a privacy budget for the propose-test-release mechanism once. In contrast, methods which sequentially output the indices such as the Report-Noisy-Max or Sparse Vector technique, require a privacy budget that scales as $\sqrt{k}$, where $k$ is the (unknown) cardinality of the extremal set.

\section*{Acknowledgements}
MD was funded by the Deutsche Forschungsgemeinschaft (DFG, German Research Foundation) under Germany's Excellence Strategy - EXC 2092 CASA - 390781972. PB and HD were partially supported  by the
DFG;  TRR 391 {\it Spatio-temporal Statistics for the Transition of Energy and Transport} (520388526);
Research unit 5381 \textit{Mathematical Statistics in the Information Age} (460867398).

\newpage

\appendix

\begin{frontmatter}
	\title{Supplement to: Differentially private testing for relevant dependencies in high dimensions}
	\runtitle{RELEVANT DEPENDENCIES IN HIGH DIMENSIONS UNDER PRIVACY}
	
	\begin{aug}
		\author[A]{\fnms{Patrick}~\snm{Bastian}\ead[label=e1]{patrick.bastian@rub.de}},
		\author[B]{\fnms{Holger}~\snm{Dette}\ead[label=e2]{holger.dette@rub.de}}
		\and
		\author[B]{\fnms{Martin}~\snm{Dunsche}\ead[label=e3]{martin.dunsche@rub.de}}
        \address[A]{Aarhus University\printead[presep={,\ }]{e1}}
		\address[B]{Ruhr-Universität Bochum\printead[presep={,\ }]{e2,e3}}
		
	\end{aug}

\end{frontmatter}
\section{Limitations of the sparse vector technique}\label{sec:svt}
  \def\theequation{A\arabic{equation}}	
	\setcounter{equation}{0}
   
In this section we compare the privacy guarantees for the sparse vector technique (SVT) in \cite{NEURIPS2020_e9bf14a4} to ours in a high-dimensional regime where $p \gg n$. For that, let us first quickly recap how the SVT technique works: the basic technique lets you answer many noisy queries privately by only revealing which ones are \textit{big enough} to cross a noisy threshold. In simple terms, it adds noise to both the threshold and the queries so you can safely say ``yes'' or ``no'' a few times without spending privacy for each comparison. However, for every positive answer the used privacy increases. First we note that, in the language of $(\ve,\delta)$-privacy our worst privacy bound $(\rho=1,\delta=1/n)$ in the finite sample study relates to $\ve \approx 6$ and $\delta =1/n$ for which our method performs well in all considered cases. Our method is based on (randomly) selecting at most $\log(p)$ coordinates from the extremal set $\hatEDP$ to ensure that privatization of the covariance is still feasible. Let us now consider the SVT with the bound
\[
\ve(\delta)\;\le\;
\frac{\Delta^{2}}{2\sigma_{1}^{2}}
+\frac{2c\,\Delta^{2}}{\sigma_{2}^{2}}
+\sqrt{\,2\!\left(\frac{\Delta^{2}}{2\sigma_{1}^{2}}+\frac{2c\,\Delta^{2}}{\sigma_{2}^{2}}\right)
\left(\log(\delta^{-1})+\log\!\left(c\binom{p}{c}\right)\right)}\,,
\]
from \cite{NEURIPS2020_e9bf14a4} with $c=\log(p)$. If the queries are shuffled in advance this is, mathematically speaking, the same as the random selection in our methodology. Here, following the notation of the paper, $\Delta$ denotes the sensitivity of the queries $q_j=|U_j|-\norm{U}_\infty$,  and $\sigma_1,\sigma_2$ are the user chosen noise variances of the queries and of the threshold, respectively (see \cite{NEURIPS2020_e9bf14a4} for details). At a first glance the right-hand side also only depends logarithmically on the dimension $p$. Unfortunately the constants in this bound are too large for a fruitful application in the present context. We underpin this observation by a simple empirical illustration.
\smallskip 

\noindent\textbf{Numerical illustration.}
As in Section~\ref{sec:experiments}, we take $n=250$, $d=n$, and the pair-wise Kendall's $\tau$s, i.e. $p=d(d-1)/2=31125$. Keep $\delta=1/n$, $\Delta=4r/n$, and first set $\sigma_1=\sigma_2=0.1$. Note that this is quite a generous choice considering that we want to detect gaps between Kendall's $\tau$ associations which take values in $[-1,1]$. We then obtain the upper bound $\ve(\delta)\leq 24.21838$. For a more realistic scenario with good statistical results, i.e. $\sigma_1=\sigma_2=0.01$, we obtain $\ve(\delta)\leq 449.5438$, which, for all intents and purposes, is practically equivalent to no privacy at all. Even raising the sample size, often not possible in practice without major effort, to $n=1000$ does not really improve the results: returning to the less realistic $\sigma_1=\sigma_2=0.1$, we merely obtain $\ve\leq 8.012233$. In contrast our algorithm performs well in these scenarios even for strong privacy regimes such as $\rho=0.1$ (i.e. $\ve \approx 1)$.\\

We conclude emphasizing one more the crucial difference between SVT and our methodology for estimating the extremal set. SVT is a composition of mechanisms, which scales (with respect to privacy) like $\sqrt{k}$ due to composition, where $k$ is the number of positively answered queries. In contrast to that, whenever the data has some structure,  such as a gap of order $O(1/n)$ in the ordered (absolute) $U$-statistics, we can simplify the estimation of the extremal set to a "two-dimensional'' problem from a privacy perspective using Algorithm \ref{alg:rep_noisy_max} and the PTR framework for the estimation of the extremal set. This is already sufficient for our purposes.

\section{Additional Algorithms}
\label{appendixalg}
  \def\theequation{B.\arabic{equation}}	
	\setcounter{equation}{0}
In this section, we present the algorithms that accompany our theoretical results. That includes the resampling procedures (Algorithm \ref{alg_monte_carlo_quantile_highdim}), the Gumbel-based test (Algorithm \ref{alg:Gumbel}) and a private covariance estimator (Algorithm \ref{alg_gausscov}). We also recall the resampling procedure introduced in \cite{Dunsche2025} in Algorithm \ref{alg_monte_carlo_quantile_mult}. In Algorithm \ref{alg_gsvt}, we state the generalized SVT algorithm proposed in \cite{NEURIPS2020_e9bf14a4}.
\begin{algorithm}[H]
 \begin{algorithmic}[1]
 \Require sign matrix: $\hat S$, covariance matrix: $\hat\zeta_1$, sample size $n$,  resample parameter: $B$, privacy parameter: $\rho$, sensitivities: $\Delta_2 \norm{U}_\infty$, $\Delta_2\hat\zeta_1$
 \Ensure Empirical $(1-\alpha)$-quantile: $ q_{1-\alpha}^{*,B}$.
 \Function{HQU}{$\hat\zeta_1$, $\hat S$, $n$, $B$, $\Delta_2 \norm{U}_\infty$,  $\Delta_2 \hat \zeta_1$, $\rho$}
	\For {$b=1,\hdots,B$}
    \State Define $\hat \zeta_1^S:=\hat S\odot\hat\zeta_1$.\Comment{Hadamard product}
    \State Define $\hat \zeta_1^{\text{S,DP}}:=\textproc{Gausscov}(\hat{\zeta_1}^S, \rho, \Delta_{2}\hat\zeta_1)$~
    \State Sample ${U_b}^* \sim \mathcal{N}(0,\hat \zeta_1^{\text{S,DP}})$.
    \State Define $U_b^{\text{DP}, *} := \norm{U_b^*}_\infty^{\text{DP}}$, where $\text{DP}$ denotes the same mechanism used for $\norm{U}_\infty^{\text{DP}}$.  \Comment{Lemma \ref{Lem_steinke_gauss}}
    \State Define $T^*_b:= U_b^{\text{DP}, *} $.
	\EndFor
	\State Sort statistics in ascending order: $({T^*_{(1)}},...,{T^*_{(B)}})= sort(T^*_1,...,{T^*_B})$.
	\State Define $q_{1-\alpha}^{*,B}:={T}^*_{(\lfloor(1-\alpha)B\rfloor)}$~.
	\State \Return $q_{1-\alpha}^{*,B}$
	\EndFunction
    \end{algorithmic}
	\caption{High-Dimensional Quantile Monte Carlo U-Statistics (HQU)}\label{alg_monte_carlo_quantile_highdim}
\end{algorithm}
\begin{algorithm}[H]
  \caption{\textsc{P-GUMBEL-TEST}: Private Test for U-statistics in High Dimensions using Gumbel Approximation}
  \label{alg:Gumbel}
  \begin{algorithmic}[1]
    \Require Data set $X$, privacy budgets $\delta,\rho$, $\alpha$ level.
     \Ensure test decision and $\norm{U}_\infty^{\text{DP}}$.
     \Function{P-GUMBEL-TEST}{$U$, $\rho$, $n$, $p$,$\gamma$}
        \State Compute private estimator $\norm{U}^{\text{DP}}_\infty=\norm{U}_\infty+Z$ with $2\rho/3$. \Comment{Lemma \ref{Lem_steinke_gauss}}
        \State Define $a_p=\sqrt{2\log(p)}$ and $G\sim Gumbel(0,\sqrt{L_\infty-(\Delta-\gamma)^2})$
        \State Define $Q:=\frac{q_{1-\alpha}^G}{a_p}+a_p-\frac{\log\log(p)+\log(4\pi)}{2a_p}$~
        \If{$\norm{U}^{\text{DP}}_\infty\geq Q/\sqrt{n}+\Delta$}
        \State \Return Reject $H_0$ and output $\norm{U}_\infty^{\text{DP}}$.
        \Else
        \State \Return Fail to reject $H_0$ and output $\norm{U}_\infty^{\text{DP}}$.
     \EndIf
     \EndFunction
  \end{algorithmic}
\end{algorithm}

\begin{algorithm}
\begin{algorithmic}[1] 
    \Require $\hat \zeta_1\in \R^{d\times d}$, privacy parameter $\rho$, sensitivity $\Delta_2 \hat \zeta_1$.
    \Ensure Privatized covariance matrix $\hat\zeta_1^{\text{DP}}$.
    \Function{Gausscov}{$\hat \zeta_1$, $\rho$, $\Delta_2 \hat \zeta_1$}
		\State Generate $Z_{i,j}\sim\mathcal N(0,1)$ for $1\leq i\leq j\leq d$ with $Z_{j,i}=Z_{i,j}$.
        \State Define $Z=(Z_{i,j})_{i,j=1,\hdots, d}$ and compute 
        \begin{equation*}
    \hat\zeta_1^{\text{DP}}=\hat\zeta_1+\frac{\Delta_2 \hat \zeta_1}{\sqrt{2\rho}}Z~.
    \end{equation*}
    \Return $\hat\zeta_1^{\text{DP}}$
	 \EndFunction 
    \end{algorithmic}
	\caption{Private Covariance estimation with additive noise \textbf{Gausscov}}\label{alg_gausscov}
\end{algorithm}
\begin{algorithm}[H]
 \begin{algorithmic}[1]
  \Require private covariance matrix: $\hat\zeta_1^{\text{DP}}$, sample size: $n$,  resample parameter: $B$, privacy parameter: $\rho$, sensitivity: $\Delta_2 \norm{U}_\infty$.
  \Ensure Empirical $(1-\alpha)$-quantile: $ q_{1-\alpha}^{*,B}$.
 \Function{QU}{$\hat\zeta_1^{\text{DP}}$, $n$, $B$, $\Delta_2 \norm{U}_\infty$, $\rho$}
	\For {$b=1,\hdots,B$}
    \State Sample ${U_b}^* \sim \mathcal{N}_s(0,\hat \zeta_1^{\text{DP}}/n)$.
    \State Compute private estimator $U_b^{\text{DP}, *} := \norm{U_b^*}_\infty^{\text{DP}} =\norm{U_b^*}_\infty+Z$ with $2\rho/3$. \Comment{Lemma \ref{Lem_steinke_gauss}}
    \State Define $T^*_b:= \sqrt{n}(U_b^{\text{DP},*}) $.
	\EndFor
	\State Sort statistics in ascending order: $({T_{(1)}^*},...,{T_{(B)}^*})= sort(T^*_1,...,{T^*_B})$.
	\State Define $q_{1-\alpha}^{*,B}:={T}^*_{(\lfloor(1-\alpha)B\rfloor)}$~.
	\State \Return $q_{1-\alpha}^{*,B}$
	\EndFunction
    \end{algorithmic}
	\caption{Multivariate Quantile Monte Carlo U-Statistics (QU) \cite{Dunsche2025}}\label{alg_monte_carlo_quantile_mult}
\end{algorithm}
\begin{algorithm}
\begin{algorithmic}[1]
    \Require Data $X$, adaptive queries $q_1, q_2, \ldots \in Q$ with sensitivity $\Delta$, 
    noise mechanisms $\mathcal{M}_\rho, \mathcal{M}_\nu$, threshold $t$, 
    cut-off $c$, max length $k_{\max}$.
    \Ensure Outputs $a_i \in \{\top, \bot\}$.
    \Function{SVT}{$X$, $Q$, $t$, $\Delta$, $c$, $k_{\max}$}
        \State Sample $\hat t\sim \mathcal{M}_\rho(D, t)$
        \State count = 0
        \For{$i = 1,2,3,\ldots, k_{\max}$}
            \State Sample $\hat q_i \sim \mathcal{M}_\nu(X, q_i)$
            \If{$\hat q_i \ge \hat t$}
                \State Output $a_i = \top$, count = count + 1
                \If{count $\ge c$}
                    \State \textbf{abort}
                \EndIf
            \Else
                \State Output $a_i = \bot$
            \EndIf
        \EndFor
    \EndFunction
\end{algorithmic}
\caption{Generalized Sparse Vector Technique (SVT) (slightly adjusted Algorithm 2 in \cite{NEURIPS2020_e9bf14a4}}
\label{alg_gsvt}
\end{algorithm}

\section{Additional Simulations}\label{appadditionalsim}
  \def\theequation{C.\arabic{equation}}	
	\setcounter{equation}{0}
In this section we present further simulation results. First we study the robustness of the proposed method with respect to violations of the assumptions. For this purpose we consider \textbf{U}nfavorable scenarios, in which our theory does not predict the performance of Algorithm \ref{alg:HD_test}. Second, we compare our methodology to the state of the art in the non-private setting and demonstrate superior performance in some settings despite the additional cost incurred by ensuring DP. 
\smallskip 

\textbf{Robustness} 
From a theoretical perspective, the strong performance of the test defined by Algorithm \ref{alg:HD_test} relies on the existence of a sufficiently large gap. While this assumption is often reasonable in practice, an obvious question is how the procedure performs in less favorable settings. For this purpose, we consider two further models. All other parameters are the same as in Section \ref{simul}.

\begin{itemize}
    \item [\textbf{U1)}] A dense signal with $\norm{\text{vech}(\mathcal T)}_\infty=0.5$ but no gaps, i.e. 
    we define \( m = \left\lfloor d / \sqrt{2} \right\rfloor \) and construct the matrix \( \mathcal T \in \mathbb{R}^{d \times d} \) as follows:
\begin{equation}
    \label{taumatricesU1}
    \mathcal T = 
\begin{bmatrix}
\mathbf{A} & 0 \\
0 & \mathbf{I}_{d - m}
\end{bmatrix}
\end{equation}
where $$\mathbf{A}:=\mathbf{A}'- diag(\mathbf{A}')+\mathbf{I}_m
$$
and \( \mathbf{A}' = \mathbf{a} \mathbf{a}^\top \in \mathbb{R}^{m \times m} \). Here, \( \mathbf{a} \in \mathbb{R}^m \) is an equidistant scaled vector
\[
\mathbf{a} = \sqrt{\frac{0.5}{\max_{i < j} (b_i b_j)}} \mathbf{b}~.
\]
with $\mathbf{b} = (0.01+ {(j-1)(0.99-0.01)}/({m-1}))_{j=1,\hdots, m}\in \R^m$.

\item [\textbf{U2)}] A dense signal with $\norm{\text{vech}(\mathcal T)}_\infty=0.5$ but two gaps, i.e. for $m=\lfloor d/\sqrt{2}\rfloor$ we take
  \begin{equation}
   \label{taumatricesU2}
\mathcal T =
\begin{bmatrix}
0.5  \mathbf{I}_m + 0.5  \mathbf{J}_m & 0 \\
0 & 0.75  \mathbf{I}_{d - m} + 0.25  \mathbf{J}_{d - m}
\end{bmatrix}
\end{equation}
where identity matrix,
\(  \mathbf{J}_l \) is a \( l \times l \) matrix with  all entries equal to one.
\end{itemize}
 Note that in both scenarios at least one of the assumptions is violated. We display in Figure \ref{fig:power_curves_grid_high_unfav} and \ref{fig:power_curves_grid_high_unfav_high} the rejection probabilities of the test  in the scenarios \textbf{U1), U2)} for  a moderate dimensional ($p\approx n$) and high-dimension ($p\approx n^2/2$) setting, respectively.\\
\textbf{Performance for unfavorable setups}
Again, the test keeps its nominal level across all cases under consideration. For the scenario \textbf{U2)}, we observe very good power which is comparable but slightly inferior to power for the scenarios \textbf{F1)} and \textbf{F2)}. As this scenario has two gaps we sometimes detect the gap between the $0$ and $0.25$ valued correlations, leading to more than $\log(p)$ coordinates that are deemed relevant and thus to randomly sampling from the $0.25$ and $0.5$ valued correlation coordinates which reduces the detection power. This effect can be, to some degree, alleviated by penalizing gaps of smaller coordinates by appropriately choosing appropriate weights $\nu(j)$ in Algorithm \ref{alg:rep_noisy_max}, e.g. by $c\big (1- {j}/{n}\big )$. Some care is needed in that case, as this does not lead to uniform improvement across all possible scenarios. In scenario \textbf{U1)} there is no gap of sufficient size for reliable detection. Accordingly we do not reliably separate relevant from irrelevant coordinates, leading to use of the Gumbel approximation in most cases. Nevertheless the test works reasonably well even in this scenario. 

\begin{figure}[ht]
\centering

\begin{tabular}{@{}ccc@{}}
\includegraphics[width=0.32\linewidth]{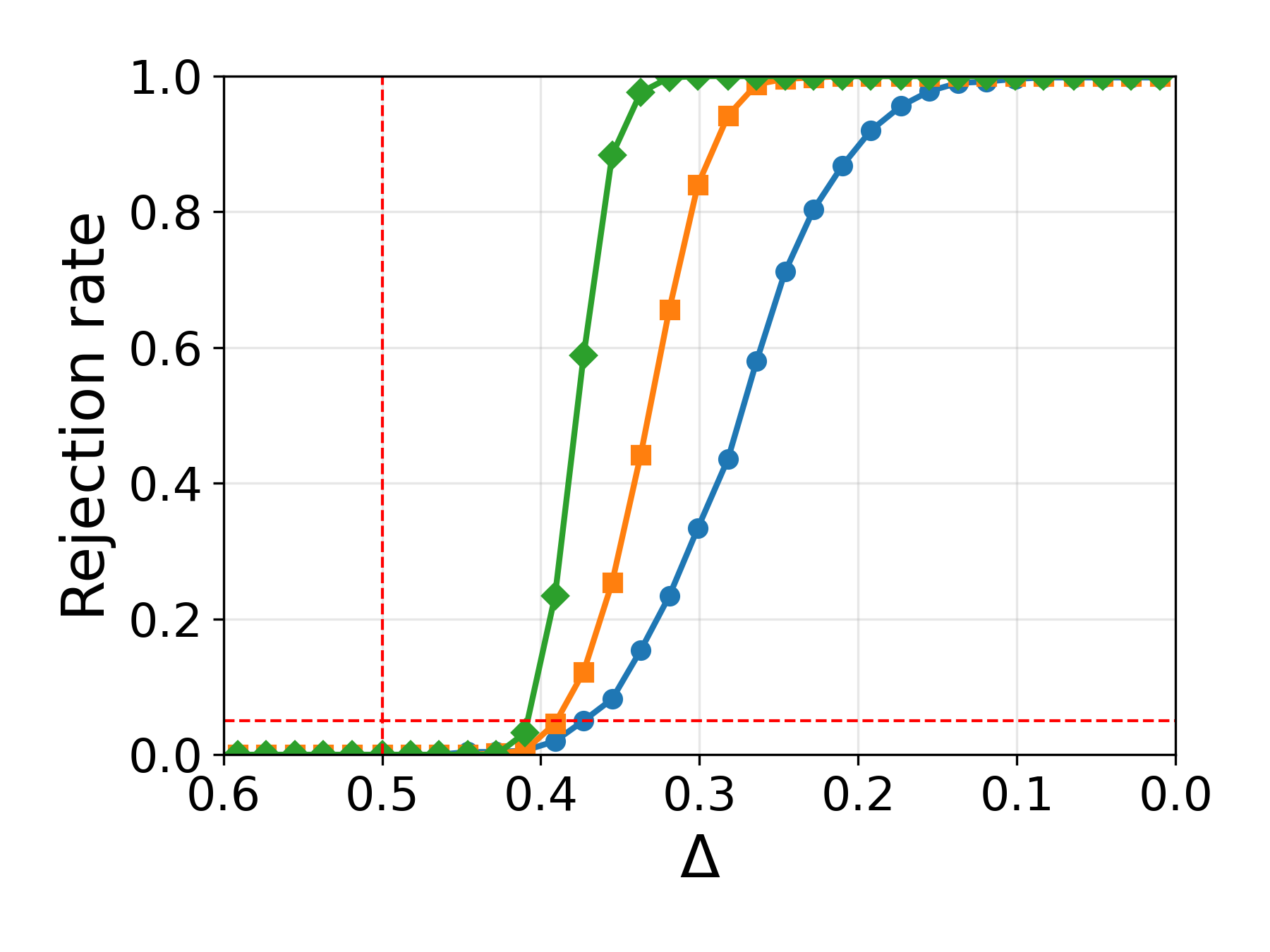} &
\includegraphics[width=0.32\linewidth]{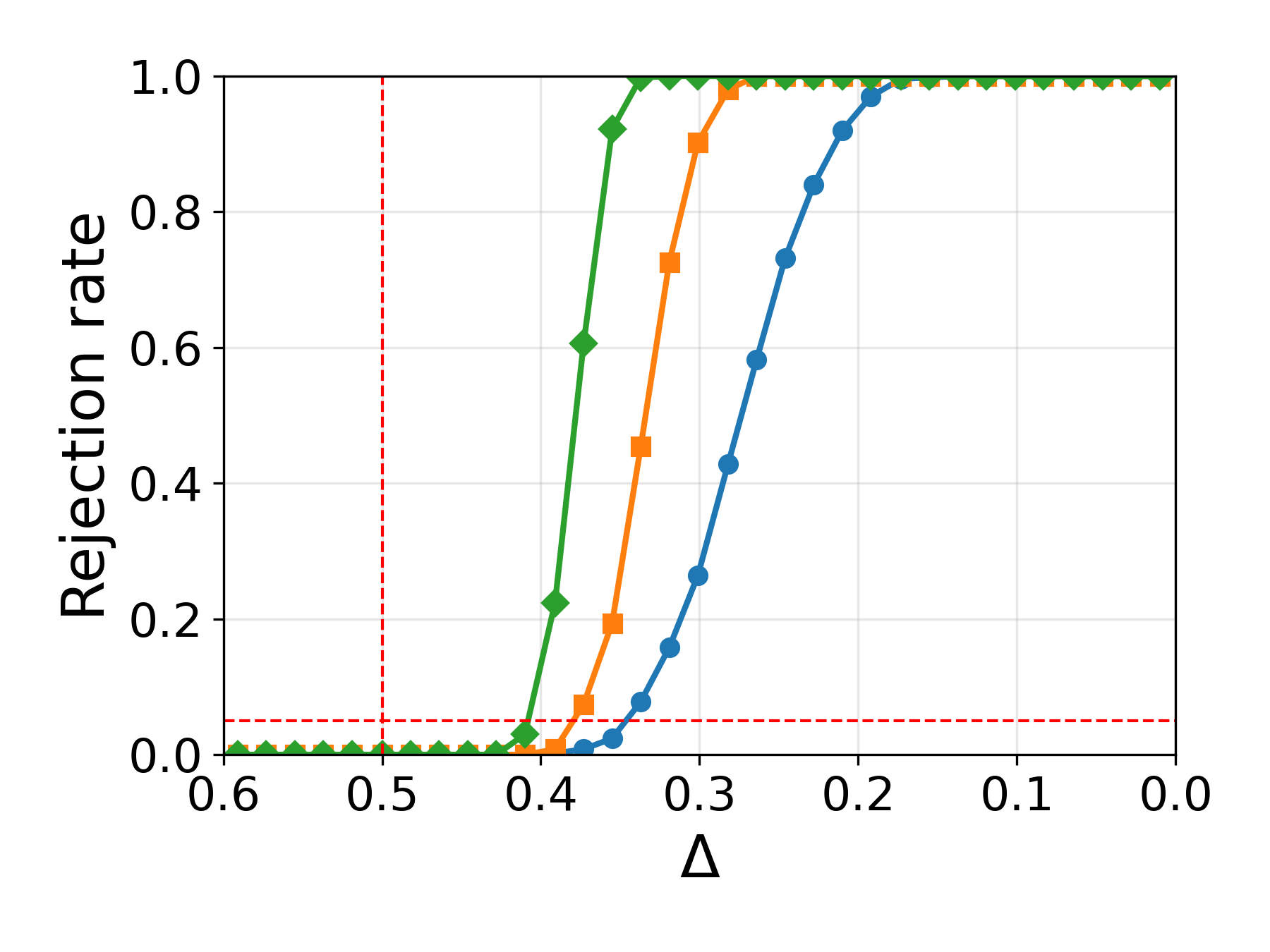} &
\includegraphics[width=0.32\linewidth]{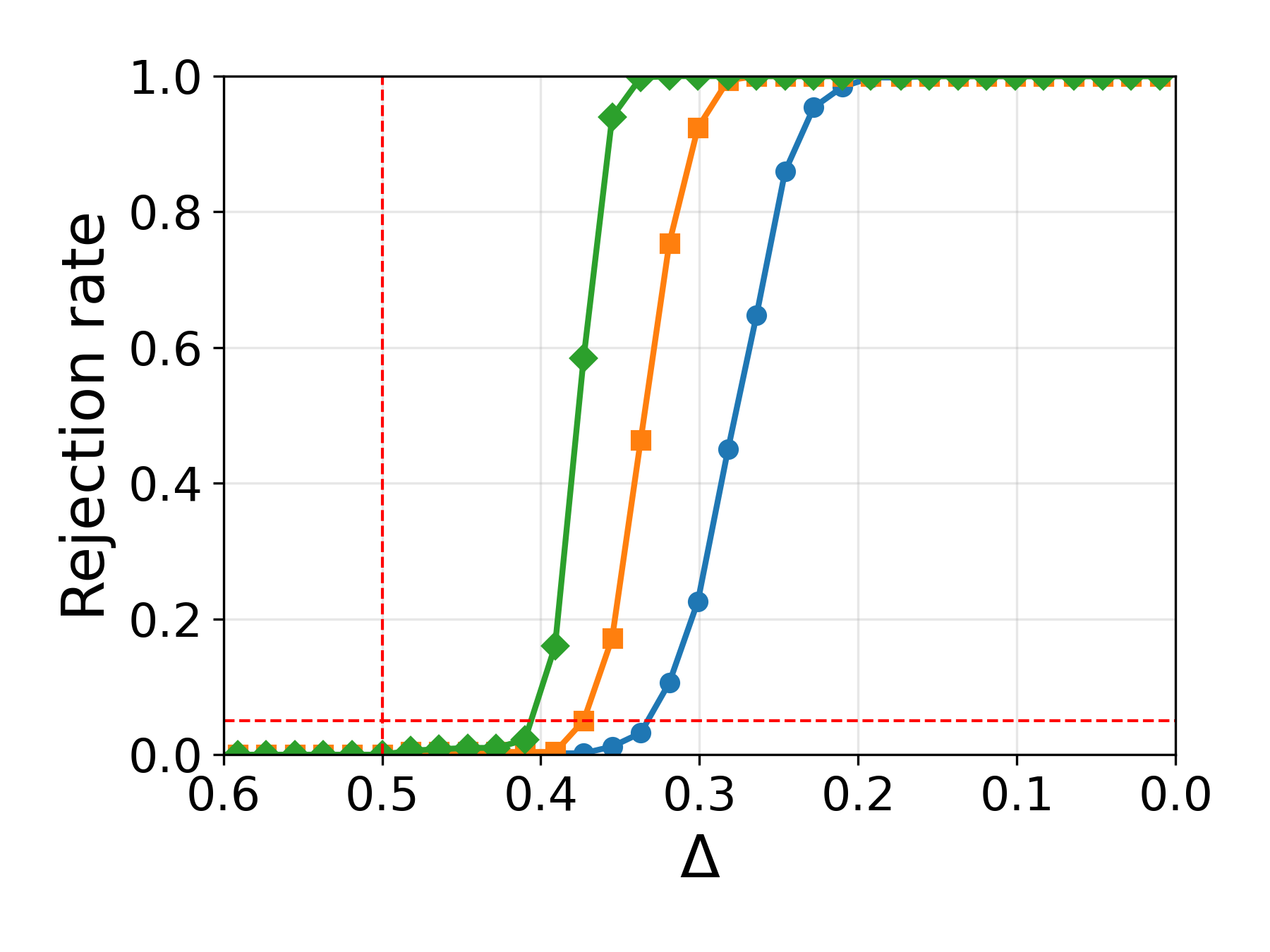} \\[-0.6cm]

\includegraphics[width=0.32\linewidth]{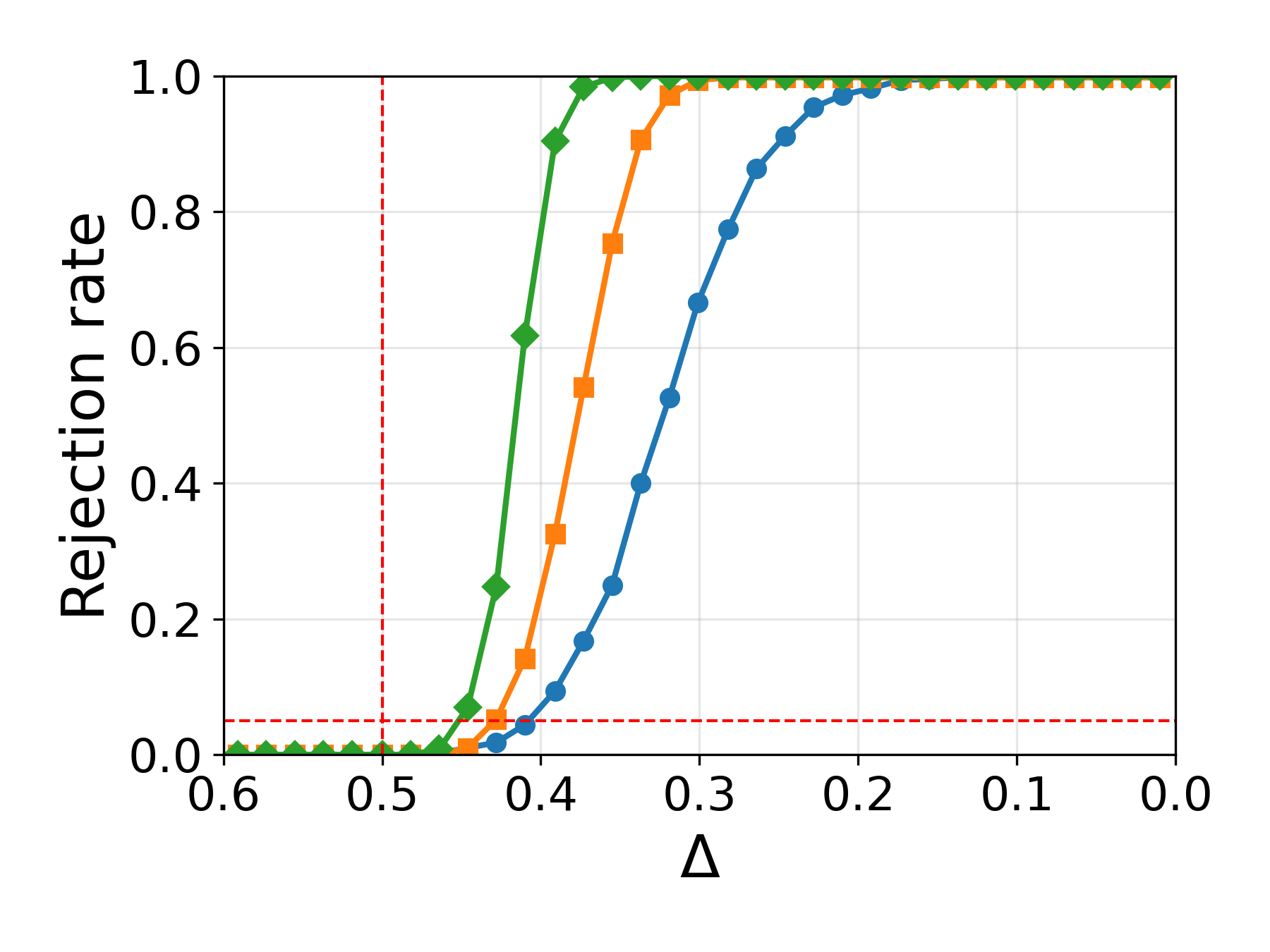} &
\includegraphics[width=0.32\linewidth]{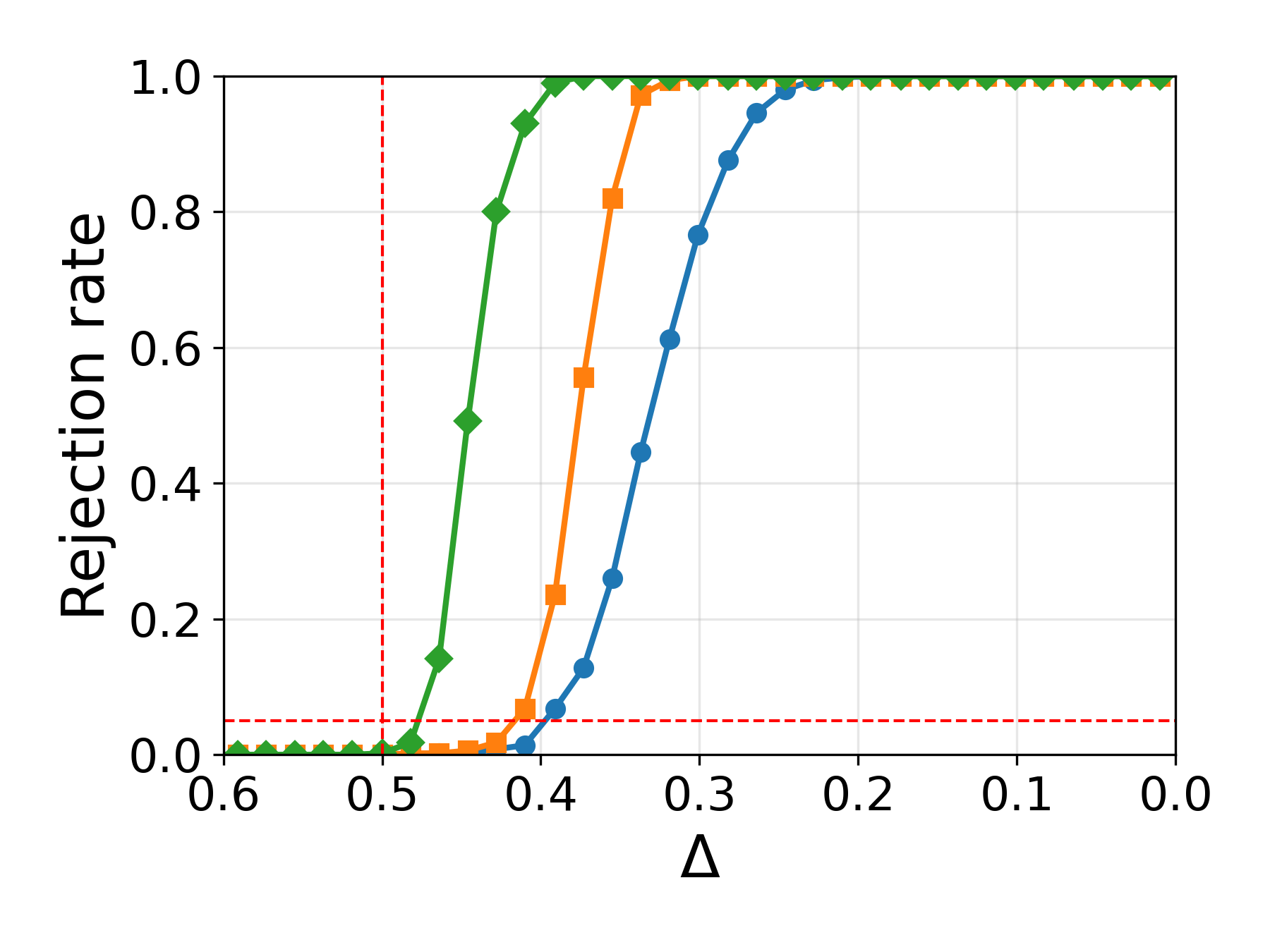} &
\includegraphics[width=0.32\linewidth]{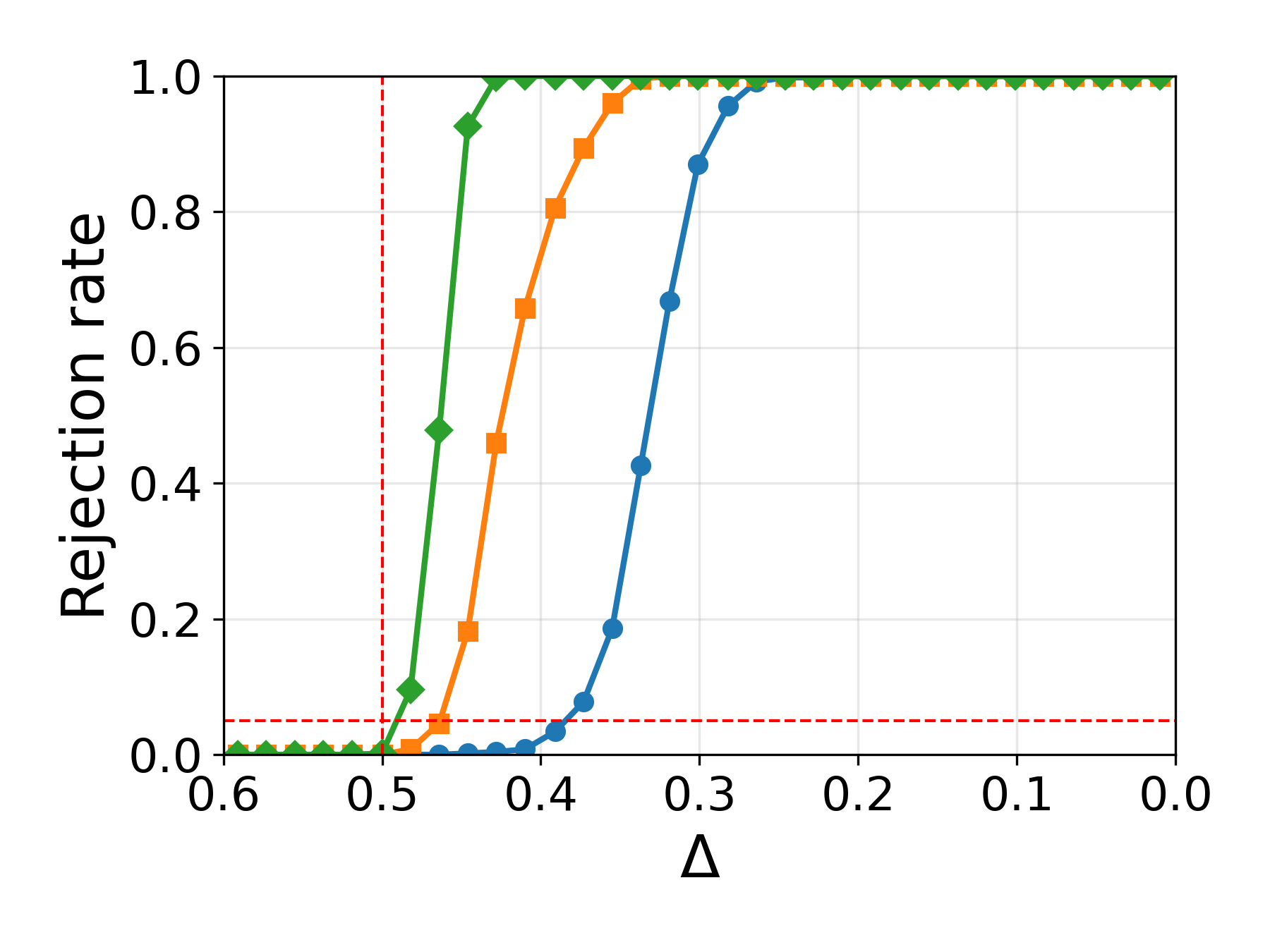}
\end{tabular}

\vspace{-0.3cm}

\makebox[\linewidth][c]{%
  \includegraphics[width=0.6\linewidth]{figures/legend_sample_size_left.png}
}
\vspace{-0.6cm}
\caption{\it Empirical rejection probabilities of the test defined by Algorithm \ref{alg:HD_test} for different privacy parameters $\rho=0.1,0.25,1$ and models \textbf{U1)} (first row) and \textbf{U2)} (second row) with $n \in \{250,500,1000\}$, $p=d(d-1)/2$ with $d=\lceil \sqrt{2n}\rceil$ (moderate dimensional regime).}
\label{fig:power_curves_grid_high_unfav}
\end{figure}
\begin{figure}[H]
\centering

\begin{tabular}{@{}ccc@{}}
\includegraphics[width=0.32\linewidth]{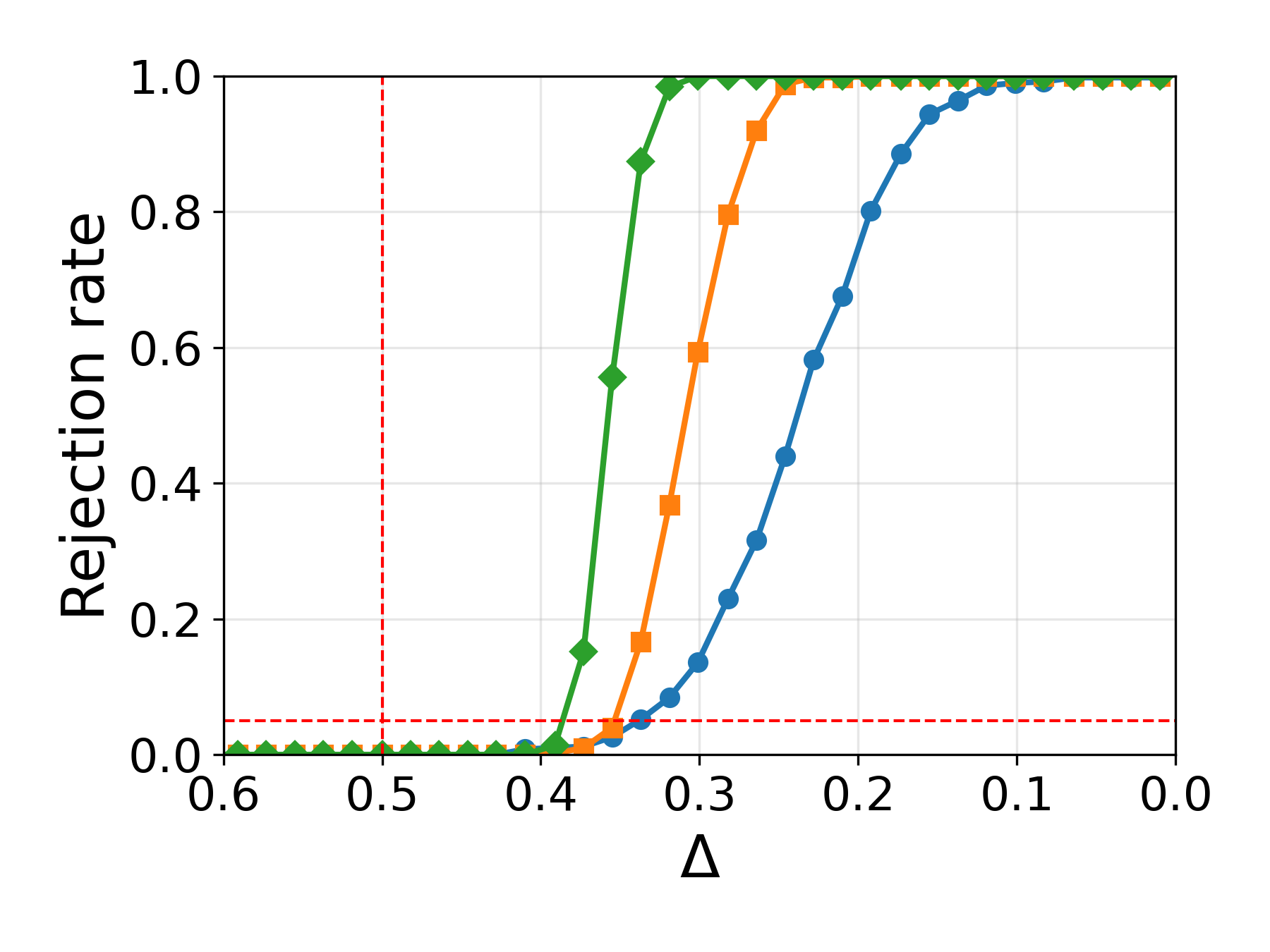} &
\includegraphics[width=0.32\linewidth]{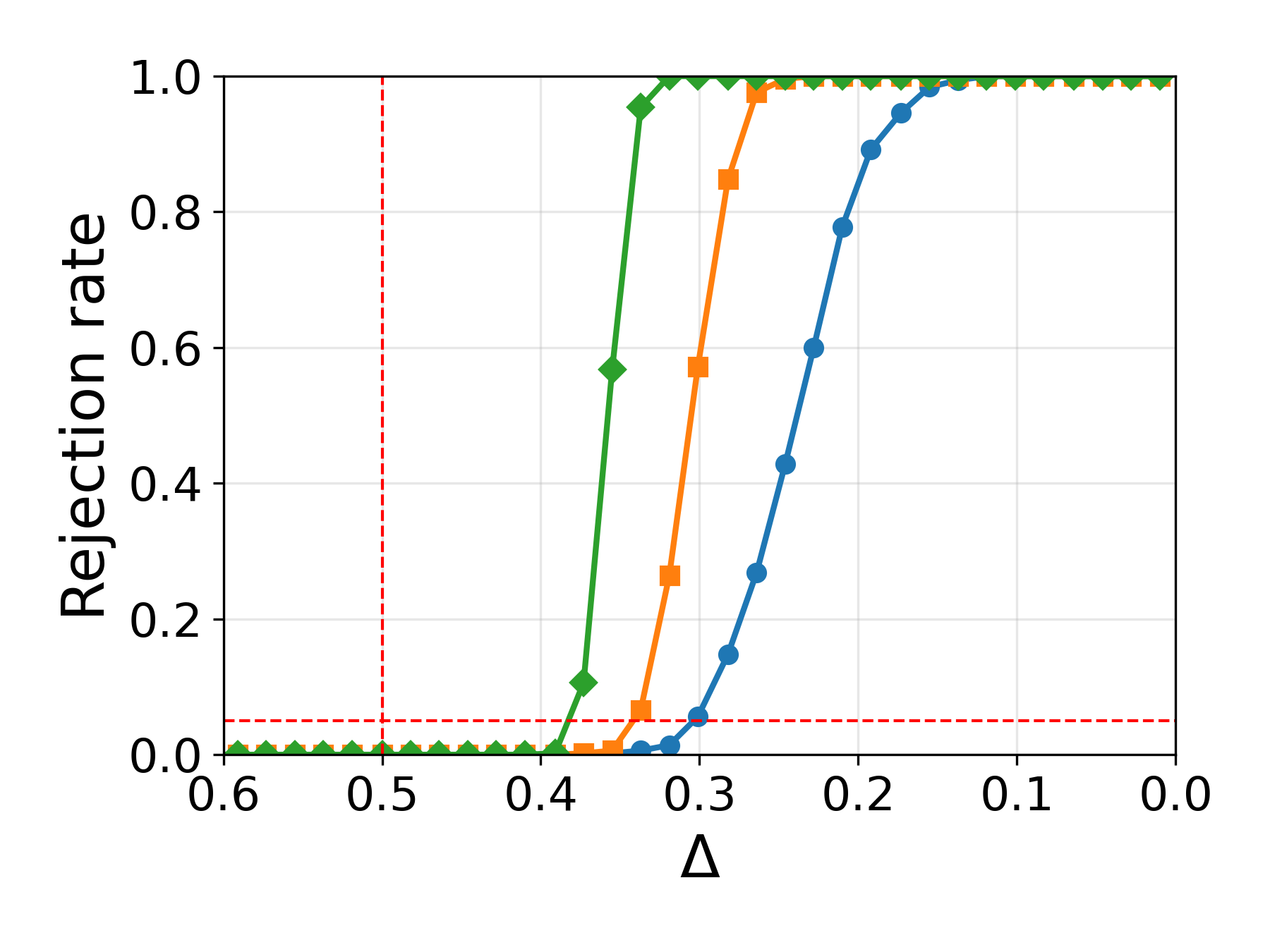} &
\includegraphics[width=0.32\linewidth]{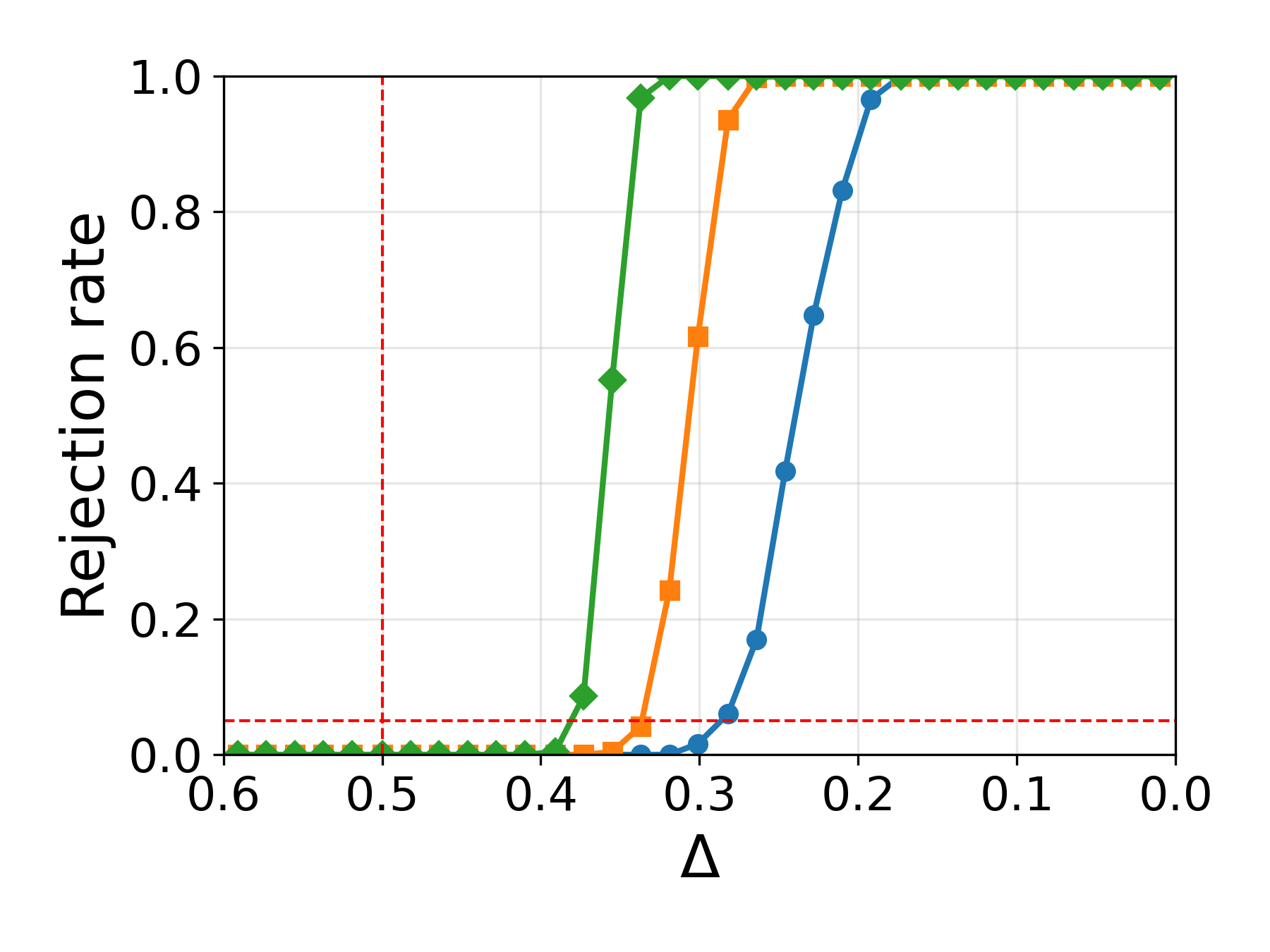} \\[-0.6cm]

\includegraphics[width=0.32\linewidth]{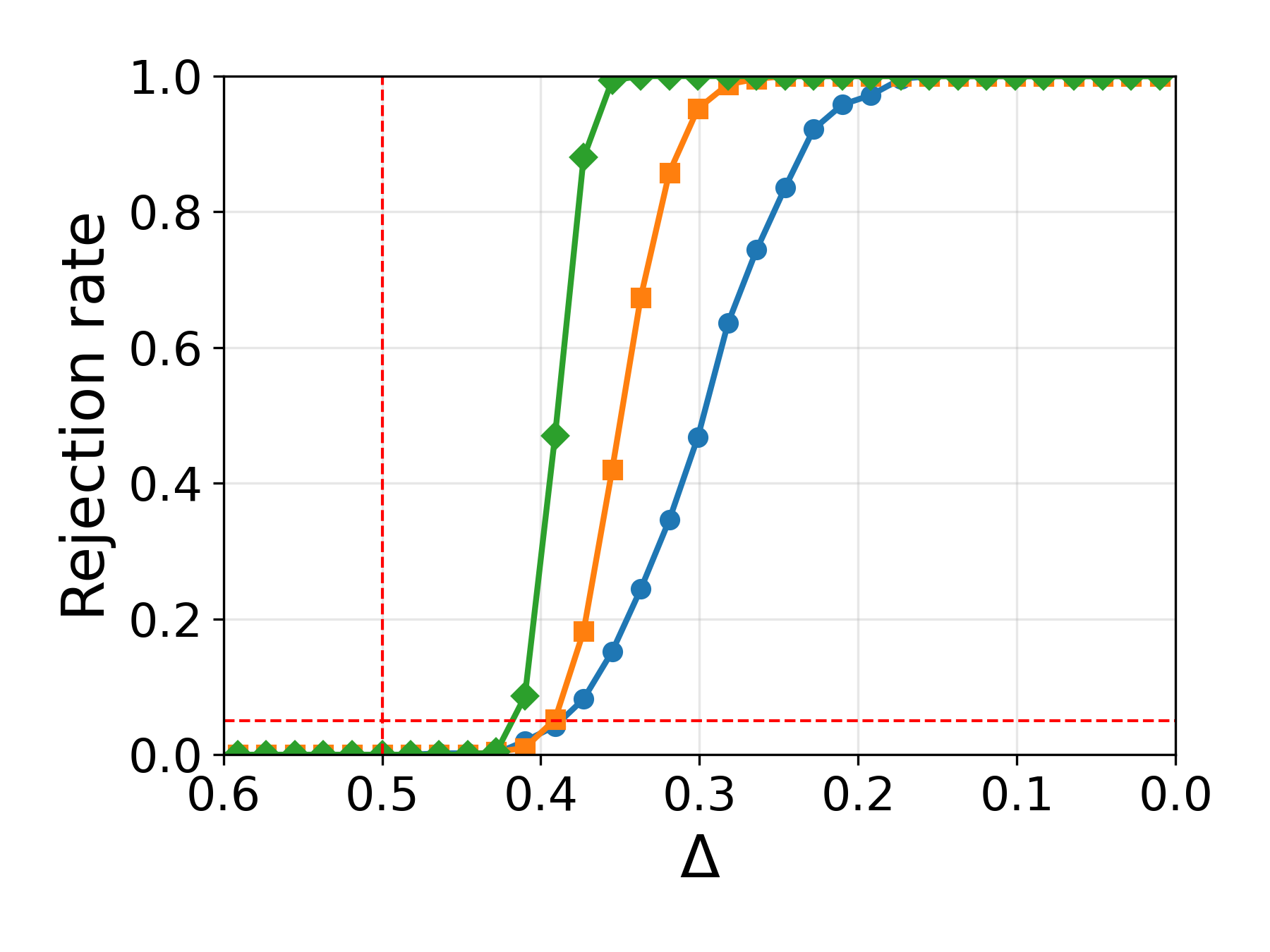} &
\includegraphics[width=0.32\linewidth]{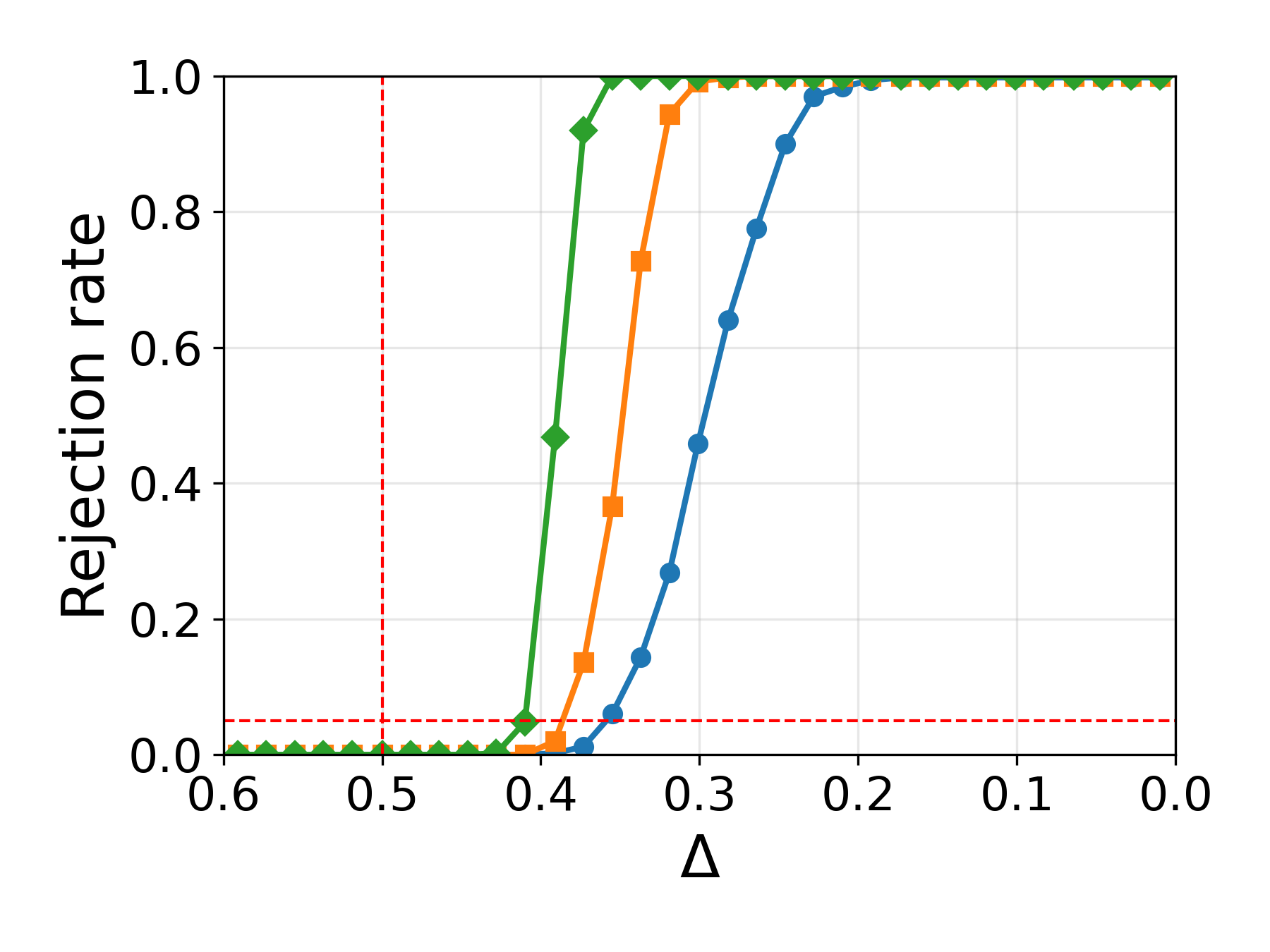} &
\includegraphics[width=0.32\linewidth]{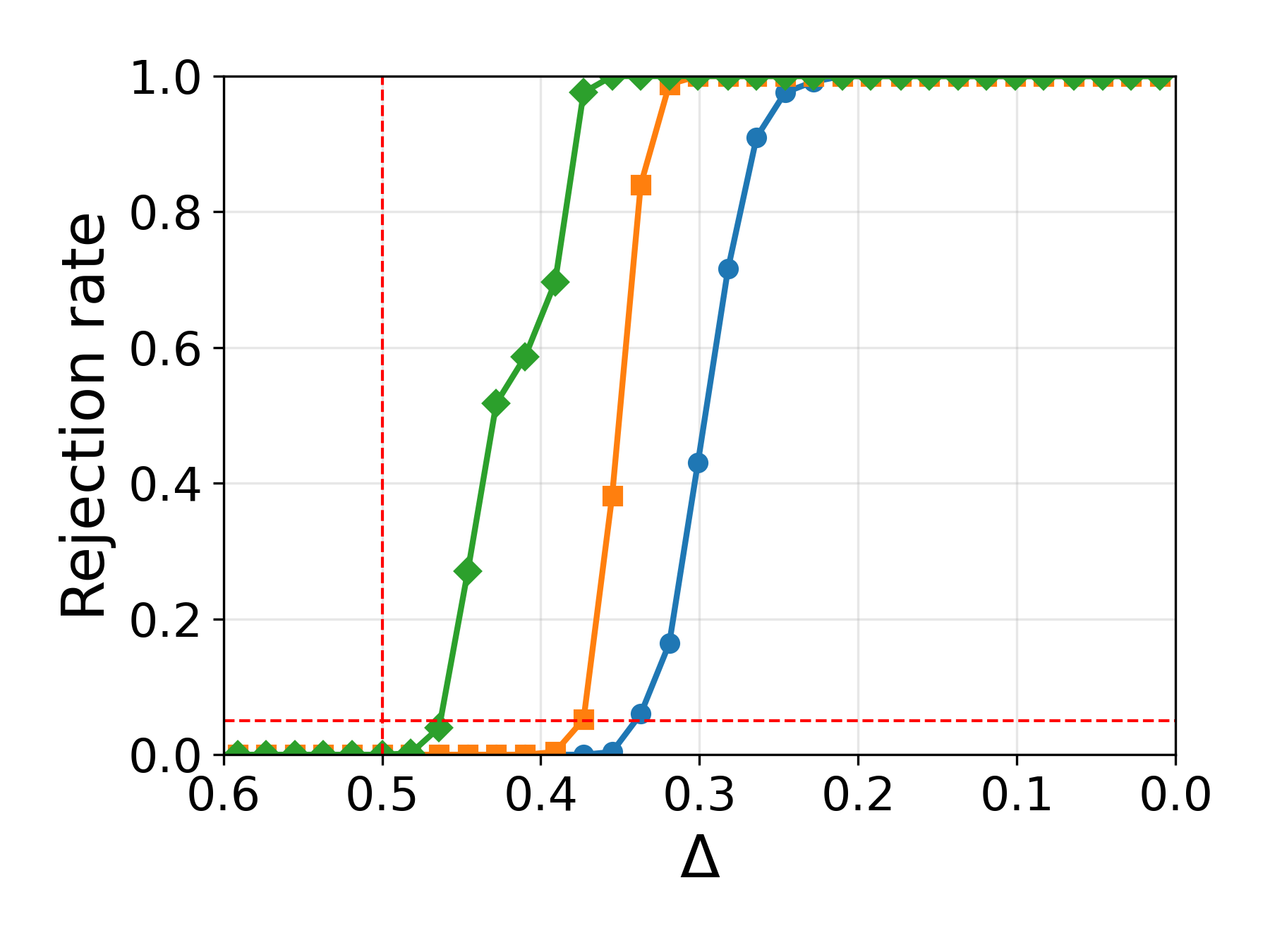}
\end{tabular}

\vspace{-0.3cm}

\makebox[\linewidth][c]{%
  \includegraphics[width=0.6\linewidth]{figures/legend_sample_size_left.png}
}
\vspace{-0.6cm}
\caption{\it Empirical rejection probabilities of the test defined by Algorithm \ref{alg:HD_test} for different privacy parameters $\rho=0.1,0.25,1$ and models \textbf{U1)} (first row) and \textbf{U2)} (second row) with $n \in \{250,500,1000\}$, $p=d(d-1)/2$ with $d=n$ (high-dimensional regime).}
\label{fig:power_curves_grid_high_unfav_high}
\end{figure}

\smallskip
\textbf{Performance compared to non-private state of the art}
In Figure \ref{fig:power_curves_grid_comp}  we display the rejection probabilities of the test defined by Algorithm \ref{alg:HD_test} and  the non-private test proposed in equation (2.27) of  \cite{patrick_annals}, where we consider   the settings \textbf{F1)} and \textbf{F2)} from Section \ref{simul}. In  \textbf{F2)} only a few coordinates carry a signal, and we observe that the new  private procedure can even outperform the non-private state-of-the-art test. This superiority  arises because the methodology in \cite{patrick_annals}  does not attempt to estimate the extremal set, and thus suffers from the difficulty discussed in Section \ref{sec312}. In the dense setting \textbf{F1)}, however, we begin to see the effect of privacy more clearly, particularly for smaller sample sizes. Here the effect of not estimating the extremal set becomes less pronounced and the loss of power by the additional privacy noise becomes more clearly visible.
\begin{figure}[H]
\centering

\begin{tabular}{@{}cc@{}}
\includegraphics[width=0.4\linewidth]{figures/power_merged_eps1.0_d45_M2.png} &
\includegraphics[width=0.4\linewidth]{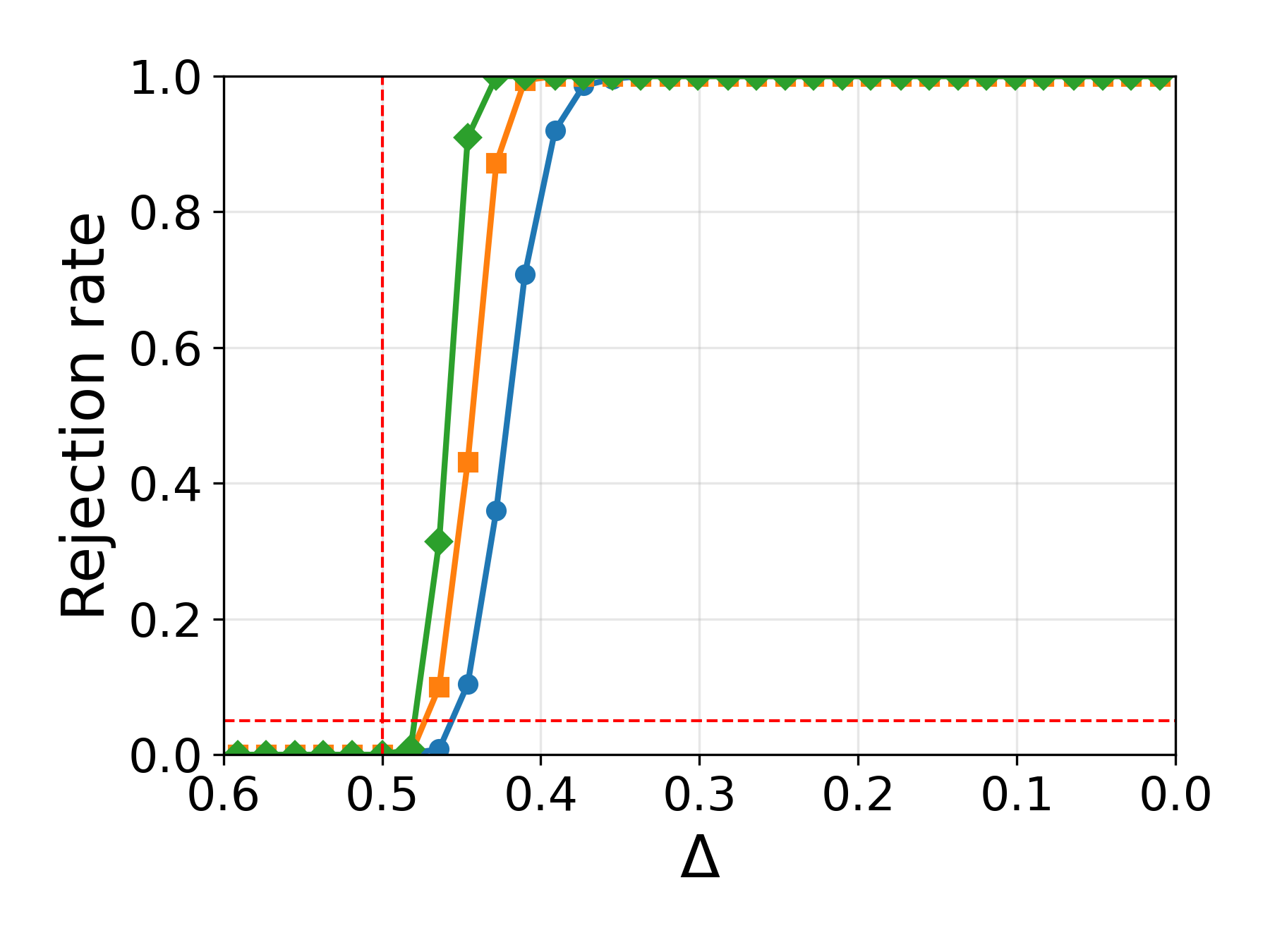} \\[-0.62cm]

\includegraphics[width=0.4\linewidth]{figures/power_merged_eps1.0_d45_M3.png} &
\includegraphics[width=0.4\linewidth]{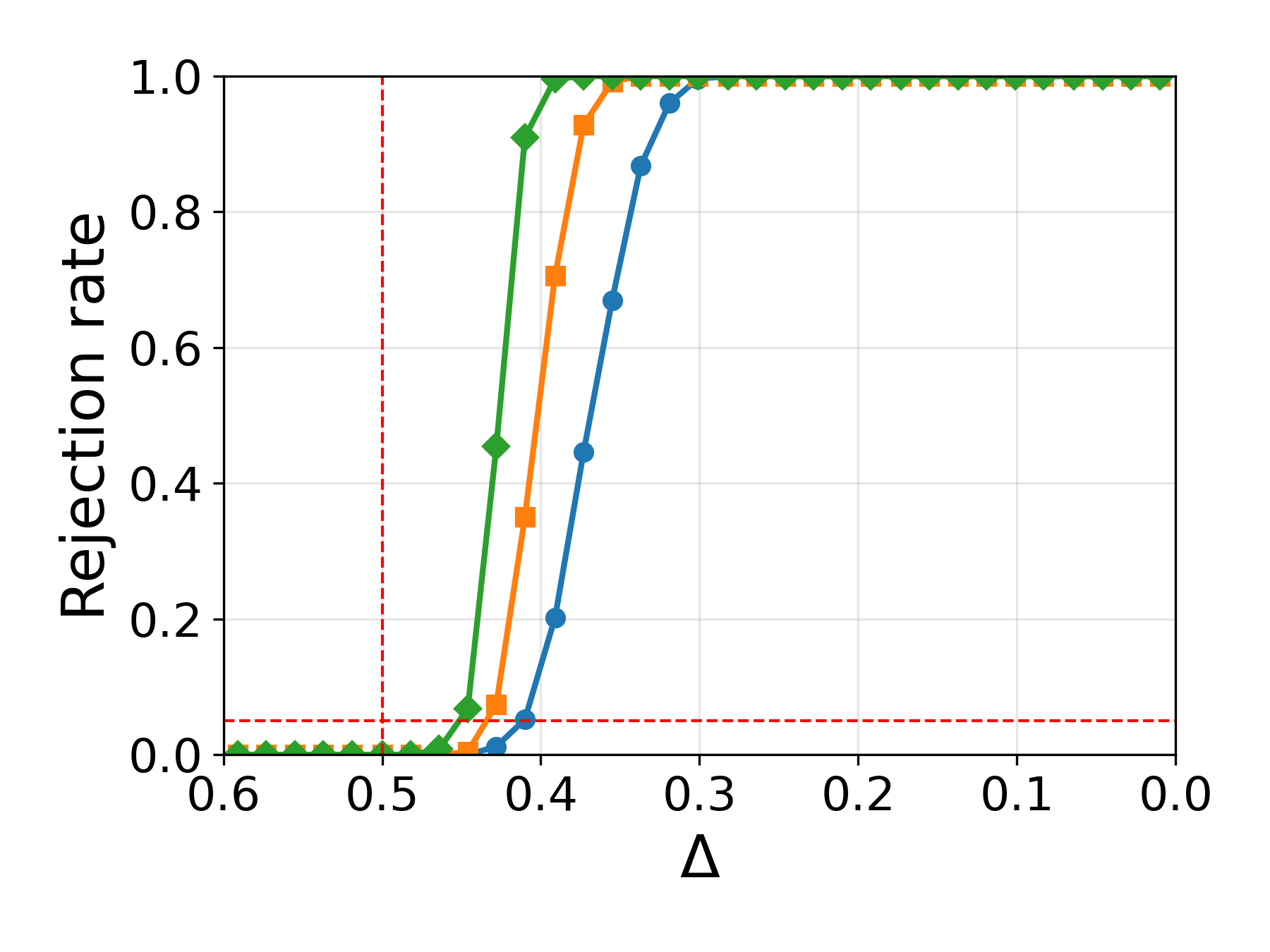} \\[-0.2cm]
\parbox{0.45\linewidth}{\centering\small (a) Algorithm \ref{alg:HD_test}} &
\parbox{0.45\linewidth}{\centering\small (b) Non-private test proposed in \cite{patrick_annals}}
\end{tabular}
\vspace{0.1cm}

\makebox[\linewidth][c]{%
  \includegraphics[width=0.6\linewidth]{figures/legend_sample_size_left.png}
}
\vspace{-0.6cm}
\caption{\it Empirical rejection probabilities of the test defined by Algorithm \ref{alg:HD_test} and the test based on \cite{patrick_annals} for privacy parameters $\rho=1$ in setting \textbf{F1)} (first row) and \textbf{F2)} (second row) with $n \in \{250,500,1000\}$ and $p=d(d-1)/2$ for $d=\lceil\sqrt{2n}\rceil$ (moderate dimensional regime).}
\label{fig:power_curves_grid_comp}
\end{figure}

\section{Finite dimensional methodology}\label{App:finite_dimensional}
  \def\theequation{D.\arabic{equation}}	
	\setcounter{equation}{0}
   
In this section, we will prove that for finite dimension $p$, the decision rule \eqref{eq_finite_dim} defines a consistent asymptotic level $\alpha$ test for the hypotheses \eqref{eq_rel_hyp_high}. Following 
 \cite{Dunsche2025}, we first state some assumptions. 
\begin{assumption}\label{As_U-stats_old} ~~
\begin{itemize}
\item [(1)] 
Let $X_1,\hdots, X_n\sim F$ be iid random vectors, where $F$ is a distribution on the d-dimensional cube $[-m,m]^d$, $m\in \R,d\in \N$.
\item [(2)]  Let $U=(U_1,\hdots, U_p)^\top$, where each component $U_j$ is a U-statistic with kernel $h_j$ of order $r(j)$. Then, we assume that for all $j=1,\hdots,p$
    \begin{equation*}
        \Var(\E[h_j(X_1,\hdots, X_r)|X_1])=\zeta_{1,j}>b,
    \end{equation*}
    where $b>0$ is fixed.
\item [(3)] For any kernel $h_j:([-m,m]^d)^{r(j)}\to \R$, $j=1,\hdots, p$, we assume
    \begin{equation*}
         \norm{h_j}_\infty\leq \frac{1}{2}L_\infty~,
    \end{equation*}
    where $L_\infty$ is a constant depending on $h_j, m$ for $j=1,\hdots, p$.
    \end{itemize}
\end{assumption}
Note that the assumptions above yield for any kernel $h$ that 
\begin{equation}\label{eq_bound}
    \norm{h(X)-h(X')}_2\leq L_2~,\quad \forall X,X'~,
\end{equation}
where $\norm{\cdot}_2$ denotes the euclidean norm and $L_2=(\sum_{j=1}^p L_\infty^2)^{1/2}$.
With that in hand, we can formulate a decision rule that is a consistent level $\alpha$ test:
\begin{theorem}[Monte Carlo]\label{Thm_Monte_Carlo_mult}
Assume Assumption \ref{As_U-stats_old} holds and let $\Delta>0$ be fixed. Furthermore consider a U-statistic $U$ with $\theta=\E_F[U]\in \R^p$. We define for a standard normal random variable $Y$
$$\norm{U}_{\infty}^{\text{DP}}:=\norm{U}_{\infty}+\frac{\Delta_2\norm{U}_{\infty}}{\sqrt{2 \rho}} Y~,$$ 
let $q_{1-\alpha}^*$ be the theoretical $(1-\alpha)$-quantile of $T_1^*$ approximated by Algorithm \ref{alg_monte_carlo_quantile_mult} with $\hat\zeta_1^{\text{DP}}=GaussCov(\hat\zeta_1,\rho,\Delta_2 \hat \zeta_1)$ (Algorithm \ref{alg_gausscov}). Then the decision rule "reject if 
\begin{equation}
\label{eq_U_dec_rule_normal_mult_mc}
   T^{\text{DP}}:=\sqrt{n}(\norm{U}_{\infty}^{\text{DP}}-\Delta)>q_{1-\alpha}^*"~,
\end{equation}
yields a consistent, asymptotic level-$\alpha$ test for hypotheses in equation \eqref{eq_rel_hyp_high}. That is for all $\norm{\theta}_\infty< \Delta $
$$
   \lim_{n \to \infty} \mathbb{P}_{\theta}(T^{\text{DP}}>q_{1-\alpha}^*) = 0,
$$
for all $\norm{\theta}_\infty= \Delta$
$$ \lim_{n \to \infty} \mathbb{P}_{\theta}(T^{\text{DP}}>q_{1-\alpha}^*) \leq \alpha
$$
(level $\alpha$). For all $\norm{\theta}_\infty>\Delta$ 
$$
    \lim_{n \to \infty} \mathbb{P}_{\theta}(T^{\text{DP}}>q_{1-\alpha}^*) = 1
$$
(consistency). Furthermore, the decision rule \eqref{eq_U_dec_rule_normal_mult_mc} is $2\rho$-zCDP.
\end{theorem}

\begin{proof}[Proof of Theorem \ref{Thm_Monte_Carlo_mult}]
 We first recall asymptotic results for the statistic $T_\theta^{\text{DP}} := \sqrt{n} (\norm{U}_\infty^{\text{DP}} - \norm{\theta}_\infty)$. First note that we have
 \begin{equation*}
     T_\theta^{\text{DP}}:=\sqrt{n}(\norm{U}_\infty^{\text{DP}}-\norm{\theta}_{\infty})+o_\PR(1)~.
 \end{equation*}
 Then using Hoeffdings CLT and the Delta method for Hadamard directional differential functionals \citep[see Theorem 2.1 in][]{carcamo2020directional}, we get for all $\norm{\theta}_\infty>0$ that
\begin{equation}\label{eq_asymptotic_mult_T_carlo}
    T_\theta^{\text{DP}} \overset{d}{\to} L=\max_{\substack{j=1,\hdots, s \\ |\theta_j|=\norm{\theta}_\infty}} sign(\theta_j)Z_j\leq \norm{Z}_\infty~,
\end{equation} 
where $Z\sim \mathcal N_p(0,\zeta_1)$. For $\norm{\theta}_\infty=0$, using the multivariate CLT of Hoeffding and the continuous mapping theorem yields
\begin{equation}\label{eq_asymptotic_mult_T_carlo_0}
    T_0^{\text{DP}} \overset{d}{\to} \norm{Z}_\infty~,
\end{equation}
where again $Z\sim \mathcal N_p(0,\zeta_1)$. Now let $q_{1-\alpha}^*$ be the theoretical quantile of $T_1^*$ approximated by Algorithm \ref{alg_monte_carlo_quantile_mult}. Since, we have used an upper bound $\norm{Z}_\infty$ in equation \eqref{eq_asymptotic_mult_T_carlo} for the resampling procedure in Algorithm \ref{alg_monte_carlo_quantile_mult}, we prove the consistency with respect to that upper bound. In particular, we know that 
\begin{equation*}
    d_K(\mathcal{L}(T_1^*|X_1,\hdots, X_n),\mathcal{L}(\norm{Z}_\infty))\overset{\PR}{\to} 0~,
\end{equation*}
where $d_K$ denotes the Kolmogorov distance and we used that $T_1^*\sim \max_{1\leq j\leq p} |Z'|+\sqrt{n} Y$, with $Z' \sim \mathcal N_p(0,\hat\zeta_1^{\text{DP}})$ and $Y\sim \mathcal N(0, (\frac{\Delta_2 \norm{U}_\infty}{\sqrt{2\rho}})^2)$ conditional on $X_1,\hdots, X_n$. 
This is indeed true due to the continuous mapping theorem and the conditional version of Slutsky's Lemma. Therefore, Lemma 23.3 in \cite{van2000asymptotic} yields that as $n\to\infty$ we have 
\begin{equation*}
    q_{1-\alpha}^*\overset{\PR}{\to} z_{1-\alpha}~
\end{equation*}
conditional on $X_1,\hdots, X_n$, where $z_{1-\alpha}$ is the $(1-\alpha)$-quantile of $\norm{Z}_\infty$. Consequently, there also exists a subsequence for which almost sure convergence holds, i.e. $q_{1-\alpha}^*\overset{a.s}{\to} z_{1-\alpha}$ on that subsequence. We can use the almost sure convergence of the bootstrap quantile $q_{1-\alpha}^*$ and the convergence of $T^{\text{DP}}_\theta$ combined with Slutsky's Lemma, to obtain
\begin{equation}\label{eq_boundedness_mult}
\lim_{n\to\infty}\PR(T^{\text{DP}}_\theta> q_{1-\alpha}^*)\leq\PR(\norm{Z}_\infty> z_{1-\alpha})=\alpha~.
\end{equation}
With that in hand, we can prove the statements of Theorem \ref{Thm_Monte_Carlo_mult} for $T^{\text{DP}}$ by considering four different cases:
\begin{itemize}
    \item [(1)] $0< \norm{\theta}_\infty < \Delta~,$
    \item [(2)] $\norm{\theta}_\infty=0~,$
    \item [(3)] $\norm{\theta}_\infty=\Delta~,$
    \item [(4)] $\norm{\theta}_\infty>\Delta~.$
\end{itemize}
For that purpose let us first decompose $T^{\text{DP}}$:
\begin{equation}\label{eq_expansion_mult}
     T^{\text{DP}}=T_\theta^{\text{DP}}+\sqrt{n}(\norm{\theta}_\infty-\Delta)~.
\end{equation}
\noindent (1) Consider the first case $0< \norm{\theta}_\infty < \Delta$. Since $q_{1-\alpha}^*$ is stochastically bounded, we can conclude that for every $M>0$ there exist an $n_0\in \N$ such that for all $n\geq n_0$ we have
\begin{equation*}
    \PR(T^{\text{DP}}> q_{1-\alpha}^*)=\PR(T_\theta^{\text{DP}}>q_{1-\alpha}^*-\sqrt{n}(\norm{\theta}_\infty-\Delta))\leq \PR(T_{\theta}^{\text{DP}}> M)~,
\end{equation*}
where we used that $\norm{\theta}_\infty-\Delta<0$. Now taking the limit for $n\to\infty$ on both sides yields
\begin{equation*}
    \limsup_{n\to\infty} \PR(T^{\text{DP}}> q_{1-\alpha}^*)\leq \PR(\norm{Z}_\infty> M)~,
\end{equation*}
where $Z\sim \mathcal{N}_p(0,\zeta_1)$. Hence, taking the limit for $M\to\infty$ lets us conclude that
\begin{equation*}
\lim_{n\to\infty}\PR(T^{\text{DP}}> q_{1-\alpha}^*)=0~.
\end{equation*}
\noindent (2) If $\norm{\theta}_\infty=0$, we have $T_\theta = T_0$ and can follow with the same arguments that for any $M>0$
\begin{equation*}
    \PR(T^{\text{DP}}> q_{1-\alpha}^*)=\PR(T_0^{\text{DP}}> q_{1-\alpha}^*+\sqrt{n}\Delta)\leq \PR(T_{0}^{\text{DP}}> M)
\end{equation*}
for $n$ sufficiently large, since $\Delta>0$. Consequently, observing \eqref{eq_asymptotic_mult_T_carlo_0}, the same arguments as in (1) give  
\begin{equation*}
    \lim_{n\to\infty}\PR(T^{\text{DP}}> q_{1-\alpha}^*)=0~.
\end{equation*}
Therefore, we obtain with (1) and (2) for all $\norm{\theta}_\infty<\Delta$ that
\begin{equation*}
    \lim_{n\to\infty}\PR(T^{\text{DP}}> q_{1-\alpha}^*)=0~.
\end{equation*}
\noindent (3) For $\norm{\theta}_\infty=\Delta$, we have that $T^{\text{DP}}=T_\theta^{\text{DP}}$, and we obtain from  \eqref{eq_boundedness_mult} to obtain
\begin{equation*}
   \limsup_{n\to\infty} \PR(T^{\text{DP}}> q_{1-\alpha}^*)= \lim_{n\to\infty} \PR(T_\theta^{\text{DP}}> q_{1-\alpha}^*)\leq \alpha~.
\end{equation*}
\noindent (4) For $\norm{\theta}_\infty>\Delta$, we  use the decomposition  \eqref{eq_expansion_mult} and equation \eqref{eq_boundedness_mult} to conclude that for any $M>0$
\begin{equation*}
     \PR(T^{\text{DP}}> q_{1-\alpha}^*)=\PR(T_\theta^{\text{DP}}> q_{1-\alpha}^*-\sqrt{n}(\norm{\theta}_\infty-\Delta))\geq \PR(T_{\theta}^{\text{DP}}>-M)
\end{equation*}
for $n$ sufficiently large. Taking the limit on both sides for $n\to\infty$, we have
\begin{equation*}
    \liminf_{n\to\infty}  \PR(T^{\text{DP}}> q_{1-\alpha}^*)\geq \PR(\norm{Z}_\infty>-M)~,
\end{equation*}
where again we have used the upper bound in equation \eqref{eq_asymptotic_mult_T_carlo}. Now taking the limit for $M\to\infty$ yields
\begin{equation*}
     \lim_{n\to\infty}\PR(T^{\text{DP}}\geq q_{1-\alpha}^*)=1~,
\end{equation*}
where again we have used that $\norm{Z}_\infty$ is tight. The privacy guarantee holds, due to the sensitivity of $U$ derived in Lemma \ref{Lem_sense_U}. For the private covariance, we can use Lemma \ref{Lem_jackknife_priv}. The result finally  follows from  the composition theorem for $\rho$-zCDP mechanisms.
\end{proof}

\section{Proofs of privacy statements}

  \def\theequation{E.\arabic{equation}}	
	\setcounter{equation}{0}

\subsection{Auxiliary results}

\begin{Lemma}\label{lem:local_sensitivity}
For any fixed $l\in \{1,\hdots, p-1\}$ for which $q_l(X)>4\frac{r}{n}L_\infty$, the local $\ell_2$ sensitivity of $\indvec$ is $0$.
\end{Lemma}
\begin{proof}
    Let $l\in\{1,\hdots, p-1\}$ be arbitrary but fixed. We have to prove that $\indvec$ and $\indvecp$ remain unchanged (i.e. the set of indices of the $l$ U-statistics with largest absolute value do not change) if we alter one entry of $X$. Changing any single entry of $X$ at worst yields a change of size $2\frac{r}{n}L_\infty$ for $|U|_{(l)}$ and $|U|_{({l+1})}$. Hence, if the gap between those two is lower bounded by $q_l(X)>4\frac{r}{n}L_\infty$, the index set does not change. Here it is important to note that the order of $|U|_{(1)},\hdots, |U|_{(l)}$ may indeed change, but the set of the corresponding indices does not. Consequently, we have 
      $$\indvec=\indvecp~,$$ 
    which yields the desired result 
    \begin{equation*}
        \max_{X', d_H(X,X')=1}\norm{\indvec-\indvecp}_2=0~.
    \end{equation*}
\end{proof}

\begin{Lemma}
\label{Lem_sense_U}
     Assume that Assumption \ref{As_Sparse} holds and let $Y\sim \mathcal N(0, \frac{(2L_\infty r/n)^2}{2 \rho} )$. Then, 
    \begin{equation*}
        \norm{U}_\infty^{\text{DP}}:= \norm{U}_\infty+Y
    \end{equation*}
    is $\rho-$zCDP.
\end{Lemma}

\begin{proof}
In order to bound the sensitivity $\Delta_2 \norm{U}_\infty$, we first consider two neighboring data sets $X,X'$ which differ without loss of generality in the last component. By definition, we have that
\begin{align*}
       & \left|\norm{U}_\infty-\norm{U'}_\infty\right|\leq\max_{1\leq j\leq p}|U_j-U_j'|\leq \frac{2r}{n}L_\infty~,
\end{align*}
where we have used that, by Assumption \ref{As_Sparse} (B), for $j=1, \ldots , p$:
\begin{align*}
       |U_j-U'_j|&=\binom{n}{r}^{-1} \left|\sum_{1\leq i_1<\hdots< i_r\leq n} h_j(X_{i_1},\hdots, X_{i_r})- \sum_{1\leq i_1<\hdots< i_r\leq n} h_j(X_{i_1}',\hdots, X_{i_r}')\right|
        \\&= \binom{n}{r}^{-1} \left|\sum_{1\leq i_1<\hdots<i_{r-1}< i_r= n} h_j(X_{i_1},\hdots,X_{i_{r-1}}, X_{n})-  h_j(X_{i_1}',\hdots, X_{n}')\right|
        \\&\leq \binom{n}{r}^{-1} 2L_\infty \binom{n-1}{r-1} =\frac{2r}{n}L_\infty ~.
\end{align*}
\end{proof}

\begin{proof}
   Note that Algorithm \ref{alg:rep_noisy_max} can be implemented using  the 
   exponential mechanism.  Therefore, it suffices to analyze the sensitivity of $q_j+\nu(j)$. For two neighboring databases $X,X'$, we have
    \begin{equation*}
        |q_j(X)+\nu(j)-(q_j(X')+\nu(j))|\leq |U_{(j)}-U_{(j)}'+U_{(j+1)}-U_{(j+1)}'|\leq 4\frac{r}{n}L_\infty~,
    \end{equation*}
    where we have used that the $\ell_1$ sensitivity of an $U$-statistic with a kernel of order $l$  bounded by $L_{\infty}$ is bounded by $2\frac{r}{n}L_\infty$. This yields $\ve-$DP. For $\rho$-zCDP, we use the relation  between $\ve-$DP and $\rho-$zCDP given in \cite{cesar2021bounding}.
\end{proof}

\begin{Lemma}[Lemma 3.4 in \citep{Dunsche2025}]\label{Lem_jackknife_priv}
 Assume that Assumption \ref{As_Sparse} holds. Then the Jackknife variance estimator defined in equation \eqref{eq_jackknife} has sensitivity 
    \begin{equation}\label{eq_sense_var_jack}
       \Delta_2 \hat\zeta_1= \frac{(n-1)r}{n(n-r)}\sum_{c=0}^r \frac{\binom{n-r+c}{r-c}}{\binom{n-1}{r}}\binom{r}{c}|cn-r^2|\sqrt{2}dL_\infty^2~.
    \end{equation}
\end{Lemma}

\subsection{Proof of Theorem \ref{thm_topk}}

For any $\alpha>1$ we denote for two random variables $Z,Y$ the Rényi divergence of the associated distributions by $D_\alpha(Z,Y)$. We further condition on a fixed realization of $\hat k$ and denote by $\mathcal M$ as the output of Algorithm \ref{alg:topk}.
 Here, we only provide the arguments for the divergence $D_\alpha(\mathcal M(X),\mathcal M(X'))$ and note that the divergence with the reverse order works analogously.
Then for two neighboring databases $X,X'$, we have two distinct cases to investigate:
\smallskip

    \textbf{Case (1)} We assume that the indicators are the same for neighboring databases $X$ and $X'$, i.e. 
    $$\indvec=\indvecp~.$$
    Recall that the possible outputs of $\mathcal M$ are $\bot$ or $\{\hat i_{(1)}, \hdots, \hat i_{({\hat k})}\}$ for both $X$ and $X'$.  Therefore, we can derive

        \begin{equation*}
            D_\alpha(\mathcal M(X),\mathcal M(X')) = D_\alpha(\mathbbm{1}\{\hat q_{\hat k}(X)>t\},\mathbbm{1}\{\hat q_{\hat k}(X')>t\})\leq D_\alpha(\hat q_{\hat k}(X),\hat q_{\hat k}(X'))~,
        \end{equation*}
        where the first equality follows from  the definition of Algorithm \ref{alg:topk}  and the second holds by the processing property of the Rényi divergence. Furthermore, by definition
        $$
        \hat q_{\hat k}(X)\sim \mathcal N(q_{\hat k}-\sigma z_{1-\delta}, \sigma^2)~,
        $$
        with $\sigma= t/\sqrt{\rho} $, $t=4r/n L_\infty$, and from the sensitivity bound for  for bounded $U$-statistics
        we obtain
        $$|q_{\hat k}(X)-q_{\hat k}(X')|\leq t~.
        $$
       Therefore, the Rényi divergence can be calculated explicitly and estimated as follows
        \begin{align*}
            D_\alpha(\hat q_{\hat k}(X),\hat q_{\hat k}(X'))&=\frac{\alpha(q_{\hat k}(X)-q_{\hat k}(X'))^2}{2\sigma^2}=\frac{\alpha(q_{\hat k}(X)-q_{\hat k}(X'))^2}{2(t/\sqrt{\rho})^2}
            \leq \alpha\rho/2~.
        \end{align*}
         Consequently, we obtain $\delta-$approximate-$\rho/2-$zCDP for any subset $E$ of our choice that has probability at least $1-\delta$. We will specify a specific set in the proof of the second case.  
         \smallskip 
         
        \textbf{Case (2)} Now assume that \begin{equation}\label{eq:gap_ha}\indvec\neq\indvecp~,
        \end{equation} then we have that $q_{\hat k}(X), q_{\hat k}(X')\leq t$,  because otherwise there would be equality in \eqref{eq:gap_ha}. Defining the event  $E= \{ \hat q_{\hat k}\leq q_{\hat k} \} $,  we have
        $$\PR(E)=\PR(\hat q_{\hat k}\leq q_{\hat k})\geq 1-\delta~,$$
        where we have used that $\hat q_{\hat k}\sim \mathcal N(q_{\hat k}-\sigma z_{1-\delta}, \sigma^2)$. 
        Therefore conditional on $E$, we have $\hat q_{\hat k}\leq q_{\hat k}\leq t$ for any neighboring $X$ and $X'$, which yields
        \begin{equation*}
            \PR(\mathcal M(X)=\bot|E)=\PR(\mathcal M(X')=\bot|E)=1~. 
        \end{equation*}
        Thus, conditional on $E$, we have that $D_\alpha(\mathcal M(X)|\mathcal M(X'))=0$ for all $\alpha$. Therefore, we have $\delta-$ approximated-$\rho/2$-zCDP for any $\rho>0$.\\ 
        
        Now, combing both Algorithm \ref{alg:rep_noisy_max} $\rho/2$ and the $\delta$-approximate-$\rho/2$-zCDP in the two cases of this proof, we obtain by the composition $\delta$-approximate-$\rho$-zCDP.

\subsection{Proof of Theorem \ref{thm:privacy_test}}

Let $X,X'$ be neighboring data sets. By Theorem \ref{thm_topk} and its proof, selecting the set of relevant coordinates \(\hat{\mathcal E}^{\mathrm{\text{DP}}}\)  by Algorithm \ref{alg:topk} with privacy budget $\rho/3$ is $\delta$-approximate-\((\rho/3)\)-zCDP, where the set $E$ in Definition \ref{defcdp} can be chosen as  $E:= \{\hat q_{\hat k}\leq q_{\hat k}\}$ 
(see the proof of Theorem \ref{thm_topk} for details). 

Now, let \(S\) be the event that for at least one index $i \in \{ 1 , \ldots , \hat k\} $ one $U_i $ and $U_i'$ have a different sign.
If $i$ is such an index it follows that $|U|_i\leq 2r/n L_\infty$ since the sensitivity of $U_i$ is at most \((2r/n)L_\infty\).  As Algorithm \ref{alg:topk} uses the threshold \(t=(4r/n)L_\infty\),  any sign change forces the bad event $E^c$ from Theorem \ref{thm_topk}. In fact, when we have
\begin{equation*}
    |U|_{(\hat k)}-|U|_{(\hat k+1)}\geq 4r/n L_\infty~,
\end{equation*}
we necessarily also have $|U|_{(i)}\geq|U|_{(\hat k)}>2r/nL_\infty$ because $|U|_{(\hat k+1)}\geq 0$. Consequently, we have \(S\subseteq E^c\) and
$$
\P(S)\le \P(E^c)\le \delta.
$$
On the event $E$ (i.e., when no sign change occurs), \textproc{Gausscov} is $(\rho/3)$-zCDP. Hence, with probability at least $1-\delta$, we have the following compositions:
\begin{itemize}
\item If the first branch of Algorithm \ref{alg:HD_test} is taken, then $\|U\|_\infty$ is released with privacy cost \(\rho/3\). Composing \(\hat {\mathcal E}^{\mathrm{\text{DP}}}\) $(\rho/3)$, \textproc{Gausscov} $(\rho/3)$, and \(\|U\|_\infty\) $(\rho/3)$ yields total budget $\delta$-approximate-\(\rho\)-zCDP.
\item Otherwise, the alternative branch releases \(\|U\|_\infty\) with cost \(2\rho/3\). Composing with \(\hat {\mathcal E}^{\mathrm{\text{DP}}}\) $(\rho/3)$ again gives total $\delta$-approximate-\(\rho\)-zCDP.
\end{itemize}
Thus the overall mechanism is \(\rho\)-zCDP on $E^{c}$ and fails only with probability at most $\delta$, which yields the desired result.

\section{Statistical guarantees}

  \def\theequation{F.\arabic{equation}}	
	\setcounter{equation}{0}
   
\subsection{Some results on bounded $U$-statistics}

\begin{Lemma}
    \label{VarConv}
   Consider a $U$-statistic $U$ of fixed order $r$, kernel $h$ and expected value $\theta = \mathbb{E} [U] \in \R^p$. If $\log(p)=o(n^{1/3})$ and $\|h\|_\infty \leq L_\infty$ we have
\begin{align*}
\max_{1 \leq i\leq j \leq p} |\hat {\zeta}_{1,{ij}}-\zeta_{1,ij}| \lesssim L_\infty^2\sqrt{\frac{\log(np)}{n}}
\end{align*}
 with probability at least $1-o(1)$. Here $\zeta_1$ and $\hat \zeta_1$ are defined in \eqref{eq_var_def} and \eqref{eq_jackknife}, respectively.
\end{Lemma}
\begin{proof}
    The case where the maximum is taken only over $1 \leq i=j \leq p$ can be found in \cite{Zhou:Han:Zhang:Liu:2019}. Our case can be handled by exactly the same arguments. 
\end{proof}

\begin{Lemma}
	\label{UStatConc}
	Consider a $U$-statistic $U$ of order $r$ with kernel $h$ and expected value $\theta= \mathbb{E} [U]  \in \R^p$. Assume that $\|h\|_\infty \leq L_\infty$ and that $\log p=o(n^{1/3})$. Then 
    \begin{itemize}
        \item[(1)] we have that \begin{align*}
		\norm{U - \theta }_\infty \leq L_\infty\sqrt{\frac{8r\log(p\lor n)}{n}}
		\end{align*}
		holds with probability at least $1-o(1)$.
        \item[(2)] Suppose that $\theta=\E[U]$ satisfies $\min_{1 \leq i \leq p}|\theta_i| >\underline{c}$  for some $\underline{c}>0$. Then 
        \[
            \PR(\text{sign}(U_i)=\text{sign}(\theta_i), i=1,...,p)=1-o(1)~.
        \]
    \end{itemize}
		
\end{Lemma}
\begin{proof}
	Both results follow by simple applications of \cite{Hoeffding1963}'s inequality and the union bound.
\end{proof}

\begin{Lemma}
    \label{lem:gaussian:approx}
    Consider a U-statistic $U$ of order $r$ with $\norm{h}_\infty\leq L_\infty$. Assume that $\log(p)=o(n^{1/5})$ and that $\min_{1 \leq i \leq p}|\theta_i|\geq \underline{c}>0$. Then there exists a zero-mean Gaussian vector $Z =(Z_1, \ldots , Z_p)^\top $ with
    \[
        \text{Cov}(Z_i,Z_j)=r^2\zeta_{1,ij}\text{sign}(\theta_i\theta_j)
    \]
    such that,  uniformly with respect to $t$, 
    \begin{align*}
        \PR \Big (\sqrt{n}(\max_{1 \leq i \leq p}|U_i|-\max_{1 \leq i \leq p}|\theta_i|)\geq t\Big  )\leq \PR \Big  (\max_{1 \leq i \leq p}Z_i\geq t \Big  )+o(1)
    \end{align*}
    with equality whenever $|\theta_j|=\max_{1 \leq i \leq p}|\theta_i|$ for all $j=1,\ldots ,p$.
\end{Lemma}

\begin{proof}
    This result is established in the course of the proof of Theorem 2.2 in \cite{patrick_annals}.
\end{proof}

\subsection{Analysis of Algorithm \ref{alg:rep_noisy_max}}
\begin{Lemma}\label{Lem:uti_rnm}
    Assume that condition \eqref{eq:gap} holds, and define  $q_j^{\theta}:=|\theta|_{(j)}-|\theta|_{(j+1)}$ and $k:=\argmax_{j=1,\hdots, p-1} q_j^{\theta}$. Then the output $\hat k$ of Algorithm \ref{alg:rep_noisy_max} fulfills 
    \begin{equation*}
        \PR(\hat k=k)=1-o(1)~.
    \end{equation*}    
\end{Lemma}

\begin{proof}
By Lemma \ref{UStatConc} there exist constants  $a_n\lesssim \sqrt{\log(n\lor p)/n}$ such that 
\[
    |U|_{(j)}\in \cup_{i=1}^pI_i, \quad j=1,...,p~
\]
with high probability, where the 
 (not necessarily disjoint) intervals $I_i$ are defined by   $I_i=(|\theta|_{(i)}-a_n,|\theta|_{(i)}+a_n)$.
 We may obtain $|U|_{(1)}$ by first picking the largest $|U_j|$ contained in the interval $I_1$. By construction the remaining $|U_j|$ are contained in $\cup_{i=2}^pI_i$ with at least one element contained in $I_2$. We may pick the largest (which is necessarily in $I_2$) to obtain $|U|_{(2)}$. Continuing like this we pick $|U|_{(i)}$ from $I_i$, yielding that with high probability 
\[
    \max_{1 \leq i \leq p}||U|_{(i)}-|\theta|_{(i)}|\leq a_n 
\]
In particular it follows by the triangle inequality that 
\[
    \max_{1 \leq i \leq p-1}|q_i-q_i^\theta|=O_\PR(\sqrt{\log(n\lor p)/n})~.
\]
By equation \eqref{eq:gap} we know that
\[
    q_k^\theta>\max_{l \neq k}q_l^\theta+\sqrt{\log(n)\log(p\lor n)/n}~.
\]
Consequently it holds with high probability that
\begin{align}
\label{eq:gap:empirical}
    q_k>\max_{l \neq k}q_l+1/2 \sqrt{\log(n)\log(p\lor n)/n}~.
\end{align}
Denote by $G_j$ the Gumbel noise added to $q_j$ in Algorithm \ref{alg:rep_noisy_max}. Because $\E[\exp(nG_j/2)]<\infty$ we have that
\[
    \max_{j=1,...,p-1} G_j=O_\PR(\log(p)/n)=o_\PR(\sqrt{\log(n)\log(p\lor n)/n})~.
\]
Consequently (up to changing the constant 1/2) \eqref{eq:gap:empirical} remains unaffected by the addition of Gumbel noise to the queries $q_l$ so that with high probability $q_k$ is the largest query, as desired.
\end{proof}

\subsection{ Proof of Theorem \ref{thm:consistency_Ustats_High}}

We distinguish  two cases corresponding to the null hypothesis and alternative. 
\smallskip 

\textbf{Case 1: $\|\theta\|_\infty\leq \Delta$.}
Due to the fact that
\begin{align}
\label{eq_private_to_nonprivate}
     \sqrt{n}(\|\tilde U\|_\infty^{\text{DP}}-\|\tilde U\|_\infty)=O_\PR(n^{-1/2})~,
\end{align}
we have by Lemma C.4 in  \cite{patrick_annals} for any $\delta>0$ that 
\begin{align*}
    \PR(\sqrt{n}(\|\tilde U\|_\infty^{\text{DP}}-\Delta)>t)&=\PR(\sqrt{n}(\|\tilde U\|_\infty-\Delta)>t+\sqrt{n}(\|\tilde U\|_\infty-\|\tilde U\|_\infty^{\text{DP}}))
\\&\leq\PR(\sqrt{n}(\|\tilde U\|_\infty-\Delta)>t)+O\Big(n^{-1/2+\delta}\sqrt{\log(p) } +\frac{\log(np)^{5/4}}{n^{1/4}}\Big)+o(1)\\
    &=\PR(\sqrt{n}(\|\tilde U\|_\infty-\Delta)>t)+o(1)
\end{align*}
We can use the same arguments to obtain the reverse inequality, which gives 
\begin{align*}
 \PR(\sqrt{n}(\|\tilde U\|_\infty^{\text{DP}}-\Delta)>t)=\PR(\sqrt{n}(\|\tilde U\|_\infty-\Delta)>t)+o(1)~.
\end{align*}

Note that by Lemma \ref{UStatConc},  we have for $\|\theta\|_\infty<\Delta-\gamma$ that
\begin{align}
\label{eq_diverging}
   \sqrt{n}(\|\tilde U\|_\infty-\Delta)=O_\PR\Big(\sqrt{\log(p \lor n)/n}\Big)-\sqrt{n}(\Delta-\|\theta\|_\infty) \overset{\PR}{\to} -\infty~.
\end{align}
From this the first statement of the Theorem follows: the test statistic diverges to $-\infty$ and the quantiles $q_{1-\alpha}$ in Algorithm \ref{alg_monte_carlo_quantile_highdim} and Algorithm \ref{alg:Gumbel} are stochastically bounded below by $0$ for $\alpha<0.5$.\\

Let us now consider statement (ii), that is the case $\|\theta \|_\infty \leq  \Delta $. By Assumption (P) we may condition on $\hat i_1=j_1,...,\hat i_{\hat k}=j_l,\hat k=l$ so that the output of Algorithm \ref{alg:topk} is deterministic and given by (a possibly randomly selected subset of) $j_1,..., j_l$ or $\bot$. For the remainder of the proof we will hence simply write $l$ instead of $\hat k$. In the former case we additionally condition on the random subselection, so that we may assume WLOG (the selection is independent of everything else) that $l \leq \log(p)$ and that $\hat i_1,...,\hat i_{\hat k}$ are fixed. To be precise Assumption (P)  yields that with high probability
\begin{align*}
        &\PR( U_{j_m}  \geq  t_{m} ~\text{ for all }~ m=1, \ldots , l  )\\
        =&\PR( \tilde U_{m} \geq  t_{m} ~\text{ for all }~ m=1, \ldots , \hat k |\hat i_1=j_1,...,\hat i_{\hat k}=j_l, \hat k=l)
    \end{align*}
so that in all following considerations we may simply assume that $\tilde U=(U_{j_1},...,U_{j_l})^\top$ or $\tilde U=U$. We will argue each case separately. \\

\textbf{Subcase 1(a): Algorithm \ref{alg:topk} returns $\{\hat i_1,...,\hat i_{\hat k}\}=\{j_1,...,j_l\}$:}~ 
From now on assume that $\|\theta\|_\infty>0$ (otherwise the previous arguments apply). We define for some $\gamma>0$ the set
\begin{align*}
    I_1:=\{i \in \{ j_1,..., j_l\} | \ |\theta_{i}|>\Delta-\gamma\}
\end{align*}
and the associated statistic
\[
    \hat S=\sqrt{n}\max_{i \in I_1}(|U_i|-\Delta)~.
\]
whose coordinate-variances are lower bounded by Assumption \ref{As_Sparse}(V). Lemma \ref{lem:gaussian:approx} (with $p = |I_1|$) then yields
\begin{align}
\label{eq_gaussian_upper_bound}
    \PR(\hat S>t)&\leq\PR(\max_{i \in I_1}Z_i>t)+o(1)\leq\PR \Big (\max_{i \in \{ j_1,...,j_l\}}Z_i>t \Big )+o(1)~.
\end{align}
where the precise definition of the Gaussian random variables $(Z_i)_{i\in I_1}$ is given in Lemma \ref{lem:gaussian:approx}. 
Finally we note that 
\begin{equation}\label{eq:decomp_diverge}
    \sqrt{n}\Big (\max_{i \in \{ j_1,..., j_l\}\setminus I_1}|U_i|-\Delta  \Big )\overset{\PR}{\to} -\infty~.
\end{equation}
which implies \begin{align*}
    \sqrt{n} (\|\tilde U\|_\infty-\Delta)=\max \Big \{\hat S,\sqrt{n}\Big (\max_{i \in \{ j_1,...,j_l\}\setminus I_1}|U_i|-\Delta \Big )\Big \} = \hat S +  o_\PR (1) ~.
\end{align*}

Combining this estimate with \eqref{eq_private_to_nonprivate} and \eqref{eq_gaussian_upper_bound} gives, for any $t \in \R$,  
\begin{align*}
    \PR(\sqrt{n}(\|\tilde U\|_\infty^{\text{DP}}-\Delta)>t)&=\PR(\sqrt{n}(\|\tilde U\|_\infty-\Delta)>t)+o(1)\\
    &=\PR(\hat S>t)+o(1)\\
    &\leq \PR \Big (\max_{i \in \{ j_1,..., j_l\}}Z_i>t)+o(1 \Big )~.
\end{align*}
With this at our disposal we now only need to establish that, on a set with high probability, the inequality 
\begin{align*}
   \Big |\PR\Big(\max_{i \in \{ j_1,..., j_l\}}Z_i>t\Big)-\PR^*\Big(\max_{i \in \{ j_1,..., j_l\}}Z_i^{\text{DP}}>t\Big)\Big | \leq c_n
\end{align*}
holds for some sequence  $c_n \to 0$. Here $Z_i^{\text{DP}} \overset{iid}{\sim}\mathcal N_{k}(0,(\hat \zeta_1(k))^{\text{DP}})$ and $\P^*$ denotes the probability space obtained by conditioning on the data and the privacy noise of the covariance privatization. The desired statement is now a consequence of Theorem 2 from \cite{chernozhukov:chetverikov:kato:2015}. For the application of this result it suffices to establish that
\[
 ( \log(l \land \log(p))^2 \hat \Delta =o_\PR(1)
\]
where
\[
    \hat \Delta=\max_{1 \leq h<j\leq l \land \log(p)}|\hat \zeta_{1,hj}^{\text{DP}}-\zeta_{1,hj}|~.
\]
However, this statement follows from Lemma \ref{VarConv} and the bound on the privatization error incurred by Algorithm \ref{alg_gausscov}  \citep[similar to the proof of Lemma 3.5 in][]{Dunsche2025}.
\smallskip

\textbf{Subcase 1(b): Algorithm \ref{alg:topk} returns $\bot$: ~}
In this case we simply have $\tilde U=U$. By the same argument that yields \eqref{eq_diverging} and by Lemma \ref{UStatConc} we obtain
\begin{align*}
    \sqrt{n}\max_{1 \leq i \leq p}(|U_i|-\Delta) &\leq\sqrt{n}\max_{i \in \tilde I_1 }(|U_i|-\Delta)+o_\PR(n^{-1/2})\\
    &= \sqrt{n}\max_{i \in \tilde I_1 }\text{sign}(\theta_i)(U_i-\theta_i)+o_\PR(n^{-1/2})=:T_n+o_\PR(n^{-1/2}) ,
\end{align*}
where the last line defines in an obvious manner and the set $\tilde I_1$ is defined by  
$$
\tilde I_1 := \big \{ i ~|~ 1 \leq i \leq p ; ~| \theta_i|\geq \Delta-\gamma \big \} 
$$
Next we may apply Lemma \ref{lem:gaussian:approx} to obtain the existence of  a Gaussian vector 
$
    Z=( Z_i  \ )_{i \in \tilde I_1}  $
with covariances given by
\[
   r^2\zeta_{1,ij}\text{sign}(\theta_i\theta_j), \quad  i,j  \in \tilde I_1
\]
such that 
\begin{align*}
    \sup_{t \in \R} \Big |\PR \Big (T_n\geq t \Big )-\PR\Big (\max_{i \in \tilde I_1 }Z_i\geq t\Big ) \Big |=o(1)~.
\end{align*}
By the Hoeffding decomposition and assumption (B) we have 
\[
    \sup_{1 \leq i<j\leq p}|\text{Cov}(U_i,U_j)-r^2\zeta_{1,ij}|=O(n^{-1}),
\]
and,by Lemma 2 from \cite{chernozhukov:chetverikov:kato:2015}, there exists another Gaussian vector $Z^1$ of the same dimension with covariance structure given by the covariance of the vector
\[
    \big (\text{sign}(U_i)(U_i-\theta_i) \big)_{i \in \tilde I_1} ~,
\]
such that 
\begin{align*}
    \sup_{t \in \R} \Big |\PR \Big (\max_{i \in \tilde I_1 }Z^1_i\geq t \Big )-\PR \Big (\max_{i \in \tilde I_1}Z_i\geq t \Big ) \Big |=o(1)~.
\end{align*}
Note that $|\theta_i|\geq \Delta-\gamma$ implies $\text{Var}(Z^1_i)\leq L_\infty^2-(\Delta-\gamma)^2$ by the \cite{Bhatia:Davis:2000} inequality. Pad the vector $Z^1_i$ with some additional independent normal random variables with variances given by $L_\infty^2-(\Delta-\gamma)^2$ and denote the resulting random vector by $\tilde Z$. Clearly,
\begin{align*}
    \max_{i \in \tilde I_1 }Z_i^1\leq \max_{1 \leq i \leq p}\tilde Z_i~.
\end{align*}
By Slepian's Lemma and assumption (E) we then have for $t \geq 0$ that
\begin{align*}
    \PR \Big (\sqrt{n}\max_{1 \leq i \leq p}\tilde Z_i\geq t\Big  )\leq \PR\Big (\sqrt{n}\max_{1\leq i \leq p} Y_i\geq t\Big  )~,
\end{align*}
where $Y_i$ is a collection of iid $\mathcal N(0,L_\infty^2-(\Delta-\gamma)^2)$ random variables. The right hand side maximum converges, appropriately rescaled, in distribution to a Gumbel distribution with scale parameter $\sqrt{L_\infty^2-(\Delta-\gamma)^2}$. In particular we obtain that
\begin{align*}
    \limsup_n\PR(\sqrt{n}\max_{1 \leq i \leq p}(|U_i|-\Delta)>t)\leq \limsup_n\PR(\sqrt{n}\max_{1\leq i\leq p}Y_i\geq t)
\end{align*}
for any $t\geq 0$. Letting $t=q_{1-\alpha}^G/a_p+a_p-\frac{\log\log(p)+\log(4\pi)}{2a_p}$ then yields 
\[
    \limsup_n\PR(\sqrt{n}\max_{1\leq i\leq p}Y_i\geq t)=\alpha
\]
as desired. Here $q_{1-\alpha}^G$ is the $(1-\alpha)$-quantile of a Gumbel distribution with scale parameter $\sqrt{L_\infty^2-(\Delta-\gamma)^2}$. \\

\textbf{Case 2: $\|\theta\|_\infty>\Delta$} 
If  Algorithm \ref{alg:topk} outputs $\bot$  i.e. when we use the Gumbel test, we may use exactly the same arguments as given in  the proof of Theorem 2.4 in \cite{patrick_annals} for the non-private setting. By equation \eqref{eq_private_to_nonprivate} we may reduce to this case and are done.\smallskip 

In the other case, again using \eqref{eq_private_to_nonprivate} to reduce to the non-private setting, let  $i_0$ be an  index for which $\|\theta\|_\infty=|\theta_{i_0}|$ (we suppress the possible dependence on $n$ in our notation). As  \eqref{eq:gap} is satisfied, we have by Lemma \ref{Lem:uti_rnm} that $\hat k=k\leq \log(p)$ with high probability. This implies that $i_0=\hat i_1$, where we recall $\hat i_1$ is the index of the largest U-statistic in absolute value. From that we can obtain
\begin{align}
\label{eq_power_lower_bound}
\sqrt{n}(\|\tilde U\|_\infty-\Delta) \geq \sqrt{n}(|U_{\hat i_1}|-|\theta_{\hat i_1}|)+\sqrt{n}(|\theta_{\hat i_1}|-\Delta) \,.
\end{align}
By Lemma \ref{UStatConc} we further obtain that
\begin{align*}
\sqrt{n}\left||U_{\hat i_1}|-|\theta_{\hat i_1}|\right| \lesssim \sqrt{n}\| U-\theta\|_\infty \leq rL_\infty\sqrt{2\log(p\lor n)}
\end{align*} 
with probability $1-o(1)$.
The second term on the right hand side of \eqref{eq_power_lower_bound} converges to $+\infty$ at rate $c\sqrt{\log(p\lor n)}$, yielding the desired conclusion upon choosing $c$ sufficiently large.

\putbib
\end{bibunit}

\end{document}